\documentclass[10pt]{article}
\usepackage{color}
\usepackage{amsmath}

\usepackage{footnote}
\usepackage{subscript}
\usepackage{amsfonts,color}
\usepackage{amsmath,amssymb}
\usepackage{amssymb} 
\usepackage{mathtools}
\usepackage{amsthm,wasysym}
\usepackage{framed}
\usepackage{cancel}
\usepackage{graphicx}
\usepackage{caption}
\usepackage{subcaption}
\usepackage{rotating}

\usepackage[english]{babel}
\usepackage[utf8]{inputenc}
\makeatletter
\usepackage{float}
\usepackage{wrapfig}
\restylefloat{figure}

\makeatletter
\let\@fnsymbol\@arabic
\makeatother

\usepackage{bbm}

\newcommand{\id}{{\mathbf{\mathbbm{1}}}}

\newcommand{\tr}{{\rm tr}}
\newcommand{\dev}{{\rm dev}}
\newcommand{\Sym}{{\rm Sym}}
\newcommand{\sym}{{\rm sym}}
\newcommand{\skw}{{\rm skew}}

\newcommand{\Curl}{{\rm Curl}}
\newcommand{\Div}{{\rm Div}}
\newcommand{\axl}{{\rm axl}}
\newcommand{\anti}{{\rm \textbf{Anti}}}
\newcommand{\so}{\mathfrak{so}}

\newcommand{\norm}[1]{\|#1\|}

\def\dd{\displaystyle}

\newtheorem{theorem}{Theorem}[section]
\newtheorem{lemma}[theorem]{Lemma}
\newtheorem{remark}[theorem]{Remark}
\newtheorem{proposition}[theorem]{Proposition}

\newtheorem{definition}[theorem]{Definition}
\def\barr{\begin{array}}
\usepackage{lscape}

	\def\earr{\end{array}}
\def\bec#1{\begin{equation}\label{#1}}
\def\becn{\begin{equation*}}
\def\endec{\end{equation}}
\def\endecn{\end{equation*}}
\def\dd{\displaystyle}
\def\bfm#1{\mbox{\boldmat}}

\begin{document}
	
	\title{Existence and uniqueness of  Rayleigh waves in isotropic  elastic Cosserat materials and algorithmic aspects}
\author{ Hassam Khan\thanks{Hassam Khan,  \ \  Lehrstuhl f\"{u}r Nichtlineare Analysis und Modellierung, Fakult\"{a}t f\"{u}r
		Mathematik, Universit\"{a}t Duisburg-Essen,  Thea-Leymann Str. 9, 45127 Essen, Germany, email: hassam.khan@stud.uni-due.de} \quad and \quad  Ionel-Dumitrel Ghiba\thanks{ Ionel-Dumitrel Ghiba (Corresponding author),  \ Department of Mathematics,  Alexandru Ioan Cuza University of Ia\c si,  Blvd.
		Carol I, no. 11, 700506 Ia\c si,
		Romania; and  Octav Mayer Institute of Mathematics of the
		Romanian Academy, Ia\c si Branch,  700505 Ia\c si, email:  dumitrel.ghiba@uaic.ro}  \quad and  \\   Angela Madeo\,\thanks{Angela Madeo,  \ Technische Universit\"at Dortmund, August-Schmidt-Str. 8, 44227 Dortmund, Germany, email: angela.madeo@tu-dortmund.de}
	 \quad and \quad    Patrizio Neff\,\thanks{Patrizio Neff,  \ \ Head of Lehrstuhl f\"{u}r Nichtlineare Analysis und Modellierung, Fakult\"{a}t f\"{u}r
		Mathematik, Universit\"{a}t Duisburg-Essen,  Thea-Leymann Str. 9, 45127 Essen, Germany, email: patrizio.neff@uni-due.de}
}

\maketitle
\begin{abstract}
	
	We discuss the propagation of  surface waves in an isotropic half space modelled with the linear  Cosserat theory of isotropic elastic materials. To this aim we use a method based on the algebraic analysis of the surface impedance matrix  and on the algebraic Riccati equation, and which  is independent of the common Stroh formalism. Due to this method,  a new algorithm which determines the amplitudes and the wave speed in the theory of isotropic elastic Cosserat materials is described. Moreover, the method allows {us} to prove  the existence and uniqueness of a subsonic solution of the secular equation, a problem which remains unsolved in almost all generalised linear theories of elastic materials. Since the results are suitable to be used for numerical implementations,     we propose  two numerical algorithms which are viable for any  elastic material. Explicit numerical calculations are made for   alumunium-epoxy  in the context of the Cosserat model. Since the  novel form of the secular equation for isotropic elastic material has not been explicitly derived elsewhere,   we establish it in this paper, too.

  \medskip
  
  
  \medskip
  
  \noindent\textbf{AMS 2020 MSC:} 74J15, 74M25, 74H05, 74J05
  
\end{abstract}

\begin{footnotesize}
	\tableofcontents
\end{footnotesize}

\section{Introduction}

{\it  Traveling waves} can exist within a small depth from a free surface of an elastic continuum, while the bulk of the continuum remains almost at rest. Such waves are called {\it Rayleigh waves} named after the English physicist Lord Rayleigh, who carried out pioneering work in studies of wave propagations in isotropic elastic media \cite{Rayleigh}. Rayleigh waves are important for modelling  {\it seismic waves} that are created by  earthquakes on the surface of the earth.  These waves are a combination of  longitudinal (horizontal) and transverse (vertical) motions. The horizontal components of the displacement are parallel to the direction of propagation, whereas vertical components are directed into the half-space. Rayleigh waves move elliptically, they take place counter-clockwise near the surface and  clockwise deep down. The amplitude of the waves decays exponentially with depth beneath the surface. The studies of Rayleigh waves captivated the attention of many scientists owing to its industrial application such as material characterization, nondestructive evaluation and acoustic microscopy. These waves are also employed to detect cracks and other defects in the material. In addition, many applications have been found in  seismology  and near surface geophysical exploration.

Rayleigh waves  are particular  {\it inhomogeneous plane waves}.  Inhomogeneous plane waves, also known as 
\textit{evanescent waves}, represent those waves for which
the planes of constant phase are not the same as the planes of
constant amplitude. Usually, these types of waves are described in terms of
 the slowness bivector (a complex vector) 
and the amplitude bivector. Hayes \cite{hayes1986inhomogeneous,boulanger1993bivectors} has developed the
directional-ellipse method for a systematic study of all
inhomogeneous plane waves that may propagate in  classical
linear elasticity.  In classical elasticity of isotropic materials, the solution of the
corresponding equation for Rayleigh surface waves speed (the {\it secular equation})
has been studied numerically, see e.g.  \cite{Rayleigh,hayes1962note}. For the statement of
the problem one may consult  the book \cite{Achenbach} and the  works \cite{Nkemzi,Nkemzi2,Rahman,Malischewsky,Li,destrade2007seismic,Ting2,Ting,VinhMalischewsky1,VinhMalischewsky2,VinhOgden}.

An important non-trivial task in classical linear elasticity is the extension of the results concerning the Rayleigh waves to  anisotropic materials. Using a formalism ({\it Stroh formalism}) constructed  in \cite{stroh1958dislocations}, Stroh \cite{stroh1962steady} was able to avoid the complex secular equation derived by Synge \cite{synge1956elastic} and he has given a real expression of it, see also Currie \cite{currie1974rayleigh}. Another issue concerning the surface wave in an anisotropic elastic half-space is the derivation of the {\it explicit secular equations} such that a solution (analytical, if possible) can be easily  found (i.e.~at least with the help of a clear numerical strategy). In this respect, we mention  \cite{taziev1989dispersion,mozhaev1995some, destrade2001explicit,destrade2007seismic,destrade2002incompressible,ting2002explicit,ting2002explicit}. 
For isotropic  as well as for anisotropic materials,   the mathematical analysis of this equation has, perhaps, at least the same importance as  the derivation of the secular equations, since it is not obvious if there exists an admissible solution of the secular equation and if this solution is unique. By an admissible solution we mean a solution which takes into account all the restrictions which were imposed in the derivation of the secular equation. In many approaches this essential aspect is neglected or only conjectured.  For linear isotropic materials this is explained in the book by Achenbach \cite{Achenbach}. While  {\it strong ellipticity} (Legendre–Hadamard ellipticity) guarantees  the existence of plane waves, it is not obvious if  seismic waves exist since three inhomogeneous body wave solutions are needed when the traction-free boundary condition specific to seismic waves is solved. For seismic waves propagating in anisotropic elastic materials, the first uniqueness result is due to Barnett et al. \cite{barnett1973elastic}, see also \cite{chadwick1977foundations} for more  details,  while for the proof of the existence of the solution of the secular equation, we mention the works by Barnett and Lothe \cite{barnett1974consideration}, Lothe and Barnett \cite{lothe1976existence} and Ingebrigtsen and Tonning \cite{ingebrigtsen1969elastic}. 
  
Motivated by  previous work of Mielke and Sprenger \cite{mielke1998quasiconvexity} which is more related to  control theory \cite{knobloch2012topics} than to surface wave propagation, Fu and Mielke \cite{fu2002new} and Mielke and Fu \cite{mielke2004uniqueness} have devised a new method for anisotropic elastic materials which is not based on the Stroh formalism, it is conceptually different from the other methods   and it is mathematically  well explained. They have shown that the {\it impedance matrix} \cite{ingebrigtsen1969elastic} defining the secular equation is the solution of an \textit{algebraic Riccati equation}. Using the properties of this equation, in \cite{mielke2004uniqueness,fu2002new} it is then proven that   the secular equation does not admit spurious roots. It can be  said that the works by Mielke and Fu give an elegant final  answer to the general case of anisotropic elastic materials, in the framework of classical linear elasticity. Since in  \cite{mielke2004uniqueness,fu2002new} the secular equation is not written explicitly, in the current paper we have computed it for linear isotropic classical elasticity, too, and we compare it with the  classical one and that resulting by using the Stroh formalism.

However, the classical theory of elasticity does not explain certain discrepancies that
occur in propagation of waves at high frequency
and short wavelength. It is now a common place that the response of the material to external stimuli depends
heavily on the motions of its inner structure and if the ratio of
the characteristic length associated with the external stimuli and the
internal characteristic length is near  to $1$, then the response
of constituent subcontinua becomes important, see \cite{Eringen99}. Classical linear elasticity ignores
these effects. 

One of the first generalization of classical linear elasticity is represented by the {\it Cosserat (micropolar) continuum} theory in which the rotational degrees
of freedom play a central role \cite{Eringen64,Mindlin64}.   
For a quick comparison, we mention that denoting the macroscopic displacement by 
$\mathbf{u}:\Omega \subset\mathbb{R}^3\mapsto\mathbb{R}^3$ and the   microrotation vector field by $ \vartheta   :\Omega \subset\mathbb{R}^3\mapsto\mathbb{R}^3$,  the elastic energy density  of the isotropic linear Cosserat model read
\begin{align}
W(\mathrm{D}u ,{ \vartheta  } , \mathrm{D} \vartheta    )=& \,\mu_{\rm e} \,\lVert \dev_3\, \sym\,\mathrm{D}u \rVert ^{2}+\mu_{\rm c}\,\lVert \skw\,(\mathrm{D}u -\anti\, \vartheta   )\rVert ^{2}+\frac{2\,\mu_{\rm e} +3\,\lambda_{\rm e} }{6}\left[\mathrm{tr} \left(\mathrm{D}u \right)\right]^{2}\notag\\&+\frac{\mu_{\rm e}\,L_{\rm c}^2}{2}\left[ {\alpha_1}\,\lVert \dev_3\,\sym\,\mathrm{D} \vartheta    \rVert ^{2}+{\alpha_2}\,\lVert \skw\,\mathrm{D} \vartheta    \rVert ^{2}+\frac{2\,\alpha_1+3\,\alpha_3}{6}\left[\mathrm{tr} \left(\mathrm{D} \vartheta    \right)\right]^{2}\right],
\end{align}
where  $(\lambda_{\rm e}, \mu_{\rm e})$, $\mu_{\rm c} $, $L_{\rm c}$, $\alpha_1$, $\alpha_2$ and $ \alpha_3$  are six isotropic  elastic moduli representing the parameters related to the meso-scale,  the Cosserat couple modulus, the characteristic length, and the three
general isotropic curvature parameters, respectively, { $(\anti(\vartheta))_{ij}=-\epsilon_{ijk}\,\vartheta_k,
$
with $\epsilon_{ijk}$ the totally antisymmetric third order permutation tensor, see Section \ref{NS}}, while the elastic energy density of  classical linear elasticity is 
\begin{align}
W(\mathrm{D}u )=&\,  \mu_{\rm e} \,\lVert \dev_3\, \sym\,\mathrm{D}u \rVert ^{2}+\frac{2\,\mu_{\rm e} +3\,\lambda_{\rm e} }{6}\left[\mathrm{tr} \left(\mathrm{D}u \right)\right]^{2}.\qquad \qquad \qquad \qquad \qquad 
\end{align}

In the framework of the Cosserat theory, the  propagation of seismic waves in an isotropic Cosserat elastic half space  was studied in \cite{ChiritaGhiba3} using  the Stroh formalism \cite{destrade2007seismic,stroh1962steady}, see also \cite{lazar2005cosserat}.  The strong point of the approach  in \cite{ChiritaGhiba3}  is that  explicit
expressions of involved eigenvalue problems and explicit conditions upon the wave
speed were found, as well as the exact expressions of three linear independent
amplitude vectors. Then,  a simple form of the secular
equation is obtained, which,  by comparison  with other
generalized forms of the secular equation for Cosserat materials {\cite{Eringen99,Erofeyev,grekova2009waves,Koebke,kulesh2009problem,Kulesh01,Kulesh03,Kulesh05}}  does not involve the complex form of
the attenuating coefficients. For a specified
class of materials, i.e., for which the constitutive coefficients\footnote{We use different notations in comparison to  the Eringen notation in \cite{ChiritaGhiba3}, i.e. $\mu_{\rm c}=\frac{\kappa_{\rm Eringen}}{2}, \ \ \mu_{\rm e} =\mu_{\rm Eringen}+\frac{\kappa_{\rm Eringen}}{2}$, see also \cite{hassanpour2017micropolar}. In the Eringen notations
	 these are $\lambda_{\rm Eringen}+\mu_{\rm Eringen}>0$,  $\kappa_{\rm Eringen}>0$, $\alpha_1+\alpha_2>0$.} satisfy
\begin{align}\label{Chiritacond}
\mu_{\rm e}-\mu_{\rm c}+\lambda_{\rm e} >0, \qquad \quad  \mu_{\rm e} +\mu_{\rm c} >0, \qquad  \qquad    \alpha_1+\alpha_2>0,
\end{align}
Chiri\c ta and Ghiba \cite[Eq. (4.14)]{ChiritaGhiba3} have proved  the existence  of {the solution of their secular equation}. Using some
 illustrative graphics they conjectured that the solution should be also unique for this subclass of materials, but there does not exist a proof of the uniqueness in the Cosserat theory or an existence and uniqueness proof for a larger class of isotropic Cosserat materials.

It is also important to notice that in the Cosserat theory, too, it is not obvious how to avoid the  spurious roots of the secular equation, as long as the Stroh formalism is used.  Even if the relation between Mielke and Fu’s method  \cite{mielke2004uniqueness,fu2002new} and the Cosserat model is not evident, in our paper we  show  that it may be adapted to the study of propagation of seismic waves in Cosserat  materials. So,   our derivation will involve a matrix algebraic  Riccati equation which will provide a formula for the desired solution. 
We will show that the new form of the secular equations, written in terms of the impedance matrix, does not admit any spurious root, i.e. there exists {\it only one  subsonic surface wave}. 

In the current paper {\it we prove both the existence and the uniqueness} of the wave speed of the Rayleigh wave in Cosserat materials, for the first time, and this {\it under  weaker conditions}\footnote{In the sense that  the conditions $\mu_{\rm e}-\mu_{\rm c}+\lambda_{\rm e} >0,   \mu_{\rm e} +\mu_{\rm c} >0$ imply $2\, \mu_{\rm e}+\lambda_{\rm e} >0$
 but not vice versa, even when the extra conditions $\mu_{\rm c}>0$ and $\mu_{\rm e}>0$ are also assumed.} on the constitutive coefficients, i.e., we require only 
\begin{align}\label{cr}
2\,\mu_{\rm e} +\lambda_{\rm e} >0,\qquad \quad \mu_{\rm e} >0,\qquad \quad \mu_{\rm c} >0,\qquad \quad   \alpha_1+\alpha_2>0.
\end{align}
For a given direction ${\xi}\in \mathbb{R}^3$, $\xi\neq 0$  of the form $\xi=(\xi_1,\xi_2, 0)^T$, the  constitutive requirements \eqref{cr} are equivalent to   the existence of \textit{only real waves} propagating  in the plane normal to the direction of propagation and parallel to the direction of propagation. These conditions do not involve the constitutive parameter $\alpha_3$ as it would be the case for the requirement that only  real waves may propagate (in arbitrary planes) in any directions and for arbitrary wave numbers. Indeed, the propagation of  {\it real plane waves} is equivalent to the constitutive inequalities
\begin{align}\label{crw}
2\,\mu_{\rm e} +\lambda_{\rm e} >0,\qquad \quad \mu_{\rm e} >0,\qquad \quad \mu_{\rm c} >0,\qquad \quad   \alpha_1+\alpha_2>0,\qquad \quad 2\,\alpha_1+\alpha_3>0.
\end{align}
Let us notice that while in classical linear elasticity the existence of  seismic waves is guaranteed   for strongly elliptic materials, i.e.,
\begin{align}\label{LHce}
2\,\mu_{\rm e}+\lambda_{\rm e} >0, \quad \, \qquad \qquad \mu_{\rm e}>0,
\end{align}
in the Cosserat model 
 the \textit{strong ellipticity conditions}  (\textit{Legendre–Hadamard ellipticity})  \cite{Neff_JeongMMS08,shirani2020legendre,eremeyev2007constitutive} are equivalent to
\begin{align}\label{LHc}
2\,\mu_{\rm e}+\lambda_{\rm e} >0, \quad \, \qquad \qquad \mu_{\rm e}+\mu_{\rm c} >0,\qquad  \qquad \quad   \, \alpha_1+\alpha_2>0, \qquad \quad 2\,\alpha_1+\alpha_3>0.
\end{align}
 and the latter conditions are not sufficient to impose
 the existence of  seismic waves. The existence of  seismic waves  is  rather related  to the propagation of  {real plane waves}  than to  strong ellipticity, while the strong ellipticity conditions are useful in the study of accelerated waves \cite{eremeyev2007constitutive}. We mention that the existence for only real plane waves implies the strong ellipticity conditions but not vice versa. For the Cosserat theory,  strong ellipticity is  related only to the propagation of acceleration waves \cite{Eremeyev4,eremeyev2007constitutive}. 
 {	\begin{table}[h!]\begin{center}
 		\begin{tabular}{ |c| c | c |}
 			\hline 
 			Name&	Expression& \begin{minipage}{3.5cm} 
 				Dispersive waves/\\Non-dispersive waves	\end{minipage}\\
 			\hline 
 			\begin{minipage}{7cm}\medskip
 				the velocity of the acoustic branch of translational
 				compression (longitudinal) plane wave \medskip
 			\end{minipage}  & $\mathfrak{c}_p= \sqrt{\frac{\lambda_e+2\mu_e
 				}{\rho}}$ &  non-dispersive \\\hline 
 			\begin{minipage}{7cm}\medskip
 				the limit of the group/phase velocity of the acoustic
 				branch of the shear–rotational wave at $\omega\to 0$ ($k\to 0$)  \medskip
 			\end{minipage} & $\mathfrak{c}_{t} = \sqrt{\frac{\mu_e}{\rho}}$&  dispersive 
 			\\\hline 
 			\begin{minipage}{7cm}\medskip
 				the limit of the group/phase velocity of the acoustic
 				branch of the shear–rotational wave at $\omega\to \infty$ ($k\to \infty$)\medskip
 			\end{minipage} & $\mathfrak{c}_s=\sqrt{\frac{\mu_e+\mu_c}{\rho}}$ &  dispersive  
 			\\
 			\hline 
 			\begin{minipage}{7cm}\medskip
 				the group/phase velocity for the compressional rotational wave in the limit $\omega\to \infty$ ($k\to \infty$)\medskip
 			\end{minipage} & $\mathfrak{c}_{m,p}=\sqrt{\frac{L_c^2(2\alpha_1+\alpha_3)}{\rho \, j\, \tau_c^2}}$ &  dispersive 
 			\\\hline 
 			\begin{minipage}{7cm}\medskip
 				the limit  of the group/phase velocity of the acoustic
 				branch of the shear–rotational wave  at $\omega\to \infty$ ($k\to\infty$)\medskip
 			\end{minipage} & $\mathfrak{c}_{m,s}=\sqrt{\frac{L_c^2(\alpha_1+\alpha_2)}{\rho \, j\, \tau_c^2}}=\sqrt{\frac{L_c^2\,\gamma}{\rho \, j\, \tau_c^2}}$ &  dispersive \\\hline 
 			\begin{minipage}{7cm}\medskip
 				the limit of the optical branch (compressional-rotational and shear-rotational)  at the cut-off frequency $\omega=2\sqrt{\frac{\mu_c}{\rho j \mu_e \tau_c^2}}$, $k=0$\medskip
 			\end{minipage} & \begin{minipage}{3cm}$0$ (group velocity)\\ $/\infty$ (phase
 				velocity) \end{minipage}&  dispersive  \\\hline
 		\end{tabular}
 		\caption{Some group velocities $c=\frac{\omega}{k}$ or/and phase velocities   $\frac{d\,\omega}{d\,k}$ and the cut-off frequency in linear Cosserat elasticity.}\label{T1}\end{center} 
 \end{table}}

{
 \begin{table}[h!]\begin{center}
 		\begin{tabular}{ |c| c | c |}
 			\hline 
 			Name&	Expression& \begin{minipage}{3.5cm} 
 				Dispersive waves/\\Non-dispersive waves	\end{minipage}\\
 			\hline 
 			\begin{minipage}{7cm}\medskip
 				the velocity of the acoustic branch of translational
 				compression (longitudinal) plane wave \medskip
 			\end{minipage}  & ${c}_p= \sqrt{\frac{\lambda_e+2\mu_e
 				}{\rho}}$ &  non-dispersive  \\\hline 
 			\begin{minipage}{7cm}\medskip
 				the limit of the group/phase velocity of the acoustic
 				branch of the shear–rotational wave at $\omega\to 0$ ($k\to 0$)  \medskip
 			\end{minipage} & ${c}_{t} = \sqrt{\frac{\mu_e}{\rho}}$&  non-dispersive
 			\\\hline 
 			\begin{minipage}{7cm}\medskip
 				the limit of the group/phase velocity of the acoustic
 				branch of the shear–rotational wave at $\omega\to \infty$ ($k\to \infty$)\medskip
 			\end{minipage} & ${c}_{t}=\sqrt{\frac{\mu_e}{\rho}}$ &  non-dispersive  
 			\\
 			\hline 
 			\begin{minipage}{7cm}\medskip
 				the group/phase velocity for the compressional rotational wave in the limit $\omega\to \infty$ ($k\to \infty$)\medskip
 			\end{minipage} & $\times$ &  not present   
 			\\\hline 
 			\begin{minipage}{7cm}\medskip
 				the limit of the group/phase velocity of the acoustic
 				branch of the shear–rotational wave  at $\omega\to \infty$ ($k\to\infty$)\medskip
 			\end{minipage} & $\times$ &  not present  \\\hline 
 			\begin{minipage}{7cm}\medskip
 				cut-off frequency\medskip
 			\end{minipage} & $\times$ &  not present  \\\hline
 		\end{tabular}
 	\end{center} 
 	\caption{The group velocities $c=\frac{\omega}{k}$ or/and phase velocities   $\frac{d\,\omega}{d\,k}$ in classical linear elasticity.}\label{T2}
\end{table}}
 
 {Since different authors use different notation for elastic constants,   we will also interpret almost all the conditions upon the constitutive  parameters with the  help of some relations between the velocities of compression/transversal  acoustic/optical 
 	plane waves and with  the help of some cut-off
 	frequency of the optical branches and group velocities of plane bulk waves in the limit
 	of high and low frequencies. This will give a
 	better understanding of the physical sense of the formulae.
To this aim, we will use the  quantities from  Table \ref{T1}, see \cite[Figure 5.11.2., Page 150]{Eringen99} and \cite{neff2017real}. In Table \ref{T2} we list the corresponding informations and notations for classical linear elasticity.}

 	{With the help of the quantities presented in Table \ref{T1}, the
 		first two inequalities of the set of conditions \eqref{Chiritacond} considered by Chiri\c t\u a and Ghiba \cite{ChiritaGhiba3} imply that the translational compressional wave is real and  that the shear-rotational wave (optical branch) is real at high frequencies, the first inequality also implies that at the limit of high frequencies the translational compressional wave is faster than the shear–rotational wave (if they
 		both exist), while the third one means that the
 		shear–rotational wave (acoustic branch) is real at high frequencies. The inequalities \eqref{Chiritacond} do not imply that  the shear-rotational wave (optical branch) is real at low frequencies, i.e., $\omega\to 0$. In fact, the inequalities \eqref{Chiritacond} do not imply that the plane waves are real, i.e., that the propagation plane waves  are defined only by real frequencies.}
 		
 		{ To the contrary, the first implication of the set of conditions
 			\eqref{crw} means that all these waves (compressional/ shear-rotational  waves, acoustic/optical branch) are real. Then, we can treat and interpret further the propagation of plane waves.  It is clear that under conditions 	\eqref{crw} all branches of waves are real for the entire range $[0,\infty)$ of the frequency. However, we may also see directly from  the first condition  \eqref{crw}$_1$ that  the translational compressional wave
 		is real, the second condition \eqref{crw}$_2$ means that the acoustic branch of the shear–rotational
 		wave is real at low frequencies and together with the third condition \eqref{crw}$_3$ means that  the acoustic branch of the shear–rotational
 		wave is real at high frequencies, the fourth inequality \eqref{crw}$_4$ implies that the optical branch of the
 		shear–rotational wave is real at high frequencies. In addition, the third inequality \eqref{crw}$_3$ means that the
 		optical branch of the shear–rotational wave at high frequencies has a larger
 		velocity than the acoustic branch of the same wave at low frequencies (if they
 		both exist, which is the case due to other conditions). We have just given the interpretation of the first four inequalities, but
 		the fifth one expresses directly that  the compressional rotational
 		wave at high frequencies is real.}
 	
 	{ Considering the quantities from Tables \ref{T1} and \ref{T2}, 
 		 we may infer from \eqref{LHce} that in  classical elasticity the strong ellipticity conditions correspond
 		to the existence of real compressional and shear waves, while  in the
 		Cosserat case the corresponding conditions \eqref{LHc} (strong ellipticity conditions, Legendre-Hadamard ellipticity, the positive definiteness of the acoustic tensor) implies  the existence of the real translational compressional
 		wave on the entire range of real frequencies, of the real shear–rotational waves (both branches) at high frequencies, and of real
 		rotational compressional waves at high frequencies, but at lower frequencies the
 		latter waves may not be real since \eqref{LHc} does not guarantee that $\mathfrak{c}_t=\frac{\mu_{\rm e}}{\rho}$ is real. To the contrary, the conditions \eqref{crw} imply that all these branches and types of plane wave are real, i.e., the group/phase velocities are real on the entire range of possible frequencies.}
 
 The structure  of the present paper is now the following. In Section \ref{Sp}, after a short introduction of our notation,  we present the linear Cosserat model for isotropic elastic materials as a special case of the relaxed micromorphic model. This comparison establishes the relations between these models and it is also useful for further studies, where the results obtained in the linear relaxed micromorphic model will be compared with those established in the Cosserat theory of linear elastic materials. Then, since a self contained study of the propagation of real waves in isotropic Cosserat elasticity  is still missing, as well as the explicit form of the conditions on the constitutive coefficient which imply it, we dedicate  Subsection \ref{Rpw} to it. In Subsection \ref{setRw} we present the setup in the propagation of Rayleigh waves and we  discuss the  propagation of some special real plane wave which are related to  the propagation of seismic waves. In Section \ref{anzsec} we give the ansatz of the solution and we define the limiting speed, by making the relation between the seismic waves and the special real plane waves considered in Subsection \ref{setRw}. In Section \ref{comSec} we present the common method to construct the solution using the Stroh formalism and we put it in relation with our method, in order to conclude the section with some auxiliary results. In Section \ref{NSE} we establish the new form of the secular equation using the novel method which allow to give the main result of the paper, i.e., the first proof in the literature of the existence and uniqueness of the subsonic speed of the secular equation, in the framework of isotropic linear Cosserat elastic models. 
In Section \ref{Num1} we provide two numerical algorithms which can be implemented for any  material once   the constitutive coefficients are known. We present effective  numerical results for an alumunium-epoxy composite. In Section \ref{Classic}   we provide the explicit form of the secular equations for isotropic linear elastic materials, since its  explicit form  has not been derived elsewhere and we compare it with other forms from the literature. We conclude the paper with some final remarks.

\section{Statement of the problem}\label{Sp}\setcounter{equation}{0}

\subsection{Notation}\label{NS}
We consider that the mechanical behaviour of a body accupying  the unbounded regular region of three dimensional Euclidean space  is modelled with the help of the Cosserat theory of  linear isotropic elastic materials. We denote  by $n$ the outward unit normal on $\partial\Omega$. The body is referred to a fixed system of rectangular Cartesian axes  $Ox_i (i=1,2,3)$ , $\{e_1, e_2, e_3\}$ being the unit vectors of these axes.

In the following, we recall some useful notations for the present work. For $ a,b\in \mathbb{R}^{3 \times 3} $ we let $\langle{a,b}\rangle_{\mathbb{R}^3}$ denote the scalar product on $ \mathbb{R}^3$ with associated vector norm $\norm{a}^2=\langle {a,a}\rangle$. We denote by $\mathbb{R}^{3\times 3}$   the set of real $3 \times 3$ second order tensors, written with capital letters. Matrices will be denoted by bold symbols, e.g. $\mathbf{X}\in \mathbb{R}^{3\times 3}$, while $X_{ij}$ will denote its component. The standard Euclidean product on $\mathbb{R}^{3 \times 3}$ is given by $\langle{ \mathbf{X},\mathbf{Y}}\rangle_{\mathbb{R}^{3 \times 3}}=\tr( \mathbf{X}\,\mathbf{Y}^T)$, and thus, the Frobenious tensor norm is $\norm{ \mathbf{X}}^2=\langle{ \mathbf{X}, \mathbf{X}}\rangle_{\mathbb{R}^{3 \times 3}}$. In the following we omit the index $\mathbb{R}^3, \mathbb{R}^{3\times 3}$. The identity tensor on $\mathbb{R}^{3\times 3}$ will be denoted by $\id$, so that $\tr( \mathbf{X})=\langle{ \mathbf{X},\id}\rangle$. We let $\Sym$  denote the set of symmetric tensors. We adopt the usual abbreviations of  Lie-algebra theory, i.e., $\mathfrak{so}(3):=\{ \mathbf{A}\,\in \mathbb{R}^{3\times 3}| \mathbf{A}^T=- \mathbf{A}\}$ is the Lie-algebra of skew-symmetric tensors and $\mathfrak{sl}(3):=\{ \mathbf{X}\, \in \mathbb{R}^{3\times 3} |\tr( \mathbf{X})=0 \}$  is the Lie-algebra of traceless tensors. For all $ \mathbf{X} \in\mathbb{R}^{3\times 3}$  we set $\sym\,  \mathbf{X}=\frac{1}{2}( \mathbf{X}^T+ \mathbf{X}) \in\Sym,\,\skw  \mathbf{X}=\frac{1}{2}( \mathbf{X}- \mathbf{X}^T)\in \mathfrak{so}(3)$ and the deviatoric (trace-free) part $\dev\, \mathbf{X}= \mathbf{X}-\frac{1}{3}\tr( \mathbf{X})\in\, \mathfrak{sl}(3)$ and we have the orthogonal Cartan-decomposition of the Lie-algebra $
\mathfrak{gl}(3)=\{\mathfrak{sl}(3)\cap \Sym(3)\}\oplus\mathfrak{so}(3)\oplus\mathbb{R}\cdot\id,\ 
 \mathbf{X}= \dev\,\sym\, \mathbf{X}+\skw\, \mathbf{X}+\frac{1}{3}\tr( \mathbf{X})\,\id.\
$
 We use the canonical identification of $\mathbb{R}^3$ with $\so(3)$, and, for
 \begin{align}
 \mathbf{A}=	\begin{footnotesize}\begin{pmatrix}
 0 &-a_3&a_2\\
 a_3&0& -a_1\\
 -a_2& a_1&0
 	\end{pmatrix}\end{footnotesize}\in \so(3)
 \end{align}
 we consider the operators $\axl\,:\so(3)\rightarrow\mathbb{R}^3$ and $\anti:\mathbb{R}^3\rightarrow \so(3)$ through
 \begin{align}
 \axl\,( \mathbf{A}):&=\left(
 a_1,
 a_2,
 a_3
 \right)^T,\quad \quad  \mathbf{A}.\, v=(\axl\,  \mathbf{A})\times v, \qquad \qquad (\anti(v))_{ij}=-\epsilon_{ijk}\,v_k, \quad \quad \forall \, v\in\mathbb{R}^3,
 \notag \\(\axl\,  \mathbf{A})_k&=-\frac{1}{2}\, \epsilon_{ijk} \mathbf{A}_{ij}=\frac{1}{2}\,\epsilon_{kij} {A}_{ji}\,, \quad  {A}_{ij}=-\epsilon_{ijk}\,(\axl\,  \mathbf{A})_k=:\anti(\axl\,  \mathbf{A})_{ij},
 \end{align}
 where $\epsilon_{ijk}$ is the totally antisymmetric third order permutation tensor.
 
 For a  regular enough function $f(t,x_1,x_2,x_3)$,  $f_{,t}$ denotes the derivative with respect to the time $t$, while  $ \frac{\partial\, f}{\partial \,x_i}$ denotes the $i$-component of the gradient $\nabla f$.  For vector fields $u=\left(    u_1, u_2, u_3\right)^T$ with  $u_i\in 
   {\rm H}^1(\Omega)\,=\,\{u_i\in {\rm L}^2(\Omega)\, |\, \nabla\, u_i\in {\rm L}^2(\Omega)\}, $  $i=1,2,3$,
we define
$
\mathrm{D} u:=\left(
\nabla\,  u_1\,|\,
\nabla\, u_2\,|\,
\nabla\, u_3
\right)^T.
$
The corresponding Sobolev-space will be also denoted by
$
{\rm H}^1(\Omega)$.  In addition, for a tensor field
  $\mathbf{P}$ with rows in ${\rm H}({\rm div}\,; \Omega)$, i.e.,
$
 \mathbf{P}=\begin{footnotesize}\begin{footnotesize}\begin{pmatrix}
 \mathbf{P}^T.e_1\,|\,
 \mathbf{P}^T.e_2\,|\,
 \mathbf{P}^T\, e_3
\end{pmatrix}\end{footnotesize}\end{footnotesize}^T$ with $( \mathbf{P}^T.e_i)^T\in {\rm H}({\rm div}\,; \Omega):=\,\{v\in {\rm L}^2(\Omega)\, |\, {\rm div}\, v\in {\rm L}^2(\Omega)\}$, $i=1,2,3$,
we define
$
{\rm Div}\, \mathbf{P}:=\begin{footnotesize}\begin{footnotesize}\begin{pmatrix}
{\rm div}\, ( \mathbf{P}^T.e_1)^T\,|\,
{\rm div}\, ( \mathbf{P}^T.e_2)^T\,|\,
{\rm div}\, ( \mathbf{P}^T\, e_3)^T
\end{pmatrix}\end{footnotesize}\end{footnotesize}^T
$
while for tensor fields $ \mathbf{P}$ with rows in ${\rm H}({\rm curl}\,; \Omega)$, i.e.,
$
 \mathbf{P}=\begin{footnotesize}\begin{footnotesize}\begin{pmatrix}
 \mathbf{P}^T.e_1\,|\,
 \mathbf{P}^T.e_2\,|\,
 \mathbf{P}^T\, e_3
\end{pmatrix}\end{footnotesize}\end{footnotesize}^T$ with $( \mathbf{P}^T.e_i)^T\in {\rm H}({\rm curl}\,; \Omega):=\,\{v\in {\rm L}^2(\Omega)\, |\, {\rm curl}\, v\in {\rm L}^2(\Omega)\}
$, $i=1,2,3$,
we define
$
{\rm Curl}\, \mathbf{P}:=\begin{footnotesize}\begin{footnotesize}\begin{pmatrix}
{\rm curl}\, ( \mathbf{P}^T.e_1)^T\,|\,
{\rm curl}\, ( \mathbf{P}^T.e_2)^T\,|\,
{\rm curl}\, ( \mathbf{P}^T\, e_3)^T
\end{pmatrix}\end{footnotesize}\end{footnotesize}^T
.
$

\subsection{Cosserat theory of isotropic elastic solids as particular case of the relaxed micromorphic model}\label{subsC}

In this subsection we show that  the dynamic Cosserat model for isotropic materials \cite{Eringen99,Mindlin64} is  not only a special case of the most general micromorphic model, but also a special case of the {\it relaxed micromorphic model} \cite{NeffGhibaMicroModel,MadeoNeffGhibaW,MadeoNeffGhibaWZAMM,madeo2016reflection,NeffGhibaMadeoLazar}. In the micromorphic theory, the micro-distortion tensor $ \mathbf{P}=( \mathbf{P}_{ij}):\Omega\times [0,T]\rightarrow \mathbb{R}^{3 \times 3}$  describes the substructure of the material which can rotate, stretch, shear and shrink, while $u=(u_i) :\Omega\times [0,T]\rightarrow  \mathbb{R}^3$  is the displacement of the macroscopic material points.

In the relaxed micromorphic model, in which the  Cosserat modulus $\mu_{\rm c} >0$ is related to the isotropic Eringen-Claus  model for dislocation dynamics \cite{Eringen_Claus69,EringenClaus,Eringen_Claus71}, {the free energy is given by }
\begin{align}\label{XXXX}
W_{\rm relax}&=\mu_{\rm e}  \|\sym (\mathrm{D}u -\mathbf{P})\|^2+\mu_{\rm c} \|\skw(\mathrm{D}u -\mathbf{P})\|^2+ \frac{\lambda_{\rm e}}{2}\, [\tr(\mathrm{D}u -\mathbf{P})]^2+\mu_{\rm micro} \|\sym \, \mathbf{P}\|^2+ \frac{\lambda_{\rm micro}}{2} [\tr(\mathbf{P})]^2\notag\\&
\quad \quad  +\frac{\mu_{\rm e}L_{\rm c}^2}{2}\left[{a_1}\| \dev\,\sym \,\Curl\, \mathbf{P}\|^2 +{a_2}\| \skw \,\Curl\, \mathbf{P}\|+ \frac{a_3}{3}\, \tr(\Curl\, \mathbf{P})^2\right],
\end{align}
where $(\mu_{\rm e},\lambda_{\rm e}) $,  $(\mu_{\rm micro},\lambda_{\rm micro})$,  $\mu_{\rm c}, L_{\rm c} $  and $( a_1, a_2, a_3)$  are the elastic moduli  representing the parameters related to the meso-scale,  the
parameters related to the micro-scale the Cosserat couple modulus, the characteristic length, and the three
general isotropic curvature parameters (weights), respectively.  Formally,  letting  $L_{\rm c}\to \infty$  means a ``zoom" into the micro-structure while $L_{\rm c}\to 0$ means considering  arbitrary large bodies while retaining the size of the unit-cell or keeping the dimensions of the body fixed while reducing the dimensions of the unit cell to zero, or, in other words, ``no special effects of the microstructure taking into account" (classical elasticity).

In the  internal energy is positive definite in terms of the independent constitutive variables $\mathrm{D}u -\mathbf{P}$, $\sym \, \mathbf{P}$, $\Curl\, \mathbf{P}$ if and only if \begin{align}\mu_{\rm e}&>0,\qquad \qquad\quad \ \,  \kappa_{\rm e}:=\frac{2\,\mu_{\rm e}+3\,\lambda_{\rm e}}{3}>0,\qquad\qquad\qquad  \ \  \mu_{\rm c}>0, \notag \\\mu_{\rm micro}&>0, \qquad\qquad \kappa_{\rm micro}:=\frac{2\,\mu_{\rm micro}+3\,\lambda_{\rm micro}}{3}>0, \\ a_1&>0, \qquad \qquad a_2>0, \qquad\qquad  a_3> 0.\notag\end{align}

The complete system of linear partial differential equations in terms of the kinematical unknowns $u$ and $P$ is given by
\begin{align}\label{eqisoup}
\rho\,u_{,tt}&={\rm Div}[\underbrace{2\,\mu_{\rm e} \, \sym(\mathrm{D}u -\mathbf{P})+2\,\mu_{\rm c} \, \skw(\mathrm{D}u -\mathbf{P})+\lambda_{\rm e} \,\tr(\mathrm{D}u -\mathbf{P}){\cdot} \id}_{ \textrm{ the non-symmetric force-stress tensor}}]+f\, ,\\\notag
\rho\, \eta\,\tau_{\rm c}^2\,\,P_{,tt}&=-{\mu_{\rm e}L_{\rm c}^2}\,\Curl [\underbrace{a_1 \, \dev\,\sym \,\Curl\, \mathbf{P}+a_2\, \skw \,\Curl\, \mathbf{P} +\frac{a_3}{3}\, \tr(\Curl\, \mathbf{P}){\cdot} \id}_{\textrm{the second-order moment stress  tensor}}]\\&\quad\ +2\,\mu_{\rm e} \, \sym(\mathrm{D}u -\mathbf{P})+\lambda_{\rm e} \tr(\mathrm{D}u -\mathbf{P}){\cdot} \id-2\,\mu_{\rm micro}\, \sym \, \mathbf{P}-\lambda_{\rm micro} \tr (\mathbf{P}){\cdot} \id+\mathbf{M}\,  \ \ \ \text {in}\ \ \  \Omega\times [0,T],\notag
\end{align}
where $f :\Omega\times [0,T]\rightarrow  \mathbb{R}^3$ describes the external body force,  $\mathbf{M}:\Omega\times [0,T]\rightarrow \mathbb{R}^{3 \times 3}$ describes the external body moment, $\rho$ is the mass density and  $\eta\,\tau_{\rm c}^2$ is the inertia coefficient, with $\eta>0$ a weight parameter and $\tau_{\rm c}$ the internal characteristic time \cite[page 163]{Eringen99}. 

In the Cosserat theory we assume that the micro-distortion tensor is skew-symmetric, i.e.  $\mathbf{P}=\mathbf{A}\in \so(3)$. Using the  Curl-$\mathrm{D}\,\axl\,$ identities,  (see \cite{Neff_curl06}, Nye's formula \cite{Nye53})
\begin{align}\label{curlaxl}
-\Curl\, \mathbf{A}&=(\mathrm{D}\, \axl  \,\mathbf{A})^T-\tr[(\mathrm{D}\, \axl  \,\mathbf{A})^T]{\cdot} \id,\qquad \qquad 
\mathrm{D}\, \axl  \,\mathbf{A}  = -(\Curl\, \mathbf{A})^T+\frac{1}{2}\tr[(\Curl\, \mathbf{A})^T]{\cdot}\id,
\end{align}
for all matrix fields $\mathbf{A}\in \so(3)$, it is easy to obtain
that the total energies admits the form (identifying $\vartheta=\axl  \,\mathbf{A}$)
\begin{align}
\mathcal{L}&(u_{,t},(\axl\, \,\mathbf{A})_{,t},\mathrm{D}u -\mathbf{A},\axl\, \,\mathbf{A})\notag\\
&=\int_\Omega\bigg(\frac{1}{2}\,\rho\,\|u_{,t}\|^2+\,\rho\,\eta\,\tau_{\rm c}^2\,\,\,\|(\axl\, \,\mathbf{A})_{,t}\|^2
+\mu_{\rm e}  \|\sym \,\mathrm{D}u  \|^2+\mu_{\rm c} \|\skw(\mathrm{D}u -\mathbf{A})\|^2+ \frac{\lambda_{\rm e}}{2} [\tr(\mathrm{D}u )]^2
\\\notag&
\ \ \ \ \ \ \ \ \ \ \quad \quad  +\frac{\mu_{\rm e}L_{\rm c}^2}{2}\left[{a_1 } \|\dev\,\sym (\mathrm{D}\, \axl  \,\mathbf{A})\|^2+{a_2}\|\skw (\mathrm{D}\, \axl  \,\mathbf{A})\|^2+ \frac{4\,a_3}{3}\,[\tr(\mathrm{D}\, \axl  \,\mathbf{A})]^2\right]\bigg)\,dv\,
\end{align}
and lead to  Euler-Lagrange equations which are equivalent to those derived from the $\Curl$-formulation.

The power functional is given by
\begin{align}
\Pi(t)&=\int_\Omega (\dd\langle f,{u}_{,t}\rangle +\langle \mathbf{M},\mathbf{A}_{,t}\rangle)\, dv=\int_\Omega (\dd\langle f,{u}_{,t}\rangle +2\langle \axl\,\skw \,M,\axl\, \,\mathbf{A}_{,t}\rangle)\, dv\, .
\end{align}
\allowdisplaybreaks
We introduce the action functional of the considered system to be
	defined as
	\begin{align}
	\mathcal{A}=&\int_{0}^{T}\int_\Omega\bigg(\frac{1}{2}\,\rho\,\|u_{,t}\|^2+\,\rho\,\eta\,\tau_{\rm c}^2\,\,\,\|(\axl\, \,\mathbf{A})_{,t}\|^2
	-\left[\mu_{\rm e}  \|\sym \mathrm{D}u  \|^2+\mu_{\rm c} \|\skw(\mathrm{D}u -\mathbf{A})\|^2+ \frac{\lambda_{\rm e}}{2} [\tr(\mathrm{D}u )]^2\right]
	\notag\\ &\qquad \qquad -\frac{\mu_{\rm e}L_{\rm c}^2}{2}\left[{a_1 } \|\dev\,\sym (\mathrm{D}\, \axl  \,\mathbf{A})\|^2+{a_2} \|\skw (\mathrm{D}\, \axl  \,\mathbf{A})\|^2+ \frac{4\,a_3}{3}\,[\tr(\mathrm{D}\, \axl  \,\mathbf{A})]^2\right]\Bigg)dv\,dt\notag\\\qquad \qquad \qquad  &+\int_{0}^{T}\int_\Omega (\dd\langle f,{u}\rangle +2\langle \axl\,\skw \,\mathbf{M},\axl\, \,\mathbf{A}\rangle)\, dv\,dt.
	\end{align}
	The condition of vanishing first variation of the action functional can thus be written as
	\begin{align}\label{actionCos}
	-\int_0^T\int_\Omega &(\dd\langle f,\delta {u}\rangle +2\langle \axl\,\skw \,\mathbf{M},\axl\, (\delta \mathbf{A})\rangle)\, dv\,dt\notag\\
	=&\int_0^T\int_\Omega\bigg(\,\rho\,\bigl\langle u_{,t},(\delta u)_{,t}\bigr\rangle+2\,\rho\,\eta\,\tau_{\rm c}^2\,\,\,\bigl\langle(\axl\, \,\mathbf{A})_{,t}, (\axl\, (\delta \mathbf{A}))_{,t}\bigr\rangle\notag
	\\&\qquad \ \ \ -2\,\mu_{\rm e}  \bigl\langle\sym \mathrm{D}\,  u,\sym \mathrm{D}\, \delta u \bigr\rangle-\mu_{\rm c} \,\bigl\langle\skw(\mathrm{D}u -\mathbf{A}),\skw(\mathrm{D}\, \delta u-\delta \mathbf{A})\bigr\rangle- {\lambda_{\rm e}}[\tr(\mathrm{D}u )]\,[\tr(\mathrm{D}\, \delta u)]\notag
	\\\notag&
	\ \ \ \quad \quad  -\frac{\mu_{\rm e}L_{\rm c}^2}{2} \,a_1  \bigl\langle\dev\,\sym (\mathrm{D}\, \axl  \,\mathbf{A}),\dev\,\sym (\mathrm{D}\, \axl\, (\delta \mathbf{A}))\bigr\rangle-\frac{\mu_{\rm e}L_{\rm c}^2}{2} a_2 \bigl\langle\skw (\mathrm{D}\, \axl   \mathbf{A}),\skw (\mathrm{D}\, \axl\, (\delta \mathbf{A}))\bigr\rangle\\
	&\ \ \ \quad \quad - \frac{\mu_{\rm e}L_{\rm c}^2}{2}\,\frac{4\,a_3}{3}\,[\tr(\mathrm{D}\, \axl   \mathbf{A})]\,[\tr(\mathrm{D}\, \axl\, (\delta \mathbf{A}))]\bigg)\,dv\, dt\notag\\
	=&\int_0^T\int_\Omega\bigg(\,\rho\,\bigl\langle u_{,t},(\delta u)_{,t}\bigr\rangle+2\, \,\rho\,\eta\,\tau_{\rm c}^2\,\,\,\bigl\langle(\axl\, \,\mathbf{A})_{,t}, (\axl\, (\delta \mathbf{A}))_{,t}\bigr\rangle\notag
	\\&\qquad \ \ \ -\bigl\langle 2\,\mu_{\rm e}  \sym \,\mathrm{D}\,  u+\mu_{\rm c} \,\skw(\mathrm{D}u -\mathbf{A})+ {\lambda_{\rm e}}[\tr(\mathrm{D}u )]\cdot \id, \mathrm{D}\, \delta u\bigr\rangle-2\,\mu_{\rm c} \,\bigl\langle\skw(\mathrm{D}u -\mathbf{A}),\delta \mathbf{A}\bigr\rangle
	\\\notag&
	\ \ \ \quad \quad  -\frac{\mu_{\rm e}L_{\rm c}^2}{2}\,\bigl\langle {a_1 } \dev\,\sym (\mathrm{D}\, \axl  \,\mathbf{A})+{a_2} \skw (\mathrm{D}\, \axl   \mathbf{A}) +\frac{4\,a_3}{3}\,[\tr(\mathrm{D}\, \axl   \mathbf{A})]\cdot\id, \mathrm{D}\, \axl\, (\delta \mathbf{A})\bigr\rangle\bigr\rangle\bigg)\,dv\, dt\notag\\
	=&\int_0^T\int_\Omega\bigg(-\,\rho\,\bigl\langle u_{,tt},\delta u\bigr\rangle-2\,\rho\,\eta\,\tau_{\rm c}^2\,\,\,\bigl\langle(\axl\, \,\mathbf{A})_{,tt}, \axl\, (\delta \mathbf{A})\bigr\rangle\notag
	\\&\qquad \ \ \ +\bigl\langle {\rm Div}[2\,\mu_{\rm e} \, \sym \mathrm{D}\,  u+\mu_{\rm c} \,\skw(\mathrm{D}u -\mathbf{A})+ {\lambda_{\rm e}}[\tr(\mathrm{D}u )]\cdot \id],  \delta u\bigr\rangle-4\,\mu_{\rm c} \,\bigl\langle\axl\, \skw(\mathrm{D}u -\mathbf{A}),\axl\, (\delta \mathbf{A})\bigr\rangle\notag
	\\\notag&
	\ \ \ \quad \quad  +\frac{\mu_{\rm e}L_{\rm c}^2}{2}\,\bigl\langle {\rm Div}[{a_1} \dev\,\sym (\mathrm{D}\, \axl  \,\mathbf{A})+{a_2} \skw (\mathrm{D}\, \axl   \mathbf{A})+\frac{4\,a_3}{3}\,[\tr(\mathrm{D}\, \axl   \mathbf{A})]\cdot\id],  \axl\, (\delta \mathbf{A})\bigr\rangle\bigr\rangle\bigg)\,dv\,dt\notag
	\end{align}
	for all  virtual displacements $\delta u\in C_0^{\infty}([0,T],C_0^{\infty}(\Omega))$ and for all virtual axial vectors $ \axl\,(\delta \mathbf{A})\in C_0^{\infty}([0,T],C_0^{\infty}(\Omega))$.
	Summarizing, in view of \eqref{actionCos}, the Euler-Lagrange equation gives us the following system of partial differential equations for $u$ and $\mathbf{A}$
	\begin{align}\label{eqisaxl}
	\rho\,u_{,tt}&=\Div[2\,\mu_{\rm e} \, \sym\,\mathrm{D}u +2\,\mu_{\rm c} \,\skw(\mathrm{D}u -\mathbf{A})+ \lambda_{\rm e}\, \tr(\mathrm{D}u ){\cdot} \id]+f\, ,\notag\\
	2\,\rho\,\eta\,\tau_{\rm c}^2\,\,\,(\axl\, \,\mathbf{A})_{,tt}&=\frac{\mu_{\rm e}L_{\rm c}^2}{2}\,\bigg[{a_1 }\,\dev\,\sym (\mathrm{D}\, \axl  \,\mathbf{A})+{a_2}\,\skw (\mathrm{D}\, \axl  \,\mathbf{A})+ \frac{4\,a_3}{3}\,\tr(\mathrm{D}\, \axl  \,\mathbf{A}){\cdot } \id\bigg]
	\\&\ \ \ \ \ \ -4\,\mu_{\rm c} \,\axl\,(\skw\,\mathrm{D}u -\mathbf{A})+2\,\axl\,\skw \,\mathbf{M}\,  \ \ \ \text {in}\ \ \  \Omega\times [0,T],\notag
	\end{align}
	which is in complete agreement to the equations proposed in the Cosserat theory \cite{Eringen99}, but written in indices.

Therefore, in the Cosserat theory of linear elastic materials, two vector fields are used to describe the macro- and micro-behaviour of the solid body, i.e.,  
 the {displacement} $u:\Omega \subset\mathbb{R}^3\mapsto\mathbb{R}^3$ and the   microrotation vector field $ \vartheta   :\Omega \subset\mathbb{R}^3\mapsto\mathbb{R}^3$, $ \vartheta   ={\rm axl}\, \mathbf{A}$, where $\mathbf{A}\in \mathfrak{so}(3)$ represents the micro-distortion tensor in the  Cosserat theory, and  within the framework of the linear isotropic hyperelastic theory,  the elastic energy density  of the Cosserat model read
  \begin{align}
  W(\mathrm{D}u ,{\rm Anti}\, \vartheta , \mathrm{D} \vartheta    )=& \underbrace{\mu_{\rm e} \,\lVert \dev_3\, \sym\,\mathbf{e}\rVert ^{2}+\mu_{\rm c}\,\lVert \skw\,\mathbf{e}\rVert ^{2}+\frac{2\,\mu_{\rm e} +3\,\lambda_{\rm e} }{6}\left[\mathrm{tr} \left(\mathbf{e}\right)\right]^{2}}_{:=W_1(\mathbf{e})}\notag\\& +\underbrace{{\mu_{\rm e}L_{\rm c}^2}\,\left[\ {\alpha_1}\,\lVert \dev_3\,\sym\,\mathbf{\mathfrak{K}}\rVert ^{2}+{\alpha_2}\,\lVert \skw\,\mathbf{\mathfrak{K}}\rVert ^{2}+\frac{2\,\alpha_1+3\,\alpha_3}{6}\left[\mathrm{tr} \left(\mathbf{\mathfrak{K}}\right)\right]^{2}\right]}_{:=W_2(\mathbf{\mathfrak{K}})},
  \end{align}
  where  we have used the  notations for the  weight parameters $\alpha_1=\frac{a_1 }{2}$, $\alpha_2=\frac{a_2}{2}$  and $\alpha_3=4\,\frac{a_3}{3}$ and the following definitions for the independent constitutive  variables $\mathbf{e}$ and $\mathbf{\mathfrak{K}}$ (geometrical equations)
  \begin{align}
 \mathbf{e}:=\mathrm{D}u -\mathbf{A}=\mathrm{D}u -{\rm Anti}\, (\vartheta)   , \qquad \qquad \mathbf{\mathfrak{K}}:=\mathrm{D} \vartheta    .
  \end{align}

  Hence, the stress-strain relations for the homogeneous  isotropic Cosserat elastic solid are
 \begin{align}\label{01}
 \boldsymbol{\sigma}&:=\frac{\partial\, W}{\partial\, \mathbf{e}}=2\,\mu_{\rm e} \, \sym \,\mathbf{e}+2\,\mu_{\rm c} \, \skw\, \mathbf{e}+\lambda_{\rm e} \,\tr (\mathbf{e})\, \id,\notag\\ 
\mathbf{ m}&:=\frac{\partial\, W}{\partial\, \mathbf{\mathfrak{K}}} ={\mu_{\rm e}L_{\rm c}^2}\,\left[2\,{\alpha_1}\, \sym\,\mathbf{\mathfrak{K}} +2\,{\alpha_2}\,\skw\, \mathbf{\mathfrak{K}}+\alpha_3\,\tr(\mathbf{\mathfrak{K}})\,\id\right],
 \end{align}
 where $\boldsymbol{\sigma}$ is the non-symmetric force stress tensor and $\mathbf{m}$ is the second-order non-symmetric couple stress tensor. In the absence of  external  body forces and of   external  body moment,  the PDE-system of the model is
\begin{align}\label{PDE}
\rho\,\frac{\partial^2 \,u_i}{\partial\, t^2}&=(\mu_{\rm e} +\mu_{\rm c})\dd\sum_{l=1}^3\frac{\partial^2 u_i }{\partial x_l^2}+(\mu_{\rm e} -\mu_{\rm c}+\lambda_{\rm e} )\dd\sum_{l=1}^3\frac{\partial^2 u_l }{\partial x_l\partial x_i}+2\,\mu_{\rm c} \,\sum_{l,s=1}^3\varepsilon _{ils}\frac{\partial \vartheta _s }{\partial x_l},
\vspace{1.2mm}\\
\rho\, j\,\mu_{\rm e}\,\tau_{\rm c}^2\,\frac{\partial^2 \,\vartheta_i}{\partial\, t^2}&={\mu_{\rm e}L_{\rm c}^2}\,\left[(\alpha_1+\alpha_2)\dd\sum_{l=1}^3\frac{\partial^2 \vartheta _i }{\partial x_l^2}+(\alpha_1-\alpha_2+\alpha_3)\dd\sum_{l=1}^3\frac{\partial^2 \vartheta_l }{\partial x_l\partial x_i}\right]+2\,\mu_{\rm c} \,\sum_{l,s=1}^3\varepsilon _{isl}\frac{\partial u_l }{\partial x_s}-4\, \mu_{\rm c}\, \vartheta _i, \notag
\end{align}
with $i=1,2,3$, where $j=2\, \eta$ is an inertia weight parameter.

 Due to the orthogonal Cartan-decomposition of the Lie-algebra $\mathfrak{gl}(3)$, 	the strict positive definiteness of the potential energy is equivalent to the following simple relations for the introduced parameters
\begin{align}
 \quad\mu_{\rm e}  >0, \qquad \quad\mu_{\rm c}>0,\qquad \quad  2\,\mu_{\rm e} +3\,\lambda_{\rm e} >0,\qquad \quad \alpha_1>0, \qquad \quad\alpha_2>0, \qquad \quad 2\,\alpha_1+3\,\alpha_3>0 . \label{posCoss}
\end{align}
However, our entire subsequent analysis will be made under weaker conditions on the constitutive parameters.

In the following we assume $\rho>0$ and $j>0$ without mentioning these conditions in the hypothesis of our results.
\subsection{Real plane waves in isotropic Cosserat elastic solids}\label{Rpw}

We say that there exists  {\it real plane waves}  in the  direction $\xi=(\xi_1,\xi_2,\xi_3)$, $\lVert{\xi}\rVert^2=1$,  if for every  wave number $k>0$  the system of partial differential equations \eqref{PDE} admits a solution in the form:
\begin{align}\label{ansatzwp}
u(x_1,x_2,x_3,t)&=\underbrace{\begin{footnotesize}\begin{pmatrix}\widehat{u}_1\\\widehat{u}_2\\\widehat{u}_3\end{pmatrix}\end{footnotesize}}_{:=\,\widehat{u}}
 \, e^{{\rm i}\, \left(k\langle {\xi},\, x\rangle_{\mathbb{R}^3}-\,\omega \,t\right)}\,,\\ \vartheta (x_1,x_2,x_3,t)&=\underbrace{\begin{footnotesize}{\rm i}\,\begin{pmatrix}\widehat{\vartheta }_1\\\widehat{\vartheta }_2\\\widehat{\vartheta }_3\end{pmatrix}\end{footnotesize}}_{:=\,\widehat{\vartheta}} \, e^{{\rm i}\, \left(k\langle {\xi},\, x\rangle_{\mathbb{R}^3}-\,\omega \,t\right)},\quad 
 \widehat{u}, \widehat{ \vartheta   }\in\mathbb{C}^{3}, \quad (\widehat{u}, \widehat{ \vartheta   })^T\neq 0\,,\notag
\end{align}
only for  real frequencies $\omega\in \mathbb{R}$, where ${\rm i}\,=\sqrt{-1}$ is the complex unit.  The plane wave is called ``real'' since it is defined by real values of $\omega$. Note that we take ${\rm i}\,\widehat{\vartheta }$  since this choice will lead us in the end only to real valued matrices. Otherwise, we would have to deal with complex valued matrices in the linear Cosserat theory. 

The functions \eqref{ansatzwp} are a solution of \eqref{PDE} if and only if the following system is satisfied
\begin{align}\label{algPDE}
-\omega^2\rho\, \widehat{u}_i&=-k^2\,(\mu_{\rm e} +\mu_{\rm c})\dd {\widehat{u}_i} -k^2(\mu_{\rm e} -\mu_{\rm c}+\lambda_{\rm e} )\dd\sum_{l=1}^3\widehat{u}_l\, \xi_i\,\xi_l -2\,k\,\mu_{\rm c}\,\sum_{l,s=1}^3\varepsilon _{ils}\,\widehat{\vartheta }_s\, \xi_l,
\vspace{1.2mm}\notag\\
-{\rm i}\,\omega^2\rho\, j\,\mu_{\rm e}\,\tau_{\rm c}^2\,\widehat{ \vartheta }_i&={\mu_{\rm e}\,L_{\rm c}^2}[-{\rm i}\,k^2\,(\alpha_1+\alpha_2)\,{\widehat{ \vartheta }_i}\,\dd-{\rm i}\,k^2\,(\alpha_1-\alpha_2+\alpha_3)\dd\sum_{l=1}^3\widehat{\vartheta }_l \xi_l \xi_i]\\&\qquad -2\,{\rm i}\,k\,\mu_{\rm c}\,\sum_{l,s=1}^3\varepsilon _{ils}\widehat{u}_s \xi_l-4\,{\rm i}\, \mu_{\rm c}\, \widehat{\vartheta }_i, \quad  i=1,2,3.\notag
\end{align}

However, since our formulation is isotropic, by  demanding real plane waves in any direction $\xi=(\xi_1,\xi_2,\xi_3)$, $\lVert \xi\rVert=1$, it is equivalent to demand real plane waves in the direction $e_1=(1,0,0)$ which means that for all $k>0$ the system 
\begin{align}\label{ansatzwp1}
u(x_1,x_2,t)&=\underbrace{\begin{footnotesize}\begin{pmatrix}\widehat{u}_1\\\widehat{u}_2\\\widehat{u}_3\end{pmatrix}\end{footnotesize}}_{:=\,\widehat{u}}
\, e^{{\rm i}\, \left(k\,x_1-\,\omega \,t\right)}\,,\quad \vartheta (x_1,x_2,t)=\underbrace{\begin{footnotesize}{\rm i}\,\begin{pmatrix}\widehat{\vartheta }_1\\\widehat{\vartheta }_2\\\widehat{\vartheta }_3\end{pmatrix}\end{footnotesize}}_{:=\,\widehat{\vartheta}} \, e^{{\rm i}\, \left(k\,x_1-\,\omega \,t\right)},\quad 
\widehat{u}, \widehat{ \vartheta   }\in\mathbb{C}^{3}, \quad (\widehat{u}, \widehat{ \vartheta   })^T\neq 0\,
\end{align}   
admits non trivial solutions  only for real positive values $\omega^2$.

Inserting  \eqref{ansatzwp1} into \eqref{PDE} we see that  $\widehat{u}_1, \widehat{u}_2$ and $\widehat{\vartheta}_3$ have to satisfy the following linear algebraic equations
\begin{align}\label{algPDE1}
-\omega^2\rho\, \widehat{u}_1&=-k^2\,(\mu_{\rm e} +\mu_{\rm c})\dd \widehat{u}_1 -k^2(\mu_{\rm e} -\mu_{\rm c}+\lambda_{\rm e} )\dd\widehat{u}_1 ,\notag
\vspace{1.2mm}\\
-\omega^2\rho\, \widehat{u}_2&=-k^2\,(\mu_{\rm e} +\mu_{\rm c})\dd \widehat{u}_2 +2\,k\,\mu_{\rm c}\,\widehat{\vartheta }_3,
\vspace{1.2mm}\\
-{\rm i}\,\omega^2\rho\, j\,\mu_{\rm e}\,\tau_{\rm c}^2\,\widehat{ \vartheta }_3&=-{\rm i}\,k^2\,{\mu_{\rm e}\,L_{\rm c}^2}\,(\alpha_1+\alpha_2)\,\widehat{ \vartheta }_3-2\,{\rm i}\,k\,\mu_{\rm c}\,\widehat{u}_2 -4\,{\rm i}\, \mu_{\rm c}\, \widehat{\vartheta }_3, \notag
\end{align}
while $\widehat{u}_3, \widehat{\vartheta}_1$ and $\widehat{\vartheta}_2$ have to satisfy the system of linear equations

\begin{align}\label{algPDE2}
-\omega^2\rho\, \widehat{u}_3&=-k^2\,(\mu_{\rm e} +\mu_{\rm c})\dd \widehat{u}_3 -2\,k\,\mu_{\rm c}\,\widehat{\vartheta }_2 ,\notag
\vspace{1.2mm}\\
-{\rm i}\,\omega^2\rho\, j\,\mu_{\rm e}\,\tau_{\rm c}^2\,\widehat{ \vartheta }_1&=-{\rm i}\,k^2\,{\mu_{\rm e}\,L_{\rm c}^2}\,(\alpha_1+\alpha_2)\,\widehat{ \vartheta }_1-{\rm i}\,k^2\,{\mu_{\rm e}\,L_{\rm c}^2}\,(\alpha_1-\alpha_2+\alpha_3)\widehat{\vartheta }_1 -4\,{\rm i}\, \mu_{\rm c}\, \widehat{\vartheta }_1,\vspace{1.2mm}\\
-{\rm i}\,\omega^2\rho\, j\,\mu_{\rm e}\,\tau_{\rm c}^2\,\widehat{ \vartheta }_2&=-{\rm i}\,k^2\,{\mu_{\rm e}\,L_{\rm c}^2}\,(\alpha_1+\alpha_2)\,\widehat{ \vartheta }_2+2\,{\rm i}\,k\,\mu_{\rm c}\,\widehat{u}_3 -4\,{\rm i}\, \mu_{\rm c}\, \widehat{\vartheta }_2. \notag
\end{align}

Hence, there exist real plane wave if for every  wave number $k>0$ the following  systems of  equations \eqref{PDE} admit  non-trivial solutions:   
\begin{align}\label{fpm}
[\mathbf{\mathbf{Q}}_1(e_1,k)-\omega^2\widehat{\id}]\, w&=0 \qquad 
w=\begin{footnotesize}\begin{pmatrix}
\widehat{u}_1,\widehat{u}_2,\widehat{ \vartheta }_3
\end{pmatrix}\end{footnotesize}^T, \\
[\mathbf{\mathbf{Q}}_2(e_1,k)-\omega^2\widetilde{\id} ]\, w&=0 \qquad 
w=\begin{footnotesize}\begin{pmatrix}
\widehat{u}_3,\widehat{ \vartheta }_1,\widehat{ \vartheta }_2
\end{pmatrix}\end{footnotesize}^T\notag
\end{align}
 only for real frequencies $\omega\in \mathbb{R}$,
where
\begin{align}
\mathbf{\mathbf{Q}}_1(e_1,k)&=\begin{footnotesize}\begin{footnotesize}\begin{pmatrix}
k^2(2\, \mu_{\rm e} +\lambda_{\rm e} )& 0& 0\vspace{2mm}\\
0& k^2(\mu_{\rm e} +\mu_{\rm c})& -2\, k\, \mu_{\rm c}\vspace{2mm}\\
0& -2\, k\, \mu_{\rm c}&k^2\, {\mu_{\rm e}\,L_{\rm c}^2}\,(\alpha_1 +\alpha_2)+4\, \mu_{\rm c}
\end{pmatrix}\end{footnotesize}\end{footnotesize},\\
\mathbf{\mathbf{Q}}_2(e_1,k)&=\begin{footnotesize}\begin{footnotesize}\begin{pmatrix}
k^2(\mu_{\rm e} +\mu_{\rm c})& 0& 2\, k\, \mu_{\rm c}\vspace{2mm}\\
0&k^2\,{\mu_{\rm e}\,L_{\rm c}^2}\,(2\,\alpha_1+\alpha_3)+4 \mu_{\rm c}& 0\vspace{2mm}\\
2\, k\, \mu_{\rm c}& 0&k^2\, {\mu_{\rm e}\,L_{\rm c}^2}\,(\alpha_1 +\alpha_2)+4\, \mu_{\rm c}
\end{pmatrix}\end{footnotesize}\end{footnotesize},\\
\widehat{\id}&=\begin{footnotesize}\begin{pmatrix} \rho&0	&0
\\
0&\rho 
&0
\\0
& 0 & \rho\,j\,\mu_{\rm e}\,\tau_{\rm c}^2\, \end{pmatrix}\end{footnotesize}, \qquad \qquad 
\widetilde{\id} =\begin{footnotesize}\begin{pmatrix} \rho&0	&0
\\
0&\rho\,j\,\mu_{\rm e}\,\tau_{\rm c}^2
&0
\\0
& 0 & \rho\,j\,\mu_{\rm e}\,\tau_{\rm c}^2\end{pmatrix}\end{footnotesize}.
\end{align}

In this form, since $\widehat{\id} \neq \id$ and $\widetilde{\id}  \neq \id$, these are not eigenvalue problems. However,  the system \eqref{fpm} is equivalent to
\begin{align}\label{vz2}
\left[\widehat{\id}^{-1/2}\mathbf{\mathbf{Q}}_1(e_1,k)\widehat{\id}^{-1/2}-\omega^2\id\right]\, d=0,\qquad\qquad  d=\widehat{\id}^{1/2}\begin{footnotesize}\begin{pmatrix}
\widehat{u}_1\\\widehat{u}_2\\\widehat{ \vartheta }_3
\end{pmatrix}\end{footnotesize}.
\end{align}
Hence, the system \eqref{fpm} is equivalent to the eigenvalue problem
\begin{align}\label{q1}
\left[\underbrace{\begin{footnotesize}\begin{pmatrix}
k^2\frac{2\, \mu_{\rm e} +\lambda_{\rm e} }{\rho}& 0& 0\vspace{2mm}\\
0& k^2\frac{\mu_{\rm e} +\mu_{\rm c}}{\rho}& -2\, k\, \frac{\mu_{\rm c}}{\rho\, \sqrt{j\,\mu_{\rm e}\,\tau_{\rm c}^2\,}}\vspace{2mm}\\
0& -2\, k\, \frac{\mu_{\rm c}}{\rho\, \sqrt{j\,\mu_{\rm e}\,\tau_{\rm c}^2\,}}\ \ &\ \ k^2\, {\mu_{\rm e}\,L_{\rm c}^2}\,\frac{\alpha_1 +\alpha_2}{\rho\, j\,\mu_{\rm e}\,\tau_{\rm c}^2\,}+4\, \frac{\mu_{\rm c}}{\rho \, j\,\mu_{\rm e}\,\tau_{\rm c}^2\,}
\end{pmatrix}\end{footnotesize}}_{:=\,\widetilde{\mathbf{\mathbf{Q}}}_1(e_1,k)}-\omega^2\id\right]\, d=0.
\end{align}

Thus, asking that for all $k>0$  the system \eqref{fpm}$_1$
admits non trivial solutions $w=\begin{footnotesize}\begin{pmatrix}
\widehat{u}_1,\widehat{u}_2,\widehat{ \vartheta }_3
\end{pmatrix}\end{footnotesize}^T\neq 0$  only for  real positive values $\omega^2$   is equivalent to the positive definiteness of $\widetilde{\mathbf{\mathbf{Q}}}_1(e_1,k)$ for all $k>0$. Using the Sylvester criterium this means the following set of conditions
\begin{align}
&k^2\,(2\, \mu_{\rm e} +\lambda_{\rm e})>0, \qquad k^2(\mu_{\rm e} +\mu_{\rm c})>0, \\ &\left(\det \widetilde{\mathbf{\mathbf{Q}}}_1(e_1,k)>0\qquad \Leftrightarrow\qquad  {\mu_{\rm e}\,L_{\rm c}^2}\,(\alpha _1+\alpha _2 )\, k^4\, \left(\mu_{\rm c}+\mu _{\rm e}\right)+k^2\,4\, \mu_{\rm c} \,\mu _{\rm e}>0\right) \qquad\quad  \forall \ k>0, \notag
\end{align}
 but it also implies that all the diagonal elements are positive, i.e., we also have that 
\begin{align}
k^2\, {\mu_{\rm e}\,L_{\rm c}^2}\,(\alpha_1 +\alpha_2)+4\, {\mu_{\rm c}}>0 \qquad \qquad \forall \ \ k>0.
\end{align}
It is now easy to remark that the positive definiteness of $\widetilde{\mathbf{\mathbf{Q}}}_1(e_1,k)$ for all $k>0$ is equivalent to the non-redundant set of inequalities 
\begin{align}\label{rendc}
2\, \mu_{\rm e} +\lambda_{\rm e}>0, \qquad \qquad\mu_{\rm e}>0, \qquad\qquad  \mu_{\rm c}>0, \qquad\qquad \alpha _1+\alpha _2 >0.
\end{align}

In a similar way, we find that  for all $k>0$  the system \eqref{fpm}$_2$
admits non trivial solutions $w=\begin{footnotesize}\begin{pmatrix}
\widehat{u}_3,\widehat{ \vartheta }_1,\widehat{ \vartheta }_2
\end{pmatrix}\end{footnotesize}^T\neq 0$  only for  real positive values $\omega^2$  if and only if the following eigenvalue problem admits only real solutions
\begin{align}
&\left[\underbrace{\begin{footnotesize}\begin{pmatrix}
	k^2\frac{ \mu_{\rm e} +\mu_{\rm c} }{\rho}& 0& -2\, k\, \frac{\mu_{\rm c}}{\rho\, \sqrt{j\,\mu_{\rm e}\,\tau_{\rm c}^2\,}}\vspace{2mm}\\
	0& \ k^2\, {\mu_{\rm e}\,L_{\rm c}^2}\,\frac{2\,\alpha_1 +\alpha_3}{\rho\, j\,\mu_{\rm e}\,\tau_{\rm c}^2\,}+4\, \frac{\mu_{\rm c}}{\rho \, j\,\mu_{\rm e}\,\tau_{\rm c}^2}& 0\vspace{2mm}\\
	-2\, k\, \frac{\mu_{\rm c}}{\rho\, \sqrt{j\,\mu_{\rm e}\,\tau_{\rm c}^2\,}}& 0&\ \ k^2\, {\mu_{\rm e}\,L_{\rm c}^2}\,\frac{\alpha_1 +\alpha_2}{\rho\, j\,\mu_{\rm e}\,\tau_{\rm c}^2\,}+4\, \frac{\mu_{\rm c}}{\rho \, j\,\mu_{\rm e}\,\tau_{\rm c}^2\,}
	\end{pmatrix}\end{footnotesize}}_{:=\,\widetilde{\mathbf{\mathbf{Q}}}_2(e_1,k)\equiv\, \widetilde{\id}^{-1/2}\mathbf{\mathbf{Q}}_2(e_1,k)\,\widetilde{\id}^{-1/2}}-\omega^2\id\right]\, f=0, 
\end{align}
with $f=\widetilde{\id}^{1/2}\begin{footnotesize}\begin{pmatrix}
\widehat{u}_3\\\widehat{ \vartheta }_1\\\widehat{ \vartheta }_2
\end{pmatrix}\end{footnotesize}$, i.e., if and only if for all $k>0$
\begin{align}
&{k^2 \left(\mu_{\rm c}+\mu _{\rm e}\right)}>0,\qquad\qquad\qquad{4\, \mu_{\rm c}}+{\left(2\, \alpha _1+\alpha _3\right) k^2 L_{\rm c}^2}>0, \notag\\& \bigg(\det \widetilde{\mathbf{\mathbf{Q}}}_2(e_1,k)>0\qquad \Leftrightarrow\qquad k^6 \left(\alpha _1+\alpha _2\right) \left(2\, \alpha _1+\alpha _3\right) L_{\rm c}^4 \mu _{\rm e}\left(\mu _{\rm e}+\mu_{\rm c}\right) \\& \qquad \qquad\qquad \qquad\qquad \qquad\qquad  +4\,k^4 L_{\rm c}^2\mu_{\rm c}\left[ \left(\alpha _1+\alpha _2\right) \mu_{\rm c}  \left(\mu _{\rm e}+ \mu_{\rm c}\right)+ \left(2\, \alpha _1+\alpha _3\right)  \mu _{\rm e}\right]+16\, k^2 \mu_{\rm c}^2>0\bigg)\notag, 
\end{align}
but it also implies that all the diagonal elements are positive, i.e., we also have that 
\begin{align}
 k^2\, {\mu_{\rm e}\,L_{\rm c}^2}\,(\alpha_1 +\alpha_2)+4\, \mu_{\rm c}>0.
\end{align}
All together this shows that the positive definiteness of $\widetilde{\mathbf{\mathbf{Q}}}_2(e_1,k)$  is equivalent  to
\begin{align}\label{rw2}
 \mu_{\rm e}+  \mu_{\rm c}>0, \qquad \mu_{\rm c}>0,\qquad 2\, \alpha _1+\alpha _3 >0, \qquad\qquad \mu_{\rm e}(\alpha _1+\alpha _2) >0, \qquad \mu_{\rm e}>0. \qquad 
\end{align}
Therefore, the non-redundant inequalities from \eqref{rw2} are
\begin{align}\label{rw3}
\mu_{\rm e}>0, \qquad\qquad  \mu_{\rm c}>0,  \qquad\qquad 2\, \alpha _1+\alpha _3 >0, \qquad\qquad (\alpha _1+\alpha _2) >0.
\end{align}

The analysis presented in this subsection is similar to that used in \cite{neff2017real}. Then, for $k>0$ and due to the isotropy, extrapolating to all directions of propagation, we have
\begin{proposition}\label{proprealw} The necessary and sufficient conditions for existence of  real planar waves in any direction ${\xi}\in \mathbb{R}^3$, $\xi\neq 0$, in the framework of the linear isotropic elastic Cosserat theory  are
\begin{align}\label{d11}
2\,\mu_{\rm e} +\lambda_{\rm e} >0,\qquad \qquad \mu_{\rm e} >0,\qquad \qquad \mu_{\rm c} >0,\qquad \qquad   \alpha_1+\alpha_2>0, \qquad \qquad 2\,\alpha_1+\alpha_3>0.
\end{align}
\end{proposition}
\begin{proof}
The proof is given by the previous calculations.
\end{proof}
\begin{remark}\label{iterpretation1}
{The conditions \eqref{d11} from Proposition \ref{proprealw} have also some direct interpretations, i.e.,
	\begin{itemize}\item[i)] the first implication of the set of conditions
	\eqref{d11} means that all these waves (compressional/shear-rotational  waves, acoustic/optical branch) are real;
	Once this aspect  is clarified, we can treat and interpret further the propagation of plane waves;
	\item[ii)] under conditions 	\eqref{d11} all branches of waves are real for the entire range $[0,\infty)$ of the frequency;
	\item[iii)] we  can also see directly from  the first condition that the translational compressional wave
	is real;
	\item[iv)]  the second means that the acoustic branch of shear–rotational
	wave is real at low frequencies and together with the third means that  the acoustic branch of shear–rotational
	wave is real at high frequencies;
	\item[v)]  the fourth implies that the optical branch of the
	shear–rotational wave is real at high frequencies;
	\item[vi)]  the third one means that the
	optical branch of the shear–rotational wave at high frequencies has a larger
	velocity than the acoustic branch of the same wave at low frequencies (if they
	both exist, which is the case due to other conditions); 
	\item[vii)] 	the fifth one expresses directly that   the compressional rotational
	wave at high frequencies is real.
\end{itemize}}
\end{remark}
\begin{remark}\label{remarkLH}
\begin{itemize}
	\item[]
	\item[i)]
In  linear isotropic classical elasticity, the necessary and sufficient conditions for existence of  real planar waves in any direction ${\xi}\in \mathbb{R}^3$, $\xi\neq 0$ are
\begin{align}\label{d1ce}
2\,\mu_{\rm e} +\lambda_{\rm e} >0,\qquad \qquad \mu_{\rm e} >0,
\end{align}
and they are equivalent to the strong ellipticity conditions (Legendre-Hadamard ellipticity).

\item[ii)] The necessary and sufficient conditions for existence of a real planar wave are slightly different compared to  the    strong ellipticity conditions  (Legendre–Hadamard ellipticity) \eqref{d11} for the Cosserat (micropolar) model investigated in \cite{eremeyev2007constitutive,Eremeyev4,shirani2020legendre}
and which are  connected to acceleration waves.   In our notation,
the strong
ellipticity condition for  Cosserat media is represented by
the inequality \cite{eremeyev2007constitutive,Eremeyev4,shirani2020legendre}
\begin{align}
\frac{d\, }{d\, \tau}W(\mathbf{e} + \tau \, \xi \otimes \eta, \mathbf{\mathfrak{K}} + \tau \, \zeta\otimes \eta)\bigg|_{\tau=0}>0 \qquad \forall\  \eta, \xi, \zeta \in \mathbb{R}^3, \ \ \lVert \eta\rVert=\lVert \xi\rVert=\lVert \zeta\rVert=1.
\end{align}
and it is satisfied if and only if
\begin{align}\label{dse1}
2\,\mu_{\rm e} +\lambda_{\rm e} >0,\qquad \quad \mu_{\rm e} +\mu_{\rm c} >0,\qquad \quad   \alpha_1+\alpha_2>0, \qquad \quad 2\,\alpha_1+\alpha_3>0.
\end{align}
The explicit calculations in our notations are made in \cite{shirani2020legendre}.  {The absence of a coupling between
$\mathbf{e}$ and $\mathbf{\mathfrak{K}}$ in the strain energy  leads to a simplification of the calculations.} 
\item[iii)]{The conditions   \eqref{dse1} (strong ellipticity conditions, Legendre-Hadamard ellipticity, the positive definiteness of the acoustic tensor) imply  the existence of the real translational compressional
waves in the entire range of real frequencies, of the real shear rotational waves (both branches) at high frequencies, and of real
rotational compressional wave at high frequencies, but at lower frequencies the
latter waves may not be real since \eqref{dse1} does not guaranty that $\mathfrak{c}_t=\frac{\mu_{\rm e}}{\rho}$ is real. To the contrary, the conditions \eqref{d11} imply that all these branches and types of plane wave are real, i.e., the group/phase velocities are real on the entire range of possible frequencies.}
\end{itemize} \end{remark}

The strong ellipticity conditions \eqref{dse1} are weaker than the conditions \eqref{d11} in the sense that they are implied by the necessary and sufficient conditions for existence of a real planar wave {(i.e., they imply the strong ellipticity and, therefore, the considered PDEs system is not unstable)}  but not vice versa. However, the strong ellipticity conditions \eqref{dse1} are not sufficient for the application of our solution method,   and we believe that they are also not suitable for any approach regarding the propagation of  Rayleigh waves in Cosserat solids. 

In the end of this section we mention that Eringen \cite[pages 149-151]{Eringen99} affirmed that one has to impose (in addition) that 
\begin{align}\label{Eringencond}
\frac{L_{\rm c}^2}{\tau_{\rm c}^2}\frac{ (\alpha_1+\alpha_2)}{\rho\,j}>\frac{\mu_{\rm e}+\mu_{\rm c}}{\rho},
\end{align}
in order that there   exist four real $\omega$ which lead to plane waves solutions (real plane waves). This seems to be only a consequence of the representation formula and the method he used to construct the plane wave solution and seems to be not a necessary condition for real plane waves propagation. {Condition \eqref{Eringencond} required by Eringen means that the shearrotational wave at
	high frequencies is faster for the acoustic branch than for the optical branch.} Eringen \cite[pages 151]{Eringen99} also claimed that this condition is in accordance to  the lattice dynamical calculations but no further explanations are given.

Nevertheless,  this Eringen-type condition \eqref{Eringencond} seems to have sense  when we are going back to the classical linear elasticity model by considering $\mu_{\rm c}\to 0$ and large values of {$\frac{L_{\rm c}}{\tau_{\rm c}}$ (${L_{\rm c}}\to \infty$ or $\tau_{\rm c}\to 0$)}, since for $\xi=e_1$ the eigenvalue problems characterising the possible real values of $\omega$ become
\begin{align}
\left[\begin{footnotesize}\begin{pmatrix}
	k^2\frac{2\, \mu_{\rm e} +\lambda_{\rm e} }{\rho}& 0& 0\vspace{2mm}\\
	0& k^2\frac{\mu_{\rm e}}{\rho}& 0\vspace{2mm}\\
	0& 0 &\ \ k^2\, \frac{L_{\rm c}^2}{\tau_{\rm c}^2}\,\frac{\alpha_1 +\alpha_2}{\rho\, j}
	\end{pmatrix}\end{footnotesize}-\omega^2\id\right]\, d=0, \qquad\quad \quad d=\widehat{\id}^{1/2}\begin{footnotesize}\begin{pmatrix}
\widehat{u}_1\\\widehat{u}_2\\\widehat{ \vartheta }_3
\end{pmatrix}\end{footnotesize}
\end{align}
and
\begin{align}
\left[\begin{footnotesize}\begin{pmatrix}
	k^2\frac{ \mu_{\rm e}  }{\rho}& 0& 0\vspace{2mm}\\
	0& \ k^2\, {\mu_{\rm e}}\,\frac{2\,\alpha_1 +\alpha_3}{\rho\, j}& 0\vspace{2mm}\\
	0& 0&\ \ k^2\, \frac{L_{\rm c}^2}{\tau_{\rm c}^2}\,\frac{\alpha_1 +\alpha_2}{\rho\, j}
	\end{pmatrix}\end{footnotesize}-\omega^2\id\right]\, f=0,\quad \quad\quad \quad f=\widetilde{\id}^{-1/2}\begin{footnotesize}\begin{pmatrix}
\widehat{u}_3\\\widehat{ \vartheta }_1\\\widehat{ \vartheta }_2
\end{pmatrix}\end{footnotesize}.
\end{align}
In classical linear elasticity, it is natural to impose that  plane waves propagate only with speeds $\sqrt{\frac{2\, \mu_{\rm e} +\lambda_{\rm e} }{\rho}}$ and $\sqrt{\frac{ \mu_{\rm e}}{\rho}}$ and no other ``exotic'' plane wave arises. Indeed, for   $\mu_{\rm c}\to 0$ and large values  of  $\frac{L_{\rm c}^2}{\tau_{\rm c}^2}$, \   by imposing
\begin{align}\label{addrectr}
 &\frac{L_{\rm c}^2}{\tau_{\rm c}^2}\,\min\left\{\frac{2\,\alpha_1 +\alpha_3}{\rho\, j},\frac{\alpha_1 +\alpha_2}{\rho\, j}\right\}> \max\left\{\frac{2\, \mu_{\rm e} +\lambda_{\rm e} }{\rho},\frac{\mu_{\rm e}}{\rho}\right\},
\end{align}
the unique propagating speeds which lead to non-vanishing displacements for plane waves are only those from classical elasticity. {Condition \eqref{addrectr} means that both shear–rotational acoustic branch and rotational compressional wave at high frequencies are faster than the shear–
	rotational acoustic branch at low frequencies and the translational compressional wave.} More precisely, under this additional restriction upon the curvature coefficients, in the limit case $\mu_{\rm c}\to 0$, we determine the amplitudes  
\begin{align}\label{algPDE11}
 &{\rm for}\qquad  \omega^2=k^2\,\frac{2\, \mu_{\rm e} +\lambda_{\rm e} }{\rho}: & \widehat{u}\times e_1=0, &\qquad\qquad \widehat{\vartheta }=0,\notag
\\
&{\rm for}\qquad  \omega^2=k^2\,\frac{ \mu_{\rm e}  }{\rho}:& \langle \widehat{u}, e_1\rangle=0, &\qquad \qquad\widehat{\vartheta }=0,\notag
\\
& {\rm for}\qquad  \omega^2=k^2\,\frac{L_{\rm c}^2}{\tau_{\rm c}^2}\,\,\frac{2\,\alpha_1 +\alpha_3}{\rho\, j}:& \widehat{u}=0,& \qquad\qquad \widehat{\vartheta }\times e_1=0,
\\
&{\rm for}\qquad  \omega^2=k^2\,\frac{L_{\rm c}^2}{\tau_{\rm c}^2}\,\frac{\alpha_1 +\alpha_2}{\rho\, j}:&\widehat{u}=0, & \qquad\qquad \langle\widehat{\vartheta }, e_1\rangle=0.\notag
\end{align}
This means that real plane waves for which $\widehat{u}\neq 0$ (i.e., only what real plane waves in  classical elasticity means) are possible only  for $\omega^2=k^2\,\sqrt{\frac{2\, \mu_{\rm e} +\lambda_{\rm e} }{\rho}}$ and
 $\omega^2=k^2\,\sqrt{\frac{ \mu_{\rm e}}{\rho}}$, situation in which the microrotation vector $\widehat{\vartheta }$ vanishes. Moreover, the extremely high values of the frequency which lead to non-vanishing $\widehat{\vartheta }$ are beyond the framework of classical elasticity and belong rather to quantum mechanics or relativistic mechanics.
 
 It is important to note that when   studying  plane waves, we have to limit  the plausible domains of the frequency  (i.e.,  the speed is less than the  speed of light) in the framework of classical mechanics, for seismic waves this is already done once the ansatz is chosen,  since only  subsonic speeds are admissible. We will explain this aspect in more details in the following sections.

\subsection{The setup for the propagation of Rayleigh waves}\label{setRw}

 In the framework of the Rayleigh wave, we consider the region $\Omega$ to be the half space $$\Sigma:=\{(x_1,x_2,x_3)\,|\,x_1,x_3\in \mathbb{R},\,\, x_2\geq0\}.$$ The boundary of the homogeneous and isotropic half-space  is free of surface traction, i.e.,
 \begin{align}\label{03}
 \boldsymbol{\sigma}.\, n=0,\qquad \qquad \mathbf{m}.\, n=0 \qquad \qquad\textrm{for}\quad x_2=0.
 \end{align}
 \begin{figure}[h!]
 	\centering
 	\includegraphics[width=8.5cm]{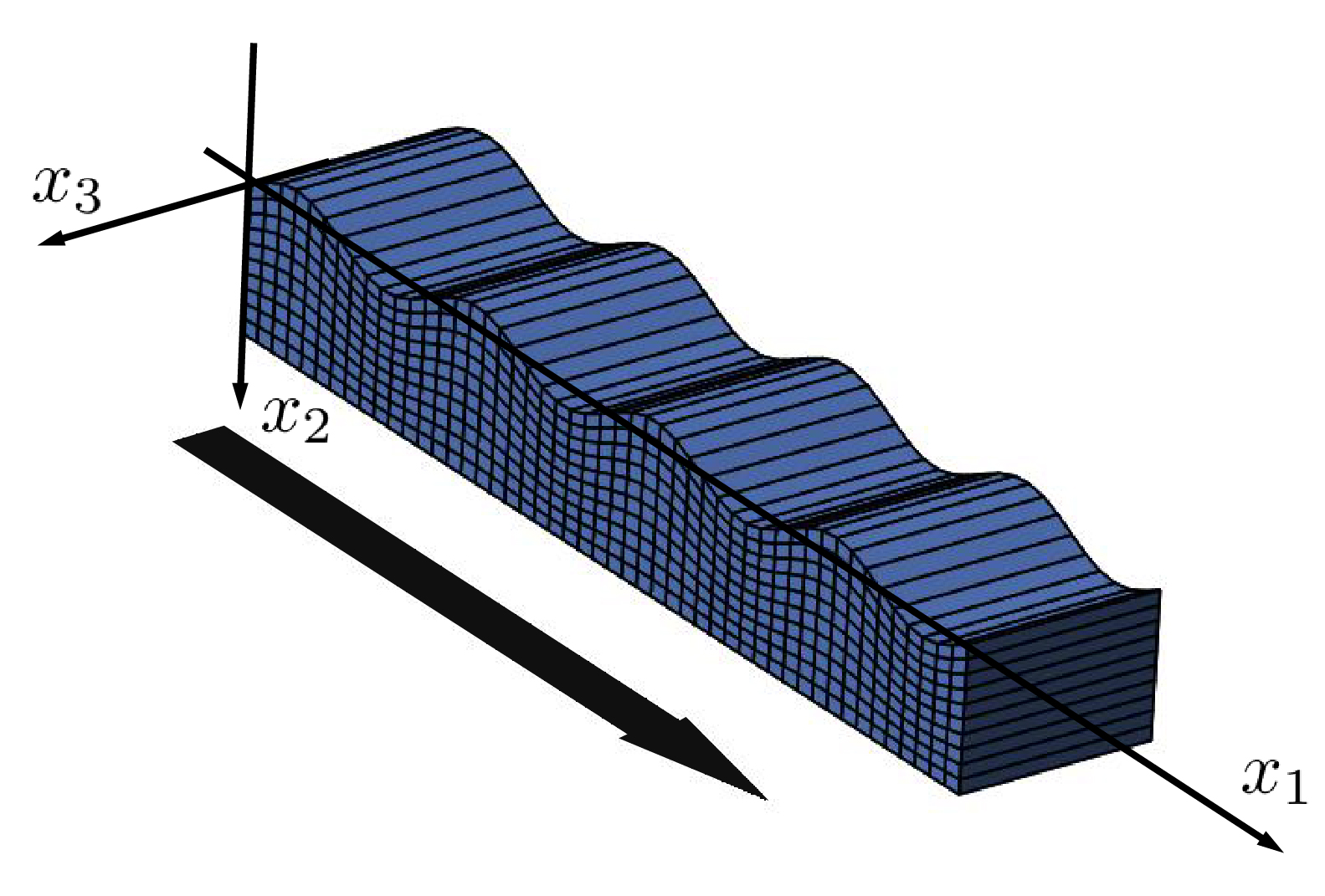}
 	\caption{\footnotesize The displacement of a Rayleigh wave, no traction on the upper free surface.}
 	\label{rayfig}
 \end{figure} 

 In addition, the solution has to satisfy the following   decay condition
 \begin{align}\label{04}
 \lim_{x_2 \rightarrow \infty}\{u_1,u_2,{\vartheta_3},\sigma_{12},\sigma_{21},\sigma_{22},m_{23}\} (x_1,x_2,t)=0 \qquad \quad \forall\, x_1\in\mathbb{R}, \quad \forall \, t\in[0,\infty).
 \end{align}

In isotropic solids and {in the context of Rayleigh wave propagation}, the surface particles move in the planes normal to the surface $x_2=0$ and parallel to the direction of propagation  $e_1=(1,0,0)^T$, see Figure \ref{rayfig}. Accordingly to these characteristics of the seismic waves, we consider the following  plain strain ansatz as a first step in our process of construction of the solution
 \begin{align}
 u(x_1,x_2,t)&=\begin{footnotesize}\begin{pmatrix} u_1(x_1,x_2,t)
 	\\
 	u_2(x_1,x_2,t)
 	\\0
 	\end{pmatrix}\end{footnotesize},\qquad \qquad \qquad\qquad \qquad{\rm D}u(x_1,x_2,t)=
 	\begin{footnotesize}\begin{pmatrix}
 	\frac{\partial\, u_{1}}{\partial\,x_1}(x_1,x_2,t)&\frac{\partial\, u_{2}}{\partial\,x_1}(x_1,x_2,t)&0\\
 	\frac{\partial\, u_{1}}{\partial\,x_2}(x_1,x_2,t)&\frac{\partial\, u_{2}}{\partial\,x_2}(x_1,x_2,t)&0\\
 	0&0&0
 	\end{pmatrix}\end{footnotesize}, \\ \mathbf{A}(x_1,x_2,t)&=
 	\begin{footnotesize}\begin{pmatrix}
 	0&-\vartheta_3(x_1,x_2,t)&0\\
 	\vartheta_3(x_1,x_2,t)&0&0\\
 	0&0&0
 	\end{pmatrix}\end{footnotesize},\qquad \quad
 	\vartheta (x_1,x_2,x_3,t)=\axl \,\mathbf{A}=\begin{footnotesize}\begin{pmatrix} 0
 	\\
 	0
 	\\ \vartheta   _3(x_1,x_2,t)
 		\end{pmatrix}\end{footnotesize}.\notag
 \end{align}
{Besides the Rayleigh waves, in the full isotropic Cosserat medium there is also another
	(transversal, ``Love-like'') surface wave decribed by $(u_1,u_2,\theta_3)^T= 0$ and $(u_3,\theta_1,\theta_2)^T\neq 0$, see \cite{Kulesh06},
	while in the reduced Cosserat medium ($L_{\rm c}\to 0$) there exist non-propagating transversal oscillations of the same kind, see \cite{kulesh2009problem}. However, this is not the purpose of the present work and the completed proof of the existence of these ``Love-like" surface waves will be considered in the future.}
 
Corresponding to our ansatz, we deduce 
the following form of the stress tensor 
 \begin{align}
 \boldsymbol{\sigma}=
\begin{footnotesize}\begin{footnotesize}\begin{pmatrix}
 (2\,\mu_{\rm e}  +\lambda_{\rm e} ) \,\frac{\partial\, u_{1}}{\partial\,x_1}+\lambda_{\rm e} \, \frac{\partial \,u_{2}}{\partial \,x_2}
 &
 (\mu_{\rm e} +\mu_{c})\,\frac{\partial\, u_{1}}{\partial \,x_2}+2\,\mu_{\rm c} \, \vartheta _3+(\mu_{\rm e} -\mu_{c})\,\frac{\partial\, u_{2}}{\partial\,x_1}
 &
 0
\vspace{2mm} \\
 (\mu_{\rm e} +\mu_{c})\,\frac{\partial\, u_{2}}{\partial \,x_1}-2\,\mu_{\rm c} \, \vartheta _3+(\mu_{\rm e} -\mu_{c})\,\frac{\partial\, u_{1}}{\partial \,x_2}
 &
 (2\,\mu_{\rm e}  +\lambda_{\rm e} ) \,u_{2,x_2}+\lambda_{\rm e} \, \frac{\partial\, u_{1}}{\partial \,x_1}
 &
 0\vspace{2mm} \\
 0
 &
 0
 &	
 \lambda_{\rm e} \, (\frac{\partial\, u_{1}}{\partial \,x_1}+\frac{\partial\, u_{2}}{\partial \,x_2})
 \end{pmatrix}\end{footnotesize}\end{footnotesize},
 \end{align}
 and of the couple stress tensor 
 \begin{align}
\mathbf{m}=
\begin{footnotesize}\begin{footnotesize}\begin{pmatrix}
 0
 &
 0
 &
{\mu_{\rm e}\,L_{\rm c}^2}\,({\alpha_1-\alpha_2})\,\frac{\partial \,\vartheta _{3}}{\partial \,x_1}\vspace{2mm}
 \\
 0
 &
 0
 &
{\mu_{\rm e}\,L_{\rm c}^2}\, ({\alpha_1-\alpha_2})\,\frac{\partial \,\vartheta _{3}}{\partial \,x_2}
 \vspace{2mm}\\
 {\mu_{\rm e}\,L_{\rm c}^2}\,({\alpha_1+\alpha_2})\,\frac{\partial \,\vartheta _{3}}{\partial \,x_1}
 &
{\mu_{\rm e}\,L_{\rm c}^2}\, ({\alpha_1+\alpha_2})\,\frac{\partial \,\vartheta _{3}}{\partial \,x_2}
 &	
0
 \end{pmatrix}\end{footnotesize}\end{footnotesize},
 \end{align}
 while the equation of motion are reduced to 
 \begin{align}\label{49}
 \rho\,\frac{\partial^2 \, {u}_{1}}{\partial\,t^2}&=\frac{\partial \,\sigma_{11}}{\partial \,x_1}+\frac{\partial \,\sigma_{12}}{\partial \,x_2},\notag\\
 \rho\, \frac{\partial^2 \,{u}_{2}}{\partial\,t^2}&= \frac{\partial \,\sigma_{22}}{\partial \,x_1}+ \frac{\partial \,\sigma_{22}}{\partial \,x_2},\\
 \rho\,j\,\mu_{\rm e}\,\tau_{\rm c}^2\, \,\frac{\partial^2 \,{\vartheta _{3}}}{\partial\,t^2}&= \frac{\partial \,m_{31}}{\partial \,x_1}+ \frac{\partial \,m_{32}}{\partial \,x_2}+2\,\mu_{\rm c} \,   \frac{\partial \,u_{2}}{\partial \,x_1}-2\,\mu_{\rm c} \,    \frac{\partial \,u_{1}}{\partial \,x_2}-4\mu_{\rm c} \, \vartheta _3,\notag
 \end{align}
 subjected to the aforementioned boundary conditions  \eqref{03}, which turn out to be
 \begin{align}\label{50}
 \sigma_{12}=0,\qquad \sigma_{22}=0, \qquad m_{32}=0 \qquad \text{at}\qquad x_2=0.
 \end{align}
 
 Therefore, the aim of this paper is to give an explicit solution  $(u, \vartheta   )$ of the following  system
 \begin{align}\label{x6}
 \rho\, \frac{\partial^2{u}_{1}}{ \partial \,t^2}&=(2\,\mu_{\rm e}  +\lambda_{\rm e} ) \,\frac{\partial^2 \,u_{1}}{\partial \,x_1^2}+\lambda_{\rm e} \,  \frac{\partial^2 \,u_{2}}{\partial \,x_2\partial\, x_1}+ (\mu_{\rm e} +\mu_{c})\,\frac{\partial^2 \,u_{1}}{\partial \,x_2^2}+2\,\mu_{\rm c} \, \frac{\partial \,\vartheta _{3}}{\partial \,x_2}+(\mu_{\rm e} -\mu_{c})\,\frac{\partial ^2\,u_{2}}{\partial \,x_1\,\partial \,x_2},\notag\\
\rho\, \frac{\partial^2{u}_{2}}{\partial \,t^2}&=(\mu_{\rm e} +\mu_{c})\,\frac{\partial^2 \,u_{2}}{\partial \,x_1^2}-2\,\mu_{\rm c} \, \frac{\partial \,\vartheta _{3}}{\partial \,x_1}+(\mu_{\rm e} -\mu_{c})\,\frac{\partial^2 \,u_{1}}{\partial \,x_2\,\partial\,x_1}+(2\,\mu_{\rm e}  +\lambda_{\rm e} ) \,\frac{\partial^2 \,u_{2}}{\partial \,x_2^2}+\lambda_{\rm e} \,  \frac{\partial^2 \,u_{1}}{\partial \,x_1\,\partial \,x_2},\\
\rho \,j\,\mu_{\rm e}\,\tau_{\rm c}^2\,\frac{\partial^2 \,{{\vartheta }_{3}}}{\partial \,t^2}&={\mu_{\rm e}\,L_{\rm c}^2}\,\gamma \,\frac{\partial^2 \,\vartheta _{3}}{\partial \,x_1^2}+{\mu_{\rm e}\,L_{\rm c}^2}\,\gamma \,\frac{\partial^2 \,\vartheta _{3}}{\partial \,x_2^2}+2\,\mu_{\rm c} \,  \frac{\partial \,u_{2}}{\partial \,x_1}-2\,\mu_{\rm c} \,  \frac{\partial\,u_{1}}{\partial \,x_2}-4\,\mu_{\rm c} \, \vartheta _3,\notag
 \end{align}
 which satisfies the boundary conditions at $x_2=0$
 \begin{align}\label{x17}
 (\mu_{\rm e} +\mu_{c})\,\frac{\partial\,u_{1}}{\partial \,x_2}+2\,\mu_{\rm c} \, \vartheta _3+(\mu_{\rm e} -\mu_{c})\,\frac{\partial\,u_{2}}{\partial \,x_1}=&\,0,\notag\\
 (2\,\mu_{\rm e}  +\lambda_{\rm e} ) \,\frac{\partial\,u_{2}}{\partial \,x_2}+\lambda_{\rm e} \,  \frac{\partial\,u_{1}}{\partial \,x_1}=&\,0,\\
 {\mu_{\rm e}\,L_{\rm c}^2}\,\gamma \,\frac{\partial\,\vartheta _{3}}{\partial \,x_2}=&\,0,\notag
 \end{align}
 where $\gamma={\alpha_1+\alpha_2}$, and which has the asymptotic behaviour \eqref{04}.
 
 {Even if until now we have considered the  propagation of  surface waves with the direction $e_1 =
 (1,0,0)^T$, i.e. some horizontal direction, $x_2$ being the vertical direction orthogonal
 to the surface along which the wave decays, in one point of our method (see Proposition \ref{lemmaGH} and its implications in Subsection \ref{NSE}) we need to consider a general direction of wave propagation $\xi = (\xi_1,\xi_2,0)^T$ and to characterize  where the wave is
 a “real” bulk wave. Therefore, it} is useful to {know} for which conditions on the constitutive parameters (for every  wave number $k>0$ and  in the  direction $\xi=(\xi_1,\xi_2,0)^T$ with $\lVert{\xi}\rVert^2=1$) the system of partial differential equations \eqref{PDE} admits a non trivial solution in the form
\begin{align}\label{rwr}
u(x_1,x_2,t)&=\begin{footnotesize}\begin{pmatrix}\widehat{u}_1\\\widehat{u}_2\\0\end{pmatrix}\end{footnotesize}
\, e^{{\rm i}\, \left(k\langle {\xi},\, x\rangle_{\mathbb{R}^3}-\,\omega \,t\right)}\,,\qquad\qquad  \vartheta (x_1,x_2,t)=\begin{footnotesize}{\rm i}\,\begin{pmatrix}0\\0\\\widehat{\vartheta }_3\end{pmatrix}\end{footnotesize} \, e^{{\rm i}\, \left(k\langle {\xi},\, x\rangle_{\mathbb{R}^3}-\,\omega \,t\right)},\\&
(\widehat{u}_1,\widehat{u}_2, \widehat{ \vartheta_3   })^T\in\mathbb{C}^{3}, \quad (\widehat{u}_1,\widehat{u}_2, \widehat{ \vartheta_3   })^T\neq 0\,\notag
\end{align}
only for real positive values $\omega^2$. According to the results given in Subsection \ref{Rpw}, see \eqref{algPDE1}, \eqref{q1} and \eqref{rendc}, we have the following result
\begin{proposition}\label{proprealw2} Let   ${\xi}\in \mathbb{R}^3$, $\xi\neq 0$ be any direction of the form $\xi=(\xi_1,\xi_2, 0)^T$. The necessary and sufficient conditions for existence of   a non trivial solution of the system of partial differential equations \eqref{PDE}  of the form given by \eqref{rwr} are
	\begin{align}\label{d12}
	2\,\mu_{\rm e} +\lambda_{\rm e} >0,\qquad \qquad \mu_{\rm e} >0,\qquad \qquad \mu_{\rm c} >0,\qquad \qquad   \alpha_1+\alpha_2>0.
	\end{align}
\end{proposition}
These  restriction \eqref{d12} are the restrictions on the  constitutive parameters that we will impose for the rest of the paper. {For an interpretation of the conditions \eqref{d12}, see Remark \ref{iterpretation1}.}. These conditions do not involve the constitutive parameter $\alpha_3$ and do not imply the existence of real waves in any directions. {However,  for a given direction ${\xi}\in \mathbb{R}^3$, $\xi\neq 0$ of the form $\xi=(\xi_1,\xi_2, 0)^T$,  the restrictions \eqref{d12} imply  the existence of only real waves defined by expressions of the form \eqref{rwr}.} 

 \section{The ansatz for the solution and the limiting speed}\label{anzsec}\setcounter{equation}{0}
 We look for a solution of \eqref{x6} and \eqref{x17} having the form\footnote{We take ${\rm i}\,z_3$ since this choice leads us, in the end, only to real matrices.}
 \begin{align}\label{x5}
 \mathcal{U}(x_1,x_2,t)=\begin{footnotesize}\begin{pmatrix}u_1(x_1,x_2,t)
 	\\
 	u_2(x_1,x_2,t)
 	\\\vartheta _3(x_1,x_2,t)
 	\end{pmatrix}\end{footnotesize}={\rm Re}\left[\begin{footnotesize}\begin{pmatrix} z_1(x_2)
 	\\
 	z_2(x_2)
 	\\{\rm i}\,z_3(x_2)
 \end{pmatrix}\end{footnotesize} e^{ {\rm i}\, k\, (  x_1-vt)}\right],
 \end{align}
 where  $v$ is the propagation speed {(the phase  velocity)}. If $z_i$, $i=1,2,3$, are solutions of the following systems
 \begin{align}&
\begin{footnotesize}\begin{pmatrix} \mu_{\rm e} +\mu_{\rm c} &0	&0
 	\\
 	0& 2\,\mu_{\rm e}  +\lambda_{\rm e} 
 	&0
 	\\0
 	& 0 & {\mu_{\rm e}\,L_{\rm c}^2}\,\gamma \end{pmatrix}\end{footnotesize}\begin{footnotesize}\begin{pmatrix}	z_1''(x_2)
 	\\
 	z_2''(x_2)
 	\\\	z_3''(x_2)
 	\end{pmatrix}\end{footnotesize}+ {\rm i}\, \begin{footnotesize}\begin{pmatrix} 0& k\,(\mu_{\rm e} -\mu_{\rm c} +\lambda_{\rm e} )& 2\,\mu_{\rm c}  
 	\\
 	k\,(\mu_{\rm e} -\mu_{\rm c} +\lambda_{\rm e} )&0 
 	&0
 	\\2\,\mu_{\rm c}  
 	& 0 & 0 \end{pmatrix}\end{footnotesize}\begin{footnotesize}\begin{pmatrix} 	z_1'(x_2)
 	\\
 	z_2'(x_2)
 	\\\	z_3'(x_2)
 	\end{pmatrix}\end{footnotesize}\notag\vspace{2mm}\\&\qquad\qquad \ \ -\begin{footnotesize}\begin{pmatrix} k^2\,(2\,\mu_{\rm e}  +\lambda_{\rm e} )-\rho \,k^2v^2&0	&0
 	\\
 	0&k^2\,(\mu_{\rm e} +\mu_{\rm c} )-\rho \,k^2v^2
 	&-2\,\mu_{\rm c} \, k
 	\\0
 	&-2\,\mu_{\rm c} \, k &k^2{\mu_{\rm e}\,L_{\rm c}^2}\,\gamma+ 4\,\mu_{\rm c} -j\,\mu_{\rm e}\,\tau_{\rm c}^2\, k^2v^2 \end{pmatrix}\end{footnotesize}\begin{footnotesize}\begin{pmatrix} 	z_1(x_2)
 	\\
 	z_2(x_2)
 	\\	z_3(x_2)
 	\end{pmatrix}\end{footnotesize}=0,
 \end{align}
 and (from the boundary conditions)
 \begin{align}
 &\qquad \qquad \qquad\begin{footnotesize}\begin{pmatrix} \mu_{\rm e} +\mu_{\rm c} &0	&0
 	\\
 	0& 2\,\mu_{\rm e}  +\lambda_{\rm e}  
 	&0
 	\\0
 	& 0 & {\mu_{\rm e}\,L_{\rm c}^2}\,\gamma \end{pmatrix}\end{footnotesize}\begin{footnotesize}\begin{pmatrix} 	z_1'(0)
 	\\
 	z_2'(0)
 	\\	z_3'(0)
 	\end{pmatrix}\end{footnotesize}+ {\rm i}\, \begin{footnotesize}\begin{pmatrix} 0& k\, (  \mu_{\rm e} -\mu_{\rm c} )& 2\,\mu_{\rm c}  
 	\\
 	k\,\lambda_{\rm e} &0 
 	&0
 	\\0
 	& 0 & 0  \end{pmatrix}\end{footnotesize}\begin{footnotesize}\begin{pmatrix}	z_1(0)
 	\\
 	z_2(0)
 	\\	z_3(0)
 \end{pmatrix}\end{footnotesize}=0,\notag
 \end{align}\normalsize
 where $\cdot '$ denotes the derivative with respect to $x_2$, then $\mathcal{U}$ given by the ansatz \eqref{x5} satisfies \eqref{x6} and \eqref{x17}.
In a more compact notation, the above equations admit the following equivalent form \begin{align}\label{11}
\frac{1}{k^2}\,{\mathbf{T}}\,z''(x_2)+ {\rm i}\,\frac{1}{k}\, ({\mathbf{R}}+{\mathbf{R}}^T)\,z'(x_2)-{\mathbf{Q}}\,z(x_2)+k^2\, v^2 \, \hat\id\,z(x_2)=&\,0,\\
\frac{1}{k^2}\,{\mathbf{T}}\,z'(0)+ {\rm i}\, \frac{1}{k}\,{\mathbf{R}}^T\,z(0)=&\,0,\notag
\end{align} where
     the matrices  ${\mathbf{T}}\,\,,{\mathbf{R}}$ and $ {\mathbf{Q}}$  are defined by
 \begin{align}\label{x47}
 {\mathbf{T}}&=k^2\,\begin{footnotesize}\begin{pmatrix}  \mu_{\rm e} +\mu_{\rm c} &0	&0
 	\\
 	0& 2\,\mu_{\rm e}  +\lambda_{\rm e} 
 	&0
 	\\0
 	& 0 & {\mu_{\rm e}\,L_{\rm c}^2}\,\gamma \end{pmatrix}\end{footnotesize},\qquad\quad  {\mathbf{R}}=k\,\begin{footnotesize}\begin{pmatrix}  0& k\lambda_{\rm e} & 0
 	\\
 	k\, (  \mu_{\rm e} -\mu_{\rm c} )&0 
 	&0
 	\\2\,\mu_{\rm c}  
 	& 0 & 0 \end{pmatrix}\end{footnotesize},\\
 {\mathbf{Q}}&=\begin{footnotesize}\begin{pmatrix} k^2\,(2\,\mu_{\rm e}  +\lambda_{\rm e} )&0	&0
 	\\
 	0&k^2\,(\mu_{\rm e} +\mu_{\rm c} )
 	&-
 	2\,\mu_{\rm c} \, k
 	\\0
 	&-2\,\mu_{\rm c} \, k &k^2{\mu_{\rm e}\,L_{\rm c}^2}\,\gamma+4\mu_{\rm c} \, \end{pmatrix}\end{footnotesize}.
 \end{align}
 The system \eqref{11} has a similar structure to that from classical linear elasticity \cite{fu2002new} but we still have to rewrite it in order to make it  manageable for our analysis.  In this respect, the  solution $z$ of \eqref{11} is equivalent to find a  solution $y$ of 
 \begin{align}\label{11t}
 \frac{1}{k^2}\,\widehat{\id}^{-1/2}\,{\mathbf{T}}\,\widehat{\id}^{-1/2}\,y''(x_2)+ {\rm i}\,\frac{1}{k}\, \widehat{\id}^{-1/2}\,({\mathbf{R}}+{\mathbf{R}}^T)\widehat{\id}^{-1/2}\,y'(x_2) -\widehat{\id}^{-1/2}\,{\mathbf{Q}}\,\widehat{\id}^{-1/2}\,y(x_2)+k^2\, v^2 \, \id\,y(x_2)&=\,0,\\
  \frac{1}{k^2}\,\widehat{\id}^{-1/2}\,{\mathbf{T}}\,\widehat{\id}^{-1/2}\,y'(0)+ {\rm i}\, \frac{1}{k}\,\widehat{\id}^{-1/2}\,{\mathbf{R}}^T\,\widehat{\id}^{-1/2}\,y(0)&=\,0,\notag
 \end{align}
 where  $y(x_2):=\widehat{\id}^{1/2}\,z(x_2)$. Therefore we use the modified matrices
 \begin{align}\label{nTQ}
 \boldsymbol{\mathcal{T}}:=&\,k^2\,\begin{footnotesize}\begin{pmatrix} \frac{\mu_{\rm e} +\mu_{\rm c}}{\rho} &0	&0
 \\
 0& \frac{2\,\mu_{\rm e}  +\lambda_{\rm e} }{\rho}
 &0
 \\0
 & 0 & \frac{{\mu_{\rm e}\,L_{\rm c}^2}\,\gamma}{\rho\,j\,\mu_{\rm e}\,\tau_{\rm c}^2\,} \end{pmatrix}\end{footnotesize}, \qquad \qquad 
\boldsymbol{\mathcal{R}}:= k\,\begin{footnotesize}\begin{pmatrix} 0& k\,\frac{\lambda_{\rm e} }{\rho}& 0
 \\
 k\,\frac{\mu_{\rm e} -\mu_{\rm c}}{\rho}&0 
 &0
 \\2\,\frac{\mu_{\rm c}}{\rho\,\sqrt{ j\,\mu_{\rm e}\,\tau_{\rm c}^2\,}}  
 & 0 & 0 \end{pmatrix}\end{footnotesize},\notag\\
\boldsymbol{\mathcal{Q}}:=&\begin{footnotesize}\begin{pmatrix} k^2\frac{2\,\mu_{\rm e}  +\lambda_{\rm e}}{\rho}-k^2\,v^2&0	&0
 \\
 0&k^2\,\frac{\mu_{\rm e} +\mu_{\rm c}}{\rho}- k^2\,v^2
 &-2\,k\,\frac{\mu_{\rm c}}{\rho\,\sqrt{ j\,\mu_{\rm e}\,\tau_{\rm c}^2\,}}
 \\0
 &-2\,k\,\frac{\mu_{\rm c}}{\rho\,\sqrt{ j\,\mu_{\rm e}\,\tau_{\rm c}^2\,}} &k^2\,\frac{{\mu_{\rm e}\,L_{\rm c}^2}\,\gamma}{\rho\,j\,\mu_{\rm e}\,\tau_{\rm c}^2\,}+ 4\,\frac{\mu_{\rm c}}{\rho\,j\,\mu_{\rm e}\,\tau_{\rm c}^2\,} - k^2\,v^2 \end{pmatrix}\end{footnotesize},
 \end{align}
 and the following equivalent form of the system \eqref{11}
 \begin{align}\label{n11}
 \frac{1}{k^2}\boldsymbol{\mathcal{T}}\,y''(x_2)+ {\rm i}\,\frac{1}{k} (\boldsymbol{\mathcal{R}}+\boldsymbol{\mathcal{R}}^T)\,y'(x_2)-\boldsymbol{\mathcal{Q}}\,y(x_2)+k^2\, v^2 \, \id\,y(x_2)=&\,0,\\
 \frac{1}{k^2}\boldsymbol{\mathcal{T}}\,y'(0)+ {\rm i}\,\frac{1}{k} \,\boldsymbol{\mathcal{R}}^T\,y(0)=&\,0.\notag
 \end{align}

 \begin{lemma}If  the constitutive coefficients satisfy the conditions 	\eqref{d12},
 then the	 matrices $\boldsymbol{\mathcal{Q}}$ and $\boldsymbol{\mathcal{T}}$ are symmetric and positive definite.	
 \end{lemma}
 \begin{proof}
 	Symmetry is clear. It is easy to see that	\begin{align}
 	{\mathbf{T}}_{11}&= k^2(\mu_{\rm e} +\mu_{\rm c}) \,,\qquad 
 	{\mathbf{T}}_{11}{\mathbf{T}}_{22}-{\mathbf{T}}_{12}{\mathbf{T}}_{21}=k^4(2\,\mu_{\rm e}  +\lambda_{\rm e} )( \mu_{\rm e} +\mu_{\rm c} )\,,\\ 
 	\det {{\mathbf{T}}}&=k^6(2\,\mu_{\rm e}  +\lambda_{\rm e} )( \mu_{\rm e} +\mu_{\rm c} ){\mu_{\rm e}\,L_{\rm c}^2}\,\gamma.\notag
 	\end{align}
 	Therefore, our constitutive  hypothesis imply that ${\mathbf{T}}$ is positive-definite.
 	In addition,  ${\mathbf{Q}}$ is positive-definite if and only if the principal minors are positive, namely
 	\begin{align}
 	{\mathbf{Q}}_{11}&=k^2(2\,\mu_{\rm e}  +\lambda_{\rm e} ),\qquad
 	{\mathbf{Q}}_{11}{\mathbf{Q}}_{22}-{\mathbf{Q}}_{12}{\mathbf{Q}}_{21}=k^4(2\,\mu_{\rm e}  +\lambda_{\rm e} )(\mu_{\rm e} +\mu_{\rm c} )\,,\\
 	\det({\mathbf{Q}})&=k^4(2\,\mu_{\rm e}  +\lambda_{\rm e} )[(\mu_{\rm e}+\mu_{\rm c})\,   k^2{\mu_{\rm e}\,L_{\rm c}^2}\,\gamma+4\,\mu_{\rm e}\mu_{\rm c} ]\,,\notag
 	\end{align}
 i.e. under the hypothesis of the lemma. 
 Since $\mathbf{T}$ and $\mathbf{Q}$ are positive definite, so there are $\boldsymbol{\mathcal{T}}$ and $\boldsymbol{\mathcal{Q}}$ defined by \eqref{nTQ}, and the proof is complete.
 \end{proof}

We now  seek  a solution $y$ of the differential system \eqref{n11} in the form 
 \begin{align}\label{x7}
 y(x_2)=\begin{footnotesize}\begin{pmatrix} 	d_1
 	\\
 	d_2
 	\\d_3
 	\end{pmatrix}\end{footnotesize} \,e^{{\rm i}\,r\,k\,x_2},\qquad \text{Im}\,r>0,
 \end{align}
 where $r \, \in \mathbb{C}$ is a complex parameter, $d=\begin{footnotesize}\begin{pmatrix} 	d_1,
 &
 d_2,
 &d_3
 \end{pmatrix}\end{footnotesize}^T \in \mathbb{C}^3 $, $a
 \neq 0$ is the amplitude and $ \text{Im}\,r$ is the coefficient of the imaginary part of $r$. The  condition $\text{Im}\,r>0$ ensures the asymptotic decay condition \eqref{04}. 
 Inserting \eqref{x7} in $\eqref{11}_1$ we obtain the  systems of algebraic equations
 \begin{align}\label{x8}
 [r^2\boldsymbol{\mathcal{T}}+r\,(\boldsymbol{\mathcal{R}}+\boldsymbol{\mathcal{R}}^T)+\boldsymbol{\mathcal{Q}}-\ k^2 v^2 {\id}]\,d=0,\qquad 
 [r\,\boldsymbol{\mathcal{T}}+\boldsymbol{\mathcal{R}}^T]\,d=0.
 \end{align}
 The characteristic equation corresponding to the eigenvalue problem \eqref{x8}$_1$, i.e., the condition to have a nontrivial solution of $d=\begin{footnotesize}\begin{pmatrix} 	d_1,
&
 d_2,
&d_3
 \end{pmatrix}\end{footnotesize}^T$, is 
 \begin{align}\label{x9}
 \det\,[r^2\boldsymbol{\mathcal{T}}+r(\boldsymbol{\mathcal{R}}+\boldsymbol{\mathcal{R}}^T)+\boldsymbol{\mathcal{Q}}-k^2 v^2 {\id}]=0,
 \end{align}
 which gives six roots of the eigenvalue $r$. The associated eigenvectors $a$ can be determined for the corresponding eigenvalues. 
\begin{definition}
	By the {\bf limiting speed} we understand a speed $\widehat{v}>0$, such that for all wave speeds satisfying 	$0\le v<\widehat{v}$  (subsonic speeds) the roots  of the characteristic equation \eqref{x9} are not real\footnote{{Since the quadratic equation does not have real solutions, there is a complex solution $r$ for which ${\rm Im}\, r>0$, since the complex solutions are pair-conjugated. Therefore, the existence of such a solution $r$, i.e., ${\rm Im}\, r>0$, implies the existence of a  wave propagating in the direction $x_1$ with the
		phase velocity $v$ and decaying exponentially in the direction $x_2$.}} and vice versa, i.e., if the roots  of the characteristic equation \eqref{x9} are not real then they correspond to wave speeds $v$ satisfying 	$0\le v<\widehat{v}$.
\end{definition} 
 \begin{proposition}\label{lemmaGH} If  the constitutive coefficients satisfy the conditions 	\eqref{d12},
 	then there	exists a limiting speed $\widehat{v}>0$. Furthermore,   if one root   $r_v$ of the characteristic equation \eqref{x9} is real then it corresponds to a speed $v\geq \widehat{v}$ (non-admissible).
 \end{proposition}
 \begin{proof}
Assume that there exists a real $r_v$ as solution  of the characteristic equation \eqref{x9}, then $ \exists \,\theta \in (-\frac{\pi}{2},\frac{\pi}{2})$ such that $r_v=\tan \theta$. Therefore, corresponding to \eqref{x7},  $\mathcal{U}$  given by \eqref{x5}  and defined by $r_v$	turns into
 \begin{align}\label{46}
 \mathcal{U}(x_1,x_2,t)=\begin{footnotesize}
\begin{pmatrix} u_1(x_1,x_2,t)
 	\\
 	u_2(x_1,x_2,t)
 	\\\vartheta _3(x_1,x_2,t)
 	\end{pmatrix}\end{footnotesize}&=
\begin{footnotesize}
\begin{pmatrix} d_1
 	\\
 	d_2
 	\\ {\rm i}\, \,d_3
 		\end{pmatrix}\end{footnotesize}e^{ik\, (  x_1+\tan\theta x_2-vt)}=\begin{footnotesize}
 		\begin{pmatrix} d_1
 	\\
 	d_2
 	\\ {\rm i}\, \,d_3
 	\end{pmatrix}\end{footnotesize}e^{\frac{ik}{\cos\theta}(\cos\theta x_1+\sin\theta x_2-\cos\theta vt)},
 \end{align}
 which means that $\mathcal{U}(x_1,x_2,t)$ is   a non-trivial  plane body wave solution with  wave number $\frac{k}{\cos\theta}$, the speed $\widetilde{v}_\theta=v\cos\theta$ and propagation in the direction ${n}_\theta$ where 
 $	{n}_\theta=(\cos\theta,\sin\theta,0)$. A direct substitution of \eqref{46} into \eqref{algPDE} implies the existence of a non-trivial solution  $\begin{footnotesize}\begin{pmatrix}d_1,&d_2,&d_3\end{pmatrix}\end{footnotesize}\neq 0$ of the algebraic system
 \begin{footnotesize}
 \begin{align}\label{47}
 k^2\,\sin^2\theta\,(\mu_{\rm e} +\mu_{\rm c} )\,d_1+k\,\sin\theta \,\cos\theta[k\,\lambda_{\rm e} +k\, (  \mu_{\rm e} -\mu_{\rm c} )]\,d_2+2\,\mu_{\rm c} \,  \sin\theta\,\cos\theta\, d_3+\cos^2\theta\,[k^2(2\,\mu_{\rm e}  +\lambda_{\rm e} )-k^2\,\rho v^2]\,d_1=0,\notag\\
 k^2\,\sin^2\theta\,(2\,\mu_{\rm e}  +\lambda_{\rm e} )\,d_2+k\,\sin\theta \,\cos\theta[k\,\lambda_{\rm e} +k\, (  \mu_{\rm e} -\mu_{\rm c} )]\,d_1-2\,\mu_{\rm c} \,  k\cos^2\theta\,  d_3+\cos^2\theta[k^2\,(\mu_{\rm e} +\mu_{\rm c} )-k^2\,\rho\, v^2]\,d_2=0,\\
 k^2\,\sin^2\theta\,{\mu_{\rm e}\,L_{\rm c}^2}\,\gamma\, d_3+2\,\mu_{\rm c} \,k\, \sin\theta\, \cos\theta \,d_1-2\,\mu_{\rm c} \,  k \cos^2\theta\, d_2 +\cos^2\theta[k^2\,{\mu_{\rm e}\,L_{\rm c}^2}\,\gamma+4\,\mu_{\rm c} - \rho\,j\,\mu_{\rm e}\,\tau_{\rm c}^2\,k^2\, v^2]\,d_3=0,\notag
 \end{align}
 \end{footnotesize}
 which  in  matrix form gives
\begin{footnotesize}
 \begin{align}\label{x09}
 &\sin^2\theta\,\,\underbrace{\left({\begin{array}{ccc} \mu_{\rm e} +\mu_{\rm c} &0	&0
 	\\
 	0& 2\,\mu_{\rm e}  +\lambda_{\rm e} 
 	&0
 	\\0
 	& 0 & {\mu_{\rm e}\,L_{\rm c}^2}\,\gamma \end{array}}\right)}_{{\mathbf{T}}}\left({\begin{array}{ccc} 	d_1
 	\\
 	d_2
 	\\	d_3
 	\end{array}}\right)+\sin\theta\cos\theta\underbrace{\left({\begin{array}{ccc} 0& k\,(\mu_{\rm e} -\mu_{\rm c} +\lambda_{\rm e} )& 2\,\mu_{\rm c}  
 	\\
 	k\,(\mu_{\rm e} -\mu_{\rm c} +\lambda_{\rm e} )&0 
 	&0
 	\\2\,\mu_{\rm c}  
 	& 0 & 0 \end{array}}\right)}_{{\mathbf{R}}+{\mathbf{R}}^T}\left({\begin{array}{ccc} 	d_1
 	\\
 	d_2
 	\\	d_3
 	\end{array}}\right)\\&\qquad \qquad\qquad \quad  +\cos^2\theta\,\underbrace{\left({\begin{array}{ccc} k^2(2\,\mu_{\rm e}  +\lambda_{\rm e} )-\rho\, k^2\,v^2&0	&0
 	\\
 	0&k^2\,(\mu_{\rm e} +\mu_{\rm c} )-\rho\, k^2\,v^2
 	&-2\,\mu_{\rm c} \, k
 	\\0
 	&-2\,\mu_{\rm c} \, k &k^2\,{\mu_{\rm e}\,L_{\rm c}^2}\,\gamma+ 4\,\mu_{\rm c} -\rho\,j\,\mu_{\rm e}\,\tau_{\rm c}^2\, k^2\,v^2 \end{array}}\right)}_{{\mathbf{Q}}-k^2\, v^2\, \widehat{\id}}\left({\begin{array}{ccc} 	d_1
 	\\
 	d_2
 	\\	d_3
 	\end{array}}\right)=0,\notag
 \end{align}
\end{footnotesize}
and this is equivalent to
\begin{footnotesize}
 \begin{align}\label{x092}
&\Bigg[k^2\,\sin^2\theta\,\,\widehat{\id}^{-1/2}\begin{footnotesize}\begin{pmatrix} \mu_{\rm e} +\mu_{\rm c} &0	&0
		\\
		0& 2\,\mu_{\rm e}  +\lambda_{\rm e} 
		&0
		\\0
		& 0 & {\mu_{\rm e}\,L_{\rm c}^2}\,\gamma \end{pmatrix}\end{footnotesize}\widehat{\id}^{-1/2}+k\,\sin\theta\cos\theta\, \widehat{\id}^{-1/2}\begin{footnotesize}\begin{pmatrix} 0& k\,(\mu_{\rm e} -\mu_{\rm c} +\lambda_{\rm e} )& 2\,\mu_{\rm c}  
		\\
		k\,(\mu_{\rm e} -\mu_{\rm c} +\lambda_{\rm e} )&0 
		&0
		\\2\,\mu_{\rm c}  
		& 0 & 0 \end{pmatrix}\end{footnotesize}\widehat{\id}^{-1/2}\notag\\&+\cos^2\theta\,\widehat{\id}^{-1/2}\begin{footnotesize}\begin{pmatrix} k^2(2\,\mu_{\rm e}  +\lambda_{\rm e} )-\rho\, k^2\,v^2&0	&0
		\\
		0&k^2\,(\mu_{\rm e} +\mu_{\rm c} )-\rho\, k^2\,v^2
		&-2\,\mu_{\rm c} \, k
		\\0
		&-2\,\mu_{\rm c} \, k &k^2\,{\mu_{\rm e}\,L_{\rm c}^2}\,\gamma+ 4\,\mu_{\rm c} -\rho\,j\,\mu_{\rm e}\,\tau_{\rm c}^2\, k^2\,v^2 \end{pmatrix}\end{footnotesize}\widehat{\id}^{-1/2}\Bigg]\begin{footnotesize}\begin{pmatrix} 	f_1
	\\
	f_2
	\\	f_3
	\end{pmatrix}\end{footnotesize}=0,
\end{align}
\end{footnotesize}
where  $\begin{footnotesize}\begin{pmatrix}	f_1
\\
f_2
\\	f_3
\end{pmatrix}\end{footnotesize}=\widehat{\id}^{1/2}\begin{footnotesize}\begin{pmatrix}	d_1
\\
d_2
\\	d_3
\end{pmatrix}\end{footnotesize}\neq 0$.

 Therefore, the assumption from the beginning of the proof implies that
there exists  a solution $\begin{footnotesize}\begin{pmatrix}	f_1,
&
	f_2,
	&	f_3
\end{pmatrix}\end{footnotesize}\neq 0$ of the algebraic system written in matrix format
\begin{align}\label{x093}
\left[	\sin^2\theta\,\boldsymbol{\mathcal{T}}+\sin\theta\cos\theta (\boldsymbol{\mathcal{R}}+\boldsymbol{\mathcal{R}}^T)+\cos^2\theta\,\boldsymbol{\mathcal{Q}}-k^2 v^2_\theta \cos^2\theta \,{\id}\right]\,\begin{footnotesize}\begin{pmatrix} 	f_1
\\
f_2
\\	f_3
\end{pmatrix}\end{footnotesize}=0.
\end{align}

Let us observe that  equation \eqref{x09}  is actually the propagation condition for plane waves in isotropic Cosserat materials, in the fixed direction $	{n}_\theta=(\cos\theta,\sin\theta,0)$. Since  the constitutive coefficients satisfy the conditions 	\eqref{d12}, according to Proposition \ref{proprealw2}, for the direction $	{n}_\theta=(\cos\theta,\sin\theta,0)$ in particular, the system of partial differential equations \eqref{PDE} admits a non trivial solution in the form
\begin{align}
u(x_1,x_2,t)&=\begin{footnotesize}\begin{pmatrix}\widehat{u}_1\\\widehat{u}_2\\0\end{pmatrix}\end{footnotesize}
\, e^{{\rm i}\, \left(k\langle {\xi},\, x\rangle_{\mathbb{R}^3}-\,\omega \,t\right)}\,,\qquad\qquad  \vartheta (x_1,x_2,t)=\begin{footnotesize}{\rm i}\,\begin{pmatrix}0\\0\\\widehat{\vartheta }_3\end{pmatrix}\end{footnotesize} \, e^{{\rm i}\, \left(k\langle {\xi},\, x\rangle_{\mathbb{R}^3}-\,\omega \,t\right)},\\&
(\widehat{u}_1,\widehat{u}_2, \widehat{ \vartheta_3   })^T\in\mathbb{C}^{3}, \quad (\widehat{u}_1,\widehat{u}_2, \widehat{ \vartheta_3   })^T\neq 0\,\notag
\end{align}
only for real positive values $\omega^2$.
To each  $\theta \in(-\frac{\pi}{2},\frac{\pi}{2})$  we associate these  real frequencies $\omega_\theta$ satisfying 
\begin{align}
\det\,\{	\sin^2\theta\boldsymbol{\mathcal{T}}+\sin\theta\cos\theta (\boldsymbol{\mathcal{R}}+\boldsymbol{\mathcal{R}}^T)+\cos^2\theta\boldsymbol{\mathcal{Q}}-\omega^2_\theta \,{\id}\}=0.
\end{align}
Then, each $\omega_\theta$ defines a $v_\theta$ such that $ \omega_\theta=v_\theta\cos\theta$. We define $\widehat{v}$ as the minimum of the values of $v_\theta\cos\theta$ for all $\theta\in(-\frac{\pi}{2},\frac{\pi}{2})$. Hence,  this means that we find all  solutions $v_\theta$ of 
 \begin{align}\label{x11}
 \det\,\{	\sin^2\theta\boldsymbol{\mathcal{T}}+\sin\theta\cos\theta (\boldsymbol{\mathcal{R}}+\boldsymbol{\mathcal{R}}^T)+\cos^2\theta\boldsymbol{\mathcal{Q}}-k^2 v^2_\theta \cos^2\theta \,{\id}\}=0,
 \end{align}
 and we define\footnote{In others words, we can take the solution $v_b(\theta)$ of all possible plane body waves propagating in the direction $\widehat{n}$ and we take
 	\begin{align*}
 	\widehat{v}=\inf_{\theta\in(-\frac{\pi}{2},\frac{\pi}{2})}\frac{v_b(\theta)}{\cos\theta}.
 	\end{align*}}
 \begin{align}
 \widehat{v}=\inf_{\theta\in(-\frac{\pi}{2},\frac{\pi}{2})} v_\theta.\label{lims}
 \end{align}
 
 In conclusion,  if  there exists a value of $v$ such that the equation \eqref{x9} admits a real solution $r_v$, then $v$  must satisfy 
 $
 v\geq \widehat{v}.
$
Thus,  if $v$ is such that
 $
 0\le v<\widehat{v},
 $
 then $r$ can not be real and if $r$ is real, then $v\geq\widehat{v}$.
\end{proof}

\begin{proposition} If  the constitutive coefficients satisfy the conditions 	\eqref{d12}, then
	for all $\theta\in \left(-\frac{\pi}{2},\frac{\pi}{2}\right)$ and $k>0$, the tensor $\boldsymbol{\mathcal{Q}}_\theta:=\sin^2\theta\boldsymbol{\mathcal{T}}+\sin\theta\cos\theta (\boldsymbol{\mathcal{R}}+\boldsymbol{\mathcal{R}}^T)+\cos^2\theta\boldsymbol{\mathcal{Q}}$ is positive definite.
\end{proposition}
\begin{proof}
	Since  the constitutive coefficients satisfy the conditions 	\eqref{d12}, according to Proposition \ref{proprealw2}, it follows that there exists \textit{only} real numbers $\omega$ such that the following  system\footnote{This system follows directly from \eqref{algPDE}. Inserting  \eqref{ansatzwp} into \eqref{PDE} we see that if $\widehat{u}_1, \widehat{u}_2$ and $\widehat{\vartheta}_3$ have to satisfy the following partial differential equations
		\begin{align*}
		-\omega^2\rho\, \widehat{u}_i&=-\,k^2\,(\mu_{\rm e} +\mu_{\rm c})\dd \widehat{u}_i \sum_{l=1}^2 \xi_l^2-k^2(\mu_{\rm e} -\mu_{\rm c}+\lambda_{\rm e} )\dd\sum_{l=1}^2\widehat{u}_l\, \xi_i\,\xi_l -2\,k\,\mu_{\rm c}\,\sum_{l=1}^2\varepsilon _{il3}\,\widehat{\vartheta }_3\, \xi_l, \quad  i=1,2,
		\vspace{1.2mm}\\
		-{\rm i}\,\omega^2\rho\, j\,\mu_{\rm e}\,\tau_{\rm c}^2\,\widehat{ \vartheta }_3&=-{\rm i}\,k^2\,{\mu_{\rm e}L_{\rm c}^2}\,(\alpha_1+\alpha_2)\dd\sum_{l=1}^2\widehat{ \vartheta }_3\xi_l^2-2\,{\rm i}\,\,k\,\mu_{\rm c}\,\sum_{l,s=1}^2\varepsilon _{3ls}\widehat{u}_s \xi_l-4\,{\rm i}\, \mu_{\rm c} \widehat{\vartheta }_3\sum_{l=1}^2\,\xi_l^2.\notag
		\end{align*}}  admits non trivial solutions $w=\begin{footnotesize}\begin{pmatrix}
	\widehat{u}_1,\widehat{u}_2,\widehat{ \vartheta }_3
	\end{pmatrix}\end{footnotesize}^T\neq 0$  for  real positive numbers $\omega^2$, i.e.,
	\begin{align}\label{vz1}
	[\mathbf{\mathbf{Q}}_1(\xi,k)-\omega^2\hat\id]\, w=0,
	\end{align}
	where  
	\begin{align}\label{vz3}
	\mathbf{\mathbf{Q}}_1(\xi,k)&=\begin{footnotesize}\begin{footnotesize}\begin{pmatrix}
	k^2(2\, \mu_{\rm e} +\lambda_{\rm e} )\xi_1^2+k^2(\mu_{\rm e} +\mu_{\rm c})\xi_2^2& k^2\,( \mu_{\rm e} -\mu_{\rm c}+\lambda_{\rm e} )\xi_1\xi_2& 2\, k\, \mu_{\rm c}\, \xi_2\vspace{2mm}\\
	k^2\,( \mu_{\rm e} -\mu_{\rm c}+\lambda_{\rm e} )\dd\xi_1\xi_2& k^2(\mu_{\rm e} +\mu_{\rm c})\xi_1^2+k^2(2\, \mu_{\rm e} +\lambda_{\rm e} )\xi_2^2& -2\, k\, \mu_{\rm c}\, \xi_1\vspace{2mm}\\
	2\, k\, \mu_{\rm c}\, \xi_2& -2\, k\, \mu_{\rm c}\, \xi_1&k^2\,{\mu_{\rm e}L_{\rm c}^2}\,(\alpha_1 +\alpha_2)(\xi_1^2+\xi_2^2)+4\, \mu_{\rm c}(\xi_1^2+\xi_2^2)
	\end{pmatrix}\end{footnotesize}\end{footnotesize}.
	\end{align}

	But this is equivalent to  the positive definiteness of the matrix
	\begin{align}
	\widetilde{\mathbf{\mathbf{Q}}}_1(\xi,k)
	&=\widehat{\id}^{-1/2}\begin{footnotesize}\begin{pmatrix}
	k^2(2\, \mu_{\rm e} +\lambda_{\rm e} )\xi_1^2+(\mu_{\rm e} +\mu_{\rm c})\xi_2^2& k^2\,( \mu_{\rm e} -\mu_{\rm c}+\lambda_{\rm e} )\xi_1\xi_2& 2\, k\, \mu_{\rm c}\, \xi_2\vspace{2mm}\\
	k^2\,( \mu_{\rm e} -\mu_{\rm c}+\lambda_{\rm e} )\dd\xi_1\xi_2& (\mu_{\rm e} +\mu_{\rm c})\xi_1^2+k^2(2\, \mu_{\rm e} +\lambda_{\rm e} )\xi_2^2& -2\, k\, \mu_{\rm c}\, \xi_1\vspace{2mm}\\
	2\, k\, \mu_{\rm c}\, \xi_2& -2\, k\, \mu_{\rm c}\, \xi_1&k^2(\alpha_1 +\alpha_2)+4\, \mu_{\rm c}
	\end{pmatrix}\end{footnotesize}\widehat{\id}^{-1/2}.
	\end{align}
	Since the conditions \eqref{d12} imply the positive definiteness of the matrix $\widetilde{\mathbf{\mathbf{Q}}}_1(\xi,k)$ for all $\xi=(\xi_1,\xi_2,0)\in \mathbb{R}^3$, $\lVert \xi\rVert=1$, taking $\xi=(\cos\,\theta, \sin\,\theta,0)$ we deduce  the desired result.
\end{proof}
\begin{remark}
	The Legendre–Hadamard ellipticity condition is not sufficient for the positive definiteness of $\boldsymbol{\mathcal{Q}}_\theta$, since they are equivalent to \eqref{LHc}
which does not imply \eqref{d11}.  {Instead, we will see that  the   Legendre–Hadamard ellipticity condition leads to the  acoustic tensor and its positive definiteness.} {Note that there does not exist a direction $\xi\in \mathbb{R}^3$ such that  $\boldsymbol{\mathcal{Q}}_\theta$  equals the acoustic tensor.} Moreover, even if in classical linear elasticity this is well known, in the Cosserat theory it seems to  be impossible (or at least it is not obvious to us) to arrive from the Legendre–Hadamard ellipticity condition {(the positive definiteness of the acoustic tensor) to the structure of the tensor  $\boldsymbol{\mathcal{Q}}_\theta$ or to its positive definiteness.}
\end{remark}
\begin{proof}
	The Legendre–Hadamard ellipticity condition \cite{eremeyev2007constitutive,Eremeyev4,shirani2020legendre} demands that the {\it acoustic tensor} 
	\begin{align}
	\widehat{\mathbf{Q}}=\begin{footnotesize}\begin{pmatrix}
	\widehat{\mathbf{Q}}_1&0\\0&\widehat{\mathbf{Q}}_2
	\end{pmatrix}\end{footnotesize}
	\end{align}
	defined through
	\begin{align}
	D^2W(\mathbf{e},\mathbf{\mathfrak{K}}).((\xi\otimes \eta, \zeta\otimes \eta),(\xi\otimes \eta, \zeta\otimes \eta))&=
	D^2W_1(\mathbf{e}).(\xi\otimes \eta,\xi\otimes \eta)+D^2W_2(\mathbf{\mathfrak{K}}).( \zeta\otimes \eta, \zeta\otimes \eta)\notag\\&=\bigl\langle \eta, \widehat{\mathbf{Q}}_1(\xi)\, \eta\bigl\rangle+\bigl\langle \eta, \widehat{\mathbf{Q}}_2(\zeta)\, \eta\bigl\rangle
	\end{align}
	is strictly positive definite for any
	nonzero wave directions $\xi\in \mathbb{R}^3$ and $\zeta\in \mathbb{R}^3$. Hence, let us identify the matrix $\widehat{\mathbf{Q}}_1$ and $\widehat{\mathbf{Q}}_2$. First identify $\widehat{\mathbf{Q}}_1$.
	Since 
	\begin{align}
	D^2W_1(\mathbf{e}).(\xi\otimes \eta,\xi\otimes \eta)
	&=
	\mu_{\rm e} \,\lVert  \sym\,\xi\otimes \eta\rVert ^{2}+\mu_{\rm c}\,\lVert \skw\,\xi\otimes \eta\rVert ^{2}+\frac{\lambda_{\rm e} }{2}\left[\mathrm{tr} \left(\xi\otimes \eta\right)\right]^{2}\notag\\
	&=
	\frac{\mu_{\rm e} +\mu_{\rm c}}{2}\lVert \xi\rVert^2\, \lVert \eta\rVert^2+\frac{\mu_{\rm e} -\mu_{\rm c}+\lambda_{\rm e} }{2}\bigl\langle \xi,\eta\bigr\rangle^2
	\\
	&=
	\frac{\mu_{\rm e} +\mu_{\rm c}}{2}\sum_{i=1}^3\xi_i^2 \sum_{i=1}^3\eta_i^2+\frac{\mu_{\rm e} -\mu_{\rm c}+\lambda_{\rm e} }{2}\bigg(\sum_{i=1}^3 \xi_i^2\eta_i^2+2\sum_{\substack{i,k=1\\i\neq k}}^3 \xi_i\xi_k\eta_i\eta_k \bigg),\notag
	\end{align}
	the matrix $\widehat{\mathbf{Q}}_1$ having the property
	\begin{align}
	D^2W_1(\mathbf{e}).(\xi\otimes \eta,\xi\otimes \eta)=\bigl\langle \eta, \widehat{\mathbf{Q}}_1(\xi)\, \eta\bigl\rangle
	\end{align}
	is given by
	\begin{align}
	\widehat{\mathbf{Q}}_1(\xi)
	=&\frac{1}{2}\begin{footnotesize}\begin{footnotesize}\begin{pmatrix}
	(2\,\mu_{\rm e} +\lambda_{\rm e} )\dd\xi_1^2+(\mu_{\rm e} +\mu_{\rm c})(\xi_2^2+\xi_3^2)& (\mu_{\rm e} -\mu_{\rm c}+\lambda_{\rm e} )\xi_1\xi_2 &(\mu_{\rm e} -\mu_{\rm c}+\lambda_{\rm e} )\xi_1\xi_3\vspace{2mm}\\
	(\mu_{\rm e} -\mu_{\rm c}+\lambda_{\rm e} )\xi_1\xi_2& 	(2\,\mu_{\rm e} +\lambda_{\rm e} )\dd\xi_2^2+(\mu_{\rm e} +\mu_{\rm c})(\xi_3^2+\xi_1^2)&(\mu_{\rm e} -\mu_{\rm c}+\lambda_{\rm e} )\xi_2\xi_3\vspace{2mm}\\
	(\mu_{\rm e} -\mu_{\rm c}+\lambda_{\rm e} )\xi_1\xi_3& (\mu_{\rm e} -\mu_{\rm c}+\lambda_{\rm e} )\xi_2\xi_3&	(2\,\mu_{\rm e} +\lambda_{\rm e} )\dd\xi_3^2+(\mu_{\rm e} +\mu_{\rm c})(\xi_1^2+\xi_2^2)
	\end{pmatrix}\end{footnotesize}\end{footnotesize}.\notag
	\end{align}

Next, we identify $\widehat{\mathbf{Q}}_2$.
Since 
\begin{align}
D^2W_2(\mathbf{\mathfrak{K}}).(\zeta\otimes \eta,\zeta \otimes \eta)
&=\ {\alpha_1}\,\lVert \sym\,\zeta\otimes \eta\rVert ^{2}+{\alpha_2}\,\lVert \skw\,\zeta\otimes \eta\rVert ^{2}+\frac{\alpha_3}{2}\left[\mathrm{tr} \left(\zeta\otimes \eta\right)\right]^{2}
\notag\\
&=
\frac{\alpha_1+\alpha_2}{2}\sum_{i=1}^3\zeta_i^2 \sum_{i=1}^3\eta_i^2+\frac{\alpha_1 -\alpha_2+\alpha_3 }{2}\bigg(\sum_{i=1}^3 \zeta_i^2\eta_i^2+2\sum_{\substack{i,k=1\\i\neq k}}^3 \zeta_i\zeta_k\eta_i\eta_k \bigg),\notag
\end{align}
the matrix $\widehat{\mathbf{Q}}_2$  which satisfies
\begin{align}
D^2W_2(\mathbf{\mathfrak{K}}).(\zeta\otimes \eta,\zeta\otimes \eta)=\bigl\langle \eta, \widehat{\mathbf{Q}}_2(\zeta)\, \eta\bigl\rangle
\end{align}
is given by
\begin{align}
\widehat{\mathbf{Q}}_2(\zeta)
=&\frac{1}{2}\begin{footnotesize}\begin{pmatrix}
(2\,\alpha_1 +\alpha_3)\zeta_1^2+(\alpha_1+\alpha_2 )(\zeta_1^2+\zeta_3^2)& (\alpha_1 -\alpha_2+\alpha_3)\zeta_1\zeta_2 &(\alpha_1 -\alpha_2+\alpha_3 )\zeta_1\zeta_3\\
(\alpha_1 -\alpha_2+\alpha_3 )\zeta_1\zeta_2& (2\,\alpha_1 +\alpha_3)\zeta_2^2+(\alpha_1+\alpha_2 )(\zeta_3^2+\zeta_1^2)&(\alpha_1 -\alpha_2+\alpha_3 )\zeta_2\zeta_3\\
(\alpha_1 -\alpha_2+\alpha_3)\zeta_1\zeta_3& (\alpha_1-\alpha_2+\alpha_3 )\zeta_2\zeta_3&(2\,\alpha_1 +\alpha_3)\zeta_3^2+(\alpha_1+\alpha_2 )(\zeta_1^2+\zeta_2^2)
\end{pmatrix}\end{footnotesize}.\notag
\end{align}

Actually, our entire approach is based on the assumption that the  matrix $\widetilde{\mathbf{\mathbf{Q}}}_1(\xi,k)$  defined by \eqref{vz3} is positive definite. So, if the Legendre-Hadamard condition implies this fact, or it is equivalent with, then the Legendre-Hadamard condition would be also suitable for our approach. However, it seems that this is not possible because we did not find   appropriate values for $\xi$ and $\zeta$ such that, for some particular values of them the positive definiteness of $\widehat{\mathbf{Q}}_1(\xi)$  $\widehat{\mathbf{Q}}_2(\xi)$ would imply the  positive definiteness of the matrix $\widetilde{\mathbf{\mathbf{Q}}}_1(\xi,k)$, as we need in our approach.

Even when we consider the Legendre-Hadamard ellipticity condition on the set of admissible solutions of the form
\begin{align}
u(x_1,x_2,t)=\begin{footnotesize}\begin{pmatrix} u_1(x_1,x_2,t)
\\
u_2(x_1,x_2,t)
\\0
\end{pmatrix}\end{footnotesize},\qquad  \vartheta   (x_1,x_2,x_3,t)=\begin{footnotesize}\begin{pmatrix} 0
\\
0
\\\vartheta _3(x_1,x_2,t)
\end{pmatrix}\end{footnotesize},
\end{align}
it is not clear how the
positive definiteness of the  matrix $\widetilde{\mathbf{\mathbf{Q}}}_1(\xi,k)$ can possibly  be a consequence of the Legendre-Hadamard ellipticity condition.
\end{proof}

\begin{proposition}\label{propQp}
If  the constitutive coefficients satisfy the conditions 	\eqref{d12}, then
for all $\theta\in \left(-\frac{\pi}{2},\frac{\pi}{2}\right)$, $k>0$ and $0\le v<\widehat{v}$, the tensor $\widetilde{\mathbf{Q}}_\theta:=\sin^2\theta\,\boldsymbol{\mathcal{T}}+\sin\theta\cos\theta (\boldsymbol{\mathcal{R}}+\boldsymbol{\mathcal{R}}^T)+\cos^2\theta\,\boldsymbol{\mathcal{Q}}-k^2 v^2 \cos^2\theta \,{\id}$ is positive definite.
\end{proposition}
\begin{proof}
Since $\boldsymbol{\mathcal{Q}}_\theta:=\sin^2\theta\boldsymbol{\mathcal{T}}+\sin\theta\cos\theta (\boldsymbol{\mathcal{R}}+\boldsymbol{\mathcal{R}}^T)+\cos^2\theta\boldsymbol{\mathcal{Q}}$ is positive definite, it admits only positive eigenvalues. We need to prove that for all $\theta\in \left(-\frac{\pi}{2},\frac{\pi}{2}\right)$, $k>0$ and for all $0\le v<\widehat{v}$ the matrix $\widetilde{\mathbf{Q}}_\theta=\boldsymbol{\mathcal{Q}}_\theta-k^2 v^2 \cos^2\theta \,{\id}$ is positive definite, which is equivalent  to the property that   all eigenvalues $\boldsymbol{\mathcal{Q}}_\theta$ are larger than those of the matrix $k^2 v^2 \cos^2\theta \,{\id}$, i.e., than $k^2 v^2 \cos^2\theta $.   Assuming that there exist  $\theta_0\in \left(-\frac{\pi}{2},\frac{\pi}{2}\right)$ and  $v_0\in[0,\widehat{v})$ for which there is an eigenvalue $\lambda_{\theta_0}$   of $\boldsymbol{\mathcal{Q}}_\theta$
such that $\lambda_{\theta_0}<k^2 v^2_0 \cos^2\theta_0 $, then 
\begin{align}
v_{\theta_0}:=\sqrt{\frac{\lambda_{\theta_0}}{k^2\,\cos^2\theta_0}}<v_0<\widehat{v}
\end{align}
is solution of \eqref{x11}, i.e., for fixed $\theta_0$ we have that $v_{\theta_0}<\widehat{v}$  verifies
 \begin{align}\label{x11n}
\det\,\{	\sin^2{\theta_0}\,\boldsymbol{\mathcal{T}}+\sin{\theta_0}\,\cos{\theta_0} \,(\boldsymbol{\mathcal{R}}+\boldsymbol{\mathcal{R}}^T)+\cos^2{\theta_0}\,\boldsymbol{\mathcal{Q}}-k^2 v^2_{{\theta_0}}\, \cos^2{\theta_0} \,{\id}\}=0.
\end{align} 
This is in contradiction to the definition of the limiting speed and Proposition \ref{lims}, since $\widehat{v}$ is the smallest speed having this property. Therefore, it remains that for all  $\theta\in \left(-\frac{\pi}{2},\frac{\pi}{2}\right)$, $k>0$ and for all $0\le v<\widehat{v}$, all the eigenvalues of $\boldsymbol{\mathcal{Q}}_\theta$ are larger than $k^2 v^2 \cos^2\theta$ and the proof is complete.
\end{proof}

\section{The common method to construct the solution using the Stroh formalism}\setcounter{equation}{0}\label{comSec}

In this subsection we present the main steps of the common method using the Stroh formalism, since in this  point of our approach it is still possible to switch to the Stroh formalism and vice versa, without using conceptually  different methods.

On one hand, using the ansatz $z(x_2):=\widehat{\id}^{1/2}\,\begin{footnotesize}\begin{pmatrix} 	d_1
\\
d_2
\\d_3
\end{pmatrix}\end{footnotesize} \,e^{{\rm i}\,r\,k\,x_2}$ in \eqref{11}$_2$, we are {led} to   define a vector $b\neq 0 \in \mathbb{C}^3 $ as
\begin{align}\label{61}
b=[r\,{\mathbf{T}}+{\mathbf{R}}^T]d.
\end{align}
Since ${\mathbf{T}}$ is symmetric and positive definite there exists ${\mathbf{T}}^{-1}$. Hence from  \eqref{61} we get
\begin{align}\label{63}
r\,d=-{\mathbf{T}}^{-1}{\mathbf{R}}^Td+{\mathbf{T}}^{-1}b
\end{align}
and 
\begin{align}
r\,b=[r^2{\mathbf{T}}+r\,{\mathbf{R}}^T]d.
\end{align}

On the other hand,  from \eqref{x7}$_1$ and considering the same ansatz for $z(x_2)$ as before, we deduce
\begin{align}\label{62}
\overbrace{(r^2\,{\mathbf{T}}+r\,{\mathbf{R}}^T)d}^{rb}+(r\,{\mathbf{R}}+\widetilde{\,{\mathbf{Q}}})d=0\qquad \Leftrightarrow\qquad 
r\,b=-[r\,{\mathbf{R}}+\widetilde{{\mathbf{Q}}}]d,
\end{align}
where $\widetilde{{\mathbf{Q}}}=\,{\mathbf{Q}}-\ k^2 v^2 \widehat{\id}$. Making use of \eqref{63} in \eqref{62} we obtain
\begin{align}\label{67}
rb&=-\widetilde{{\mathbf{Q}}}\,d-{\mathbf{R}}\{-{\mathbf{T}}^{-1}{\mathbf{R}}^da+{\mathbf{T}}^{-1}b\}=[-\widetilde{{\mathbf{Q}}}+{\mathbf{R}}\,{\mathbf{T}}^{-1}\,{\mathbf{R}}^T]d-{\mathbf{R}}\,{\mathbf{T}}^{-1} b.
\end{align}
From \eqref{63} and $\eqref{67}_2$ it is possible to  indicate the relation of our analysis with the study given in \cite{ChiritaGhiba3}, i.e., the scalar variable $r$ appearing in our ansatz is a solution of the following eigenvalue problem	\begin{align}\label{y13}
	\underbrace{\begin{footnotesize}\begin{pmatrix}
	-{\mathbf{T}}^{-1}{\mathbf{R}}^T  &	{\mathbf{T}}^{-1}\\
	{\mathbf{R}}\,{\mathbf{T}}^{-1}\,{\mathbf{R}}^T-\widetilde{{\mathbf{Q}}} & \ \ \ -{\mathbf{R}} \,{\mathbf{T}}^{-1}    
	\end{pmatrix}\end{footnotesize}}_{:=\mathcal{N}}\underbrace{\begin{footnotesize}\begin{pmatrix}d\\b\end{pmatrix}\end{footnotesize}}_{:=\mathcal{V}}=r\begin{footnotesize}\begin{pmatrix}d\\b\end{pmatrix}\end{footnotesize}\qquad \Leftrightarrow\qquad 
	\mathcal{N}\,\mathcal{V}=r\,\mathcal{V},
	\end{align}
	where $\mathcal{N}$ $\in \mathbb{R}^{6\times 6}$ is called the {\bf Stroh matrix} \cite{ting1996anisotropic,ChiritaGhiba3}.

For Cosserat elastic materials  \cite{ChiritaGhiba3} it is possible to find a suitable structure of the characteristic equation corresponding to the above eigenvalue problem in the form\footnote{We have rewritten everything in the notation of the present paper.}
\begin{equation}
r^{6}+P_{1}r^{4}+P_{2}r^{2}+P_{3}=0, \label{v25}%
\end{equation}
where
\begin{align}\label{defpC}
 P_{1}&=C_{1}+C_{2}+C_{3},\qquad  \qquad
 P_{2}=C_{1}C_{2}+C_{2}C_{3}+C_{3}C_{1}-\displaystyle\frac{4\,\rho\, \mu_{\rm c} ^2 v^2}{{\mu_{\rm e}\,L_{\rm c}^2}\,\gamma \,(\mu_{\rm e} +\mu_{\rm c} )^2},\notag\\ P_{3}&=C_{1}C_{2}C_{3}-\frac{4\,\rho\, \mu_{\rm c} ^2 v^2}{{\mu_{\rm e}\,L_{\rm c}^2}\,\gamma \,(\mu_{\rm e} +\mu_{\rm c} )^2}C_{1},\\
 {c}_{l}&=1-\frac{\rho\, v^2}{2\,\mu_{\rm e}  +\lambda_{\rm e} },\quad \quad\qquad {c}_{t}=1-\frac{\rho \,v^2}{\mu_{\rm e} +\mu_{\rm c} }, \quad  \quad\qquad {c}_{m}=1+\frac{4\,\mu_{\rm c} \,\mu_{\rm e} }{{\mu_{\rm e}\,L_{\rm c}^2}\,\gamma\, (\mu_{\rm e} +\mu_{\rm c} )}-\frac{\rho\,j\,\mu_{\rm e}\,\tau_{\rm c}^2\, v^2}{{\mu_{\rm e}\,L_{\rm c}^2}\,\gamma}.\notag
\end{align}
Then, due to this structure, it is possible to 
factorise  it as
\begin{align}
(r^2+{c}_{l})[(r^2+{c}_{t})(r^2+{c}_{m})-\frac{4\,\rho\, \mu_{\rm c} ^2 v^2}{{\mu_{\rm e}\,L_{\rm c}^2}\,\gamma \,(\mu_{\rm e} +\mu_{\rm c} )^2}]=0,	
\end{align}
and to find the analytical form of its  solutions
\begin{align}
&r_1^2=-{c}_{l},\notag\\
&r_2^2=\frac{1}{2}\left[-({c}_{t}+{c}_{m})+\sqrt{({c}_{t}-{c}_{m})^2+\frac{8\,\rho \,\mu_{\rm c} ^2\, v^2}{{\mu_{\rm e}\,L_{\rm c}^2}\,\gamma\, (\mu_{\rm e} +\mu_{\rm c} )^2}}\,\right],\\
&r_3^2=\frac{1}{2}\left[-({c}_{t}+{c}_{m})-\sqrt{({c}_{t}-{c}_{m})^2+\frac{8\,\rho \,\mu_{\rm c} ^2\, v^2}{{\mu_{\rm e}\,L_{\rm c}^2}\,\gamma\, (\mu_{\rm e} +\mu_{\rm c} )^2}}\,\right].\notag
\end{align}

Since  the explicit analytical form of the roots  $r_v$  of the characteristic equation \eqref{v25} as function of $v$ are known,
in \cite{ChiritaGhiba3} it is shown that 
\begin{proposition}\label{lemmaCG1} If  the constitutive coefficients satisfy the conditions 
	\begin{align}\label{d11n}
	2\,\mu_{\rm e} +\lambda_{\rm e} >0,\qquad \qquad \mu_{\rm e}+ \mu_{\rm c}>0,\qquad \qquad \alpha_1+\alpha_2>0,
	\end{align}
	then 	 the roots  $r_v$  of the characteristic equation \eqref{v25} are not real if and only if  
	\begin{align}
	0\le v<\min\left\{ \mathfrak{c}_p,\mathfrak{c}_s, \mathfrak{c}_{m_1},\mathfrak{c}_{m_2}\right\}\,,
	\end{align}	
	where
	\begin{align}\label{v36} \mathfrak{c}_{p}&=\sqrt{\displaystyle\frac{2\,\mu_{\rm e}  +\lambda_{\rm e}}{\rho}}, \qquad\qquad \mathfrak{c}_{m_1}=\sqrt{\displaystyle\frac{{\mu_{\rm e}\,L_{\rm c}^2}\,\gamma}{\rho\, j\,\mu_{\rm e}\,\tau_{\rm c}^2\,}\left[
	1+\displaystyle\frac{4\, \mu_{\rm c}\,\mu_{\rm e}}{{\mu_{\rm e}\,L_{\rm c}^2}\,\gamma\, k^{2}(\mu_{\rm e}+ \mu_{\rm c})}\right]}, \notag\\
	\mathfrak{c}_{s}&=\sqrt{\displaystyle\frac{\mu_{\rm e}+ \mu_{\rm c}}{\rho}}, \qquad\ \, \qquad
	\mathfrak{c}_{m_2}=\sqrt{\frac{1}{2}\left(  \mathfrak{c}_{s}^2+\mathfrak{c}_{m_1}^2%
	+\frac{4\, \mu_{\rm c}^{2}}{k^{2}\rho \, j\,\mu_{\rm e}\,\tau_{\rm c}^2\,(\mu_{\rm e}+ \mu_{\rm c})}-\sqrt{\Delta}\right)}  ,\vspace{1.2mm}\\
	\Delta&=\left(  \mathfrak{c}_{s}^2-\mathfrak{c}_{m_1}^2\right)  ^{2}+2\left(
	\mathfrak{c}_{s}^2+\mathfrak{c}_{m_1}^2\right)  \frac{4\, \mu_{\rm c}^{2}}{k^{2}\rho\,j\,\mu_{\rm e}\,\tau_{\rm c}^2\,(\mu_{\rm e}+ \mu_{\rm c})%
	}+\left(  \frac{4\, \mu_{\rm c}^{2}}{k^{2}\rho\,j\,\mu_{\rm e}\,\tau_{\rm c}^2\,(\mu_{\rm e}+ \mu_{\rm c})}\right)  ^{2}>0.
\notag 
	\end{align}
\end{proposition}

\begin{remark}{The inequalities \eqref{d11n} from Proposition \ref{lemmaCG1} have the following interpretations:
\begin{itemize}
\item[i)] the first two inequalities of the set of conditions \eqref{Chiritacond} considered by Chiri\c t\u a and Ghiba \cite{ChiritaGhiba3} imply that the translational compressional wave is real and  that the shear-rotational wave (optical branch) is real at high frequencies;
\item[ii)] the first inequality also implies that at the limit of high frequencies the translational compressional wave is faster than the shear–rotational wave (if they
both exist);
\item[iii)]  the third one means that the
shear–rotational wave (acoustic branch) is real at high frequencies. 
\end{itemize} 
In addition, 
\begin{itemize} \item[iv)] the inequalities \eqref{Chiritacond} do not imply that  the shear-rotational wave (optical branch) is real at low frequencies, i.e., $\omega\to 0$;
	\item[v)] the inequalities \eqref{Chiritacond} do not imply that the plane waves are real, i.e., that the propagating plane waves  are defined only by real frequencies.
\end{itemize}}
\end{remark}

Notice that we have slightly changed the results obtained in \cite{ChiritaGhiba3}, and we do not assume that $\mathfrak{c}_{p}\geq\mathfrak{c}_{s}$ because all the calculations from \cite{ChiritaGhiba3} are indeed valid without this additional assumption. Actually,  the inequality $\mathfrak{c}_{p}\geq\mathfrak{c}_{s}$ holds true once the internal energy is assumed to be positive definite, i.e., when $2\,\mu_{\rm e} +3\,\lambda_{\rm e} >0, \mu_{\rm e}  >0$. However, as we show in the present paper, the existence of the seismic waves is  true for weaker conditions on the constitutive parameters.

We also mention that under the hypothesis of the above proposition, the following inequalities are valid 
\begin{align} 
\mathfrak{c}_{m_2}<\mathfrak{c}_{m_1}, \qquad \qquad \text{\rm and}\qquad \qquad \mathfrak{c}_{m_2}<\mathfrak{c}_s,
\end{align}
since
\begin{align}
	-\sqrt{\Delta}&\leq -\left|\left(  \mathfrak{c}_{s}^2-\mathfrak{c}_{m_1}^2\right)  + \frac{4\, \mu_{\rm c}^{2}}{k^{2}\rho\,j\,\mu_{\rm e}\,\tau_{\rm c}^2\,(\mu_{\rm e}+ \mu_{\rm c})%
}\right|,\notag\\
	-\sqrt{\Delta}&\leq -\left|\left(  \mathfrak{c}_{m_1}^2-\mathfrak{c}_{s}^2\right)  + \frac{4\, \mu_{\rm c}^{2}}{k^{2}\rho\,j\,\mu_{\rm e}\,\tau_{\rm c}^2\,(\mu_{\rm e}+ \mu_{\rm c})%
}\right|\notag,\\
\mathfrak{c}_{m_2}^2-\mathfrak{c}_{m_1}^2&=\frac{1}{2}\left(  \mathfrak{c}_{s}^2-\mathfrak{c}_{m_1}^2%
+\frac{4\, \mu_{\rm c}^{2}}{k^{2}\rho \, j\,\mu_{\rm e}\,\tau_{\rm c}^2\,(\mu_{\rm e}+ \mu_{\rm c})}-\sqrt{\Delta}\right)<0 ,\\
\mathfrak{c}_{m_2}^2-\mathfrak{c}_s^2&=\frac{1}{2}\left(  \mathfrak{c}_{m_1}^2-\mathfrak{c}_{s}^2%
+\frac{4\, \mu_{\rm c}^{2}}{k^{2}\rho \, j\,\mu_{\rm e}\,\tau_{\rm c}^2\,(\mu_{\rm e}+ \mu_{\rm c})}-\sqrt{\Delta}\right)<0.\notag
\end{align}
Therefore, if the constitutive parameters satisfy the  conditions \ref{d11n}, then the speeds $\mathfrak{c}_i$, $i=1,2,3,4$ defined in \eqref{v36}, are ordered as follows
\begin{align}
\mathfrak{c}_{m_2}<\min\{\mathfrak{c}_{m_1},\mathfrak{c}_s\}.
\end{align}
This inequality does not mean that the  conditions \ref{d11n} impose (or is in contradiction with) a sort of a priori order of Eringen-type \eqref{Eringencond} between the  speed of microscopic real waves and the  speed of macroscopic real waves, since in the definition of $\mathfrak{c}_{m_2}$  both the microscopic and macroscopic constitutive parameters are involved.

Using this last remark and combining Proposition \ref{lemmaCG1} and  Proposition \ref{lemmaGH}, we conclude
\begin{proposition}\label{ChGhi}For $
2\,\mu_{\rm e} +\lambda_{\rm e} >0, \ \mu_{\rm e} >0, \ \mu_{\rm c} >0,\ {\mu_{\rm e}\,L_{\rm c}^2}\,\gamma>0,$ the limiting speed in Cosserat elastic materials is given by
\begin{align}\label{limCh}
\widehat{v}:=\inf_{\theta\in(-\frac{\pi}{2},\frac{\pi}{2})} v_\theta\equiv \min\left\{ 	 \mathfrak{c}_p,\mathfrak{c}_{m_2}\right\}.
\end{align}
\end{proposition}

 The usual Stroh algorithm  to construct the solution of the seismic wave propagation problem is based on the possibility to explicitly know the analytical form of the  solution   $r_v$ of equation  \eqref{v25} as function of $v$.  Then, the next step in the common methods is to let  $r_k,\,\,\,k=1,2,3$  be  the eigenvalues  that satisfy \eqref{v25}  and the associated eigenvector are given by $\mathcal{V}^{(k)}=\begin{footnotesize}\begin{pmatrix}
 d^{(k)}\\b^{(k)}
 \end{pmatrix}\end{footnotesize}$
  with
 \begin{align}
 d^{(1)}&=\begin{footnotesize}\begin{pmatrix}\frac{1}{k}
 	\vspace{2mm}\\
 	\frac{1}{k}\, r_1
 	\vspace{2mm}\\0
 	\end{pmatrix}\end{footnotesize},\qquad d^{(2)}=\begin{footnotesize}\begin{pmatrix}\frac{2\,\mu_{\rm c}  }{(\mu_{\rm e} +\mu_{\rm c} )\,k}r_2
 	\vspace{2mm}\\
 	-\frac{2\,\mu_{\rm c}  }{(\mu_{\rm e} +\mu_{\rm c} )\,k}
 	\vspace{2mm}\\- {\rm i}\, (r_2^2+{c}_{t})
 	\end{pmatrix}\end{footnotesize},\qquad d^{(3)}=\begin{footnotesize}\begin{pmatrix}\frac{2\,\mu_{\rm c}  }{(\mu_{\rm e} +\mu_{\rm c} )\,k}\,r_3
 	\vspace{2mm}\\
 	-\frac{2\,\mu_{\rm c}  }{(\mu_{\rm e} +\mu_{\rm c} )\,k}
 	\vspace{2mm}\\- {\rm i}\, (r_3^2+{c}_{t})
 	\end{pmatrix}\end{footnotesize},
\\
 b^{(1)}&=\begin{footnotesize}\begin{pmatrix}2\,\mu_{\rm e}  \, r_1
 \vspace{2mm}	\\
 	(2\,\mu_{\rm e}  +\lambda_{\rm e} ) \,r_1^2+\lambda_{\rm e} 
 \vspace{2mm}	\\0
 \end{pmatrix}\end{footnotesize},\quad b^{(2)}=\begin{footnotesize}\begin{pmatrix}\frac{2\,\mu_{\rm c}  }{(\mu_{\rm e} +\mu_{\rm c} )}[\lambda_{\rm e} +{c}_{l}\,(2\,\mu_{\rm e}  +\lambda_{\rm e} )]	\vspace{2mm}\\
 	-\frac{4\,\mu_{\rm c} \, \mu_{\rm e} }{(\mu_{\rm e} +\mu_{\rm c} )}\,r_2
 	\vspace{2mm}\\-{\rm i}\,k\,{\mu_{\rm e}\,L_{\rm c}^2}\,\gamma\, r_2\,(r_2^2+{c}_{t})
 	\end{pmatrix}\end{footnotesize},\quad b^{(3)}=\begin{footnotesize}\begin{pmatrix}\frac{2\,\mu_{\rm e}  }{(\mu_{\rm e} +\mu_{\rm c} )}[\lambda_{\rm e} +{c}_{l}(2\,\mu_{\rm e}  +\lambda_{\rm e} )]	\vspace{2mm}\\
 	-\frac{4\,\mu_{\rm c} \, \mu_{\rm e} }{(\mu_{\rm e} +\mu_{\rm c} )}\,r_3
 	\vspace{2mm}\\-{\rm i}\,k\,{\mu_{\rm e}\,L_{\rm c}^2}\,\gamma\, r_3\,(r_3^2+{c}_{t})
 	\end{pmatrix}\end{footnotesize}.\notag
 \end{align}
 
  Without  loss of  generality, we may assume that the $r_k$ are distinct and, in consequence,     $\mathcal{V}^{(k)}$ are linearly independent\footnote{ If $r_k$ are not distinct, then we would find $d^{(k)}$ eigenvectors associated to  $r_k$ which are independent.} \cite{ChiritaGhiba3}.
  Then a general solution of \eqref{04} having a proper decay is taken in the form
 \begin{align}
 z=\sum_{k=1}^{3}q_{k }\, d^{(k)} e^{{\rm i}\,r_k\,k\, x_2},
 \end{align}
 where $q_k$ are constants to be determined by the boundary condition $\eqref{11}_2$, i.e., we will have the system 
 \begin{align}
 \sum_{k=1}^{3}q_k \,b^{(k)}=\mathbf{B}\,q=0,
 \end{align}
 where \begin{align}
 \mathbf{B}&=(b^{(1)}\,|\,b^{(2)}\,|\,b^{(3)}),\qquad \qquad 
 b^{(1)}=p_k {\mathbf{T}} d^{(k)}+{\mathbf{R}}^T d^{(k)},\qquad \qquad
 q=(q_1,q_2,q_3)^T.\notag
 \end{align}
 We have nonzero solutions $q$ if and only if
 \begin{align}
 \det \mathbf{B}=0.
 \end{align}
 We remark that the  expression of $\mathbf{B}$ contains $v$, so the condition
 $
 \det \mathbf{B}=0
 $
 is in fact a condition to determine $v$ and it is called {\bf the secular equation}.  Its explicit form $\forall\ v\in[0,\widehat{v})$ was determined in \cite{ChiritaGhiba3} in the following form
 \begin{align}
 s(v) & \equiv  \displaystyle\sqrt{P_{3}(v)\left(  C_{2}(v)+C_{3}%
 	(v)+2\sqrt{P(v)}\right)  }4\,\mu_{\rm e}^{2}-  \left(  C_{3}(v)+\sqrt{{P(v)}}\right)  [\lambda_{\rm e} -(2\,\mu_{\rm e}  +\lambda_{\rm e})\,C_{1}(v)]^{2}=0 ,
 \label{v50}%
 \end{align}
 where
 \begin{align}\label{v26}
 P(v)&=P_{2}(v)-C_{1}(v)\,C_{2}(v)-C_{1}(v)\,C_{3}(v)
 \end{align}
 and the notations \eqref{defpC} are used.
 In \cite[Eq. (4.14)]{ChiritaGhiba3} the following result is established
\begin{theorem} If the constitutive coefficients satisfy
 \begin{align}
 \mu_{\rm e}-\mu_{\rm c}+\lambda_{\rm e} >0, \qquad \quad  \mu_{\rm e} +\mu_{\rm c} >0, \qquad  \quad    \alpha_1+\alpha_2>0 \qquad (\textrm{these conditions imply} \  2\,\mu_{\rm e}+\lambda_{\rm e} >0)
 \end{align}
then, the secular equation  \eqref{v50} has an admissible solution, i.e. a solution $v$ such that
$0\leq v<\widehat{v}$.
 \end{theorem}

 \begin{remark}
 	\begin{itemize}
 	\item[]
 \item
 The very important aspect in analysing  the secular equation is to prove that it has an admissible solution, i.e., there exists at least one solution $v$  
 such that
$
 0\le v<\widehat{v},\notag
$
 since otherwise the constructed wave ansatz does not satisfies the asymptotic  decay condition 
  \eqref{04}. This is the essential condition in modelling seismic waves and it will validate or not the entire approach used for the construction of the solution. 
  \item
  As always in the modelling process, the uniqueness of the desired solution is also a very important aspect, because if the uniqueness is not clearly stated then the question of which solution has to be effectively chosen arises.
\end{itemize}
  \end{remark}
  However, in almost all studies concerning the propagation of seismic waves in generalized theories of solid mechanics, the complete study of the existence and uniqueness problem of an admissible solution of the corresponding secular equation is  often left  unsolved. This is also true for the case of the linear Cosserat theory. This question is  completely settled by the present paper.

  \section{The new secular equation. Existence and uniqueness \\ of Rayleigh waves }\label{NSE}\setcounter{equation}{0}
 
 \subsection{Derivation and matrix analysis of the algebraic Riccati equation}
 The main ingredient of the method used by Fu and Mielke \cite{fu2002new} is to  look at \eqref{n11} as  an initial value problem and to search for a solution in the form 
 \begin{align}\label{08}
 y(x_2)=e^{-k\,x_2 \,\boldsymbol{\mathcal{E}}}y(0),
 \end{align}
 where $\boldsymbol{\mathcal{E}}\in\mathbb{C}^{3 \times 3}$  is to be determined. On substituting \eqref{08} into  \eqref{x8}, we get
 \begin{align}\label{09}
 [\boldsymbol{\mathcal{T}}\boldsymbol{\mathcal{E}}^2- {\rm i}\, (\boldsymbol{\mathcal{R}}+\boldsymbol{\mathcal{R}}^T)\boldsymbol{\mathcal{E}}-\boldsymbol{\mathcal{Q}}+\ k^2v^2 {\id}]\, y(x_2)=0, \qquad\qquad \qquad  (-\boldsymbol{\mathcal{T}}\boldsymbol{\mathcal{E}}+ {\rm i}\, \,\boldsymbol{\mathcal{R}}^T)\,y(0)=0.
 \end{align}
 It is clear that in order to have a proper decay the eigenvalues of $\boldsymbol{\mathcal{E}}$ have to be such that their real part is positive. We anticipate and we mention that this will be the case if \,\,$0\le v<\widehat{v}$, as we will see in Theorem \ref{thmielke1}.  
 
 For the linear Cosserat model,  we introduce the so called  \textbf{surface impedance matrix} 
 \begin{align}\label{10}
\boldsymbol{\boldsymbol{\mathcal{M}}}=-(-	\boldsymbol{\mathcal{T}}\boldsymbol{\mathcal{E}}+ {\rm i}\,\boldsymbol{\mathcal{R}}^T)\qquad \iff \qquad \boldsymbol{\mathcal{E}}=\boldsymbol{\mathcal{T}}^{-1}(\boldsymbol{\mathcal{M}}+ {\rm i}\, \,\boldsymbol{\mathcal{R}}^T).
 \end{align}
 
 It seems that  this matrix was first introduced  by Ingebrigsten and Tonning \cite{ingebrigtsen1969elastic} for the classical elastic model,  by Mielke and Sprenger \cite{mielke1998quasiconvexity} on a topic indirectly connected to the surface wave problem, and then by Fu and Mielke \cite{mielke2004uniqueness,fu2002new} in order to prove the existence and the uniqueness of the surface-wave speed for linear anisotropic elastic materials. We may argue the utility of this replacement of $\boldsymbol{\mathcal{E}}$ to $\boldsymbol{\mathcal{M}}$ in order to convert the equation $\eqref{09}_{2}$ into an  equation for a Hermitian matrix $\boldsymbol{\mathcal{M}}$. On substituting $\eqref{10}_2$ into $\eqref{09}_1$, we obtain

  \begin{footnotesize}
 {\begin{align}\label{s17}
 \{\boldsymbol{\mathcal{T}}\boldsymbol{\mathcal{E}}- {\rm i}\, (\boldsymbol{\mathcal{R}}+\boldsymbol{\mathcal{R}}^T)\}\boldsymbol{\mathcal{E}}-\boldsymbol{\mathcal{Q}}+k^2 v^2 {\id}=0,\notag\\
 %\{\boldsymbol{\mathcal{T}}\, \boldsymbol{\mathcal{T}}^{-1}(\boldsymbol{\mathcal{M}}+ {\rm i}\,\boldsymbol{\mathcal{R}}^T)- {\rm i}\,\boldsymbol{\mathcal{R}}- {\rm i}\,\boldsymbol{\mathcal{R}}^T\}\boldsymbol{\mathcal{T}}^{-1}(\boldsymbol{\mathcal{M}}+iR^T)-\boldsymbol{\mathcal{Q}}+k^2 v^2 {\id}=0,\notag\\
 (\boldsymbol{\mathcal{M}}+ {\rm i}\,\boldsymbol{\mathcal{R}}^T)\boldsymbol{\mathcal{T}}^{-1}(\boldsymbol{\mathcal{M}}+ {\rm i}\, \,\boldsymbol{\mathcal{R}}^T)- {\rm i}\,\boldsymbol{\mathcal{R}}\boldsymbol{\mathcal{T}}^{-1}(\boldsymbol{\mathcal{M}}+ {\rm i}\,\boldsymbol{\mathcal{R}}^T)- {\rm i}\,\boldsymbol{\mathcal{R}}^T\boldsymbol{\mathcal{T}}^{-1}(\boldsymbol{\mathcal{M}}+ {\rm i}\,\boldsymbol{\mathcal{R}}^T)-\boldsymbol{\mathcal{Q}}+k^2 v^2 {\id}=0,\\
%\boldsymbol{\boldsymbol{\mathcal{M}}}\boldsymbol{\mathcal{T}}^{-1}(\boldsymbol{\mathcal{M}}+ {\rm i}\, \,\boldsymbol{\mathcal{R}}^T)+ {\rm i}\,\boldsymbol{\mathcal{R}}^T\boldsymbol{\mathcal{T}}^{-1}(\boldsymbol{\mathcal{M}}+ {\rm i}\,\boldsymbol{\mathcal{R}}^T)- {\rm i}\,\boldsymbol{\mathcal{R}}\boldsymbol{\mathcal{T}}^{-1}(\boldsymbol{\mathcal{M}}+ {\rm i}\,\boldsymbol{\mathcal{R}}^T)-{\rm i}\,\boldsymbol{\mathcal{R}}^T\boldsymbol{\mathcal{T}}^{-1}(\boldsymbol{\mathcal{M}}+ {\rm i}\,\boldsymbol{\mathcal{R}}^T)-\boldsymbol{\mathcal{Q}}+k^2 v^2 {\id}=0,\notag\\
\boldsymbol{\boldsymbol{\mathcal{M}}}\boldsymbol{\mathcal{T}}^{-1}(\boldsymbol{\mathcal{M}}+ {\rm i}\,\boldsymbol{\mathcal{R}}^T)- {\rm i}\,\boldsymbol{\mathcal{R}}\boldsymbol{\mathcal{T}}^{-1}(\boldsymbol{\mathcal{M}}+ {\rm i}\,\boldsymbol{\mathcal{R}}^T)-\boldsymbol{\mathcal{Q}}+k^2\, v^2 \,{\id}=0.\notag
 \end{align}}
 \end{footnotesize}
 Hence,  in terms of the surface impedance matrix, equations  \eqref{09} become
 \begin{align}\label{12}
 (\boldsymbol{\mathcal{M}}- {\rm i}\,\boldsymbol{\mathcal{R}})\boldsymbol{\mathcal{T}}^{-1}(\boldsymbol{\mathcal{M}}+ {\rm i}\, \mathcal{R^T})-\boldsymbol{\mathcal{Q}}+k\, v^2\,{\id}=0, \qquad\boldsymbol{\boldsymbol{\mathcal{M}}}\, y(0)=0.
 \end{align}
 
  We call equation $\eqref{12}_1$  \textbf{algebraic Riccati equation} for the linear Cosserat model.  For the classical anisotropic model, the same form of this equation is present in the paper by Mielke  and Sprenger \cite{mielke1998quasiconvexity}  for $v=0$ and in the papers by Fu and Mielke \cite{mielke2004uniqueness,fu2002new} in the case of general $v$, see also the earlier work by Biryukov \cite{biryukov1985impedance}. For this reason, our expectation is that the entire approach presented  by Fu and Mielke \cite{mielke2004uniqueness,fu2002new} should be suitable in the Cosserat theory, too.
  
 Since we are interested in a nontrivial solution  $y$, we impose $y(0)
\neq0$, so that the matrix $\boldsymbol{\mathcal{M}}$ has to satisfy
 \begin{align}\label{x16}
 {\rm det}\,\boldsymbol{\mathcal{M}}=0.
 \end{align}
 	The  equation \eqref{x16} is called  \textbf{secular equation} for the linear Cosserat model in terms of the \textbf{impedance matrix} $\boldsymbol{\mathcal{M}}$.
 
 In this point of our analysis we do not have additional informations about the properties of the impedance matrix $\boldsymbol{\mathcal{M}}$. Looking at the equations \eqref{12} and \eqref{x16}  we remark that  equation \eqref{12}$_1$ leads to a mapping $v\mapsto\boldsymbol{\boldsymbol{\mathcal{M}}}_v$, where $\boldsymbol{\mathcal{M}}_v$ is the solution of \eqref{12}$_1$ for a fixed $v$. If it is possible to construct this mapping, then equation \eqref{x16}  becomes the equation which determines the wave speed  $v$.  However, the construction of this mapping is possible only if we are very careful with the following aspects. The first yet unresolved problem is the knowledge of the domain of those $v$ for which  \eqref{12}$_1$ admits a unique solution. The second question is: does  the solution of the other equation \eqref{x16}, with $\boldsymbol{\mathcal{M}}_v$ expressed as function of $v$, belong (if it exists) to the domains of those $v$ for which  \eqref{12}$_1$ admits a  solution? The last but not the least  important aspect is that it is not sufficient to prove that there is a solution $\boldsymbol{\mathcal{M}}_v$ of the Riccati equation \eqref{12}, since an acceptable $\boldsymbol{\mathcal{M}}_v$ has to be such that $\text{Re(spec}\,\boldsymbol{\mathcal{E}}{\rm )}$ is positive, 	where ``$\text{\rm Re\,(spec\,}\boldsymbol{\mathcal{E}}{\rm )}$" means the ``real part of spectra of $\boldsymbol{\mathcal{E}}$".
 
 Similar arguments as for the classical linear  anisotropic elastic materials lead us to the following
 \begin{lemma}\label{lemmaGH2} If  the constitutive coefficients satisfy the conditions 	\eqref{d12}
 	and   $0\leq v<\widehat{v}$, then
 	the eigenvalues of any solution of \eqref{12} have a non-zero real part.
 \end{lemma}
\begin{proof}
	Let $\boldsymbol{\mathcal{E}}$ be a solution of \eqref{12}, $\lambda $ be an eigenvalue and $a$ an associated eigenvector, i.e. $\boldsymbol{\mathcal{E}}\, d=\lambda \, d$. Hence, if $\boldsymbol{\mathcal{E}}$ is a solution of \eqref{12}, then 
	\begin{align}
	\{\lambda ^2 \,\boldsymbol{\mathcal{T}}- {\rm i}\, \,\lambda \, (\boldsymbol{\mathcal{R}}+\boldsymbol{\mathcal{R}}^T)\boldsymbol{\mathcal{E}}-\boldsymbol{\mathcal{Q}}+ k^2 v^2{\id}\}\, d=0,
	\end{align}
	and $\lambda$ is a solution of the equation
	\begin{align}
	\det \{\lambda\, ^2 \boldsymbol{\mathcal{T}}- {\rm i}\, \,\lambda \,(\boldsymbol{\mathcal{R}}+\boldsymbol{\mathcal{R}}^T)\boldsymbol{\mathcal{E}}-\boldsymbol{\mathcal{Q}}+ k^2 v^2{\id}\}=0.
	\end{align}
	Thus, $r= {\rm i}\, \, \lambda $ is solution of the equation
	\begin{align}
	\det \{-r^2 \boldsymbol{\mathcal{T}}-r\,(\boldsymbol{\mathcal{R}}+\boldsymbol{\mathcal{R}}^T)\boldsymbol{\mathcal{E}}-\boldsymbol{\mathcal{Q}}+ k^2 v^2{\id}\}=0\quad \Leftrightarrow\quad \det \{r^2 \boldsymbol{\mathcal{T}}+r\,(\boldsymbol{\mathcal{R}}+\boldsymbol{\mathcal{R}}^T)\boldsymbol{\mathcal{E}}+\boldsymbol{\mathcal{Q}}- k^2 v^2{\id}\}=0.
	\end{align}
	Assuming that there exists an eigenvalue $\lambda $ of a solution $\boldsymbol{\mathcal{E}}$ of \eqref{12} with a non-zero real part, it follows that the equation 
	\begin{align}
	\det \{r^2 \boldsymbol{\mathcal{T}}+r\,(\boldsymbol{\mathcal{R}}+\boldsymbol{\mathcal{R}}^T)\boldsymbol{\mathcal{E}}+\boldsymbol{\mathcal{Q}}- k^2 v^2{\id}\}=0
	\end{align}
	admits a non-real solution. But we have shown in Proposition \ref{lemmaGH} that this is not possible if the wave speed is smaller than the limiting speed $\widehat{v}$.
	\end{proof}
 Under the assumption  that $\boldsymbol{\mathcal{T}}$ and $\boldsymbol{\mathcal{Q}}$ are symmetric and positive definite matrices, many aspects from the above discussions are purely mathematical questions, and they are not specific to Cosserat elastic materials. Notice that we were careful to obtain a specific formulation such that many mathematical results obtained by Mielke and Fu \cite{mielke2004uniqueness,fu2002new} can be directly applied in the Cosserat theory, too. For instance, since $\boldsymbol{\mathcal{T}}$,  $\boldsymbol{\mathcal{Q}}$ and $\widehat{\id}$ are symmetric real matrices,  the following result established in \cite{mielke2004uniqueness,fu2002new} remains valid in the  framework of the Cosserat elastic materials
  \begin{theorem}\label{thmielke1} If  the constitutive coefficients satisfy the conditions 	\eqref{d12}
  	and   $0\leq v<\widehat{v}$, the matrix problem 
 	\begin{align}\label{13E}
 	\boldsymbol{\mathcal{T}}\boldsymbol{\mathcal{E}}^2- {\rm i}\, \,(\boldsymbol{\mathcal{R}}+\boldsymbol{\mathcal{R}}^T)\boldsymbol{\mathcal{E}}-\boldsymbol{\mathcal{Q}}+ k^2 v^2{\id}=0,\qquad \text{\rm Re\,(spec\,}\boldsymbol{\mathcal{E}}{\rm )}>0,
 	\end{align}
  has a unique solution for $\boldsymbol{\mathcal{E}}_v$  and the corresponding matrix $\boldsymbol{\mathcal{M}}_v$ obtained from $\eqref{10}_2$ is Hermitian.\qedhere
 	\end{theorem}
  \begin{proof}
 	The reader may consult the paper by Fu and Mielke \cite{fu2002new}  or the Appendix, where we have rewritten the proof in our notation.
 \end{proof}

Hence,  Theorem \ref{thmielke1} proves that for all $0\leq v<\widehat{v}$ there exists a unique solution $\boldsymbol{\mathcal{M}}_v$ of the Riccati equation \eqref{12}, defined by  the unique matrix $\boldsymbol{\mathcal{E}}_v$ indicated  in Theorem \ref{thmielke1}.

 Thus, we know that we can consider the mapping which associates to each  $0\leq v<\widehat{v}$ the Hermitian matrix $\boldsymbol{\mathcal{M}}_v$ satisfying the equation \eqref{12}. However, the pair $(v,\boldsymbol{\boldsymbol{\mathcal{M}}}_v)$ must  also be a solution of the secular equation  \eqref{x16}, i.e., 
 \begin{align}\label{x16n}
 {\rm det}\,\boldsymbol{\mathcal{M}}_v=0.
 \end{align}
 Therefore, we have to check if  after replacing  the solution $\boldsymbol{\mathcal{M}}_v$ of \eqref{12} into  \eqref{x16n}, the resulting equation will lead to a unique wave speed belonging to the interval $[0,\widehat{v})$. Without explaining this aspect and having a clear answer, the analyses would be incomplete. 

 In the context of linear anisotropic elasticity   Barnett and Lothe \cite{barnett1985free} have proven that the secular equation has a unique subsonic solution and their proof is based  on the following properties of the  impedance matrix $\boldsymbol{\mathcal{M}}_v$ from the classical linear elastic model, that hold for a subsonic solution  $v$:
 	\begin{enumerate}
 		\item The surface impedance matrix $\boldsymbol{\mathcal{M}}_v$ is Hermitian,
 		\item The matrix $\frac{d\boldsymbol{\boldsymbol{\mathcal{M}}}_v}{d v}$ is negative definite,
 	\item $\langle\boldsymbol{\boldsymbol{\mathcal{M}}}_v,\id\rangle\geq 0$, and $\langle  w,\boldsymbol{\boldsymbol{\mathcal{M}}}_v\,w\rangle \geq 0$ for all real vectors $w$.
 \end{enumerate}

In the following we show that for subsonic wave speeds  $v$, the  impedance matrix $\boldsymbol{\mathcal{M}}_v$, solution of \eqref{12}, satisfies the above properties 1--3 for isotropic elastic Cosserat materials, too. It is straightforward to prove the second property using  the same arguments as in  the proof given by 
  Fu and Mielke \cite{fu2002new} in the context of  anisotropic classical elastic materials,   since the main problem is in fact a purely mathematical question, independent of the considered theory. Indeed, the proof  remains unchanged in the context of isotropic elastic Cosserat materials considered in our paper
  \begin{theorem}\label{thmielke3}
 	Assume  the constitutive coefficients satisfy the conditions 	\eqref{d12}
 	and   $0\leq v<\widehat{v}$. Let $\boldsymbol{\mathcal{M}}_v$ and $\boldsymbol{\mathcal{E}}_v$ be the same as in the conclusion of Theorem \ref{thmielke1}. Then the matrix $\frac{d\boldsymbol{\boldsymbol{\mathcal{M}}}_v}{dv}$ is negative definite.\qedhere
 \end{theorem}
 \begin{proof}
 	The reader may consult the paper by Fu and Mielke \cite{fu2002new}  or the Appendix, where we have rewritten the proof in our notations.
 \end{proof}
 
  In the following, we  prove that  the impedance matrix $\boldsymbol{\mathcal{M}}_v$ satisfies also the 3rd property. To this aim we  follow again Fu and Mielke's technique \cite{fu2002new} and we show that this method is also  applicable  to the linear isotropic  elastic Cosserat model. Hence,  in order to establish the  3rd property, we first define matrices $\widetilde{\boldsymbol{\mathcal{Q}}}_\theta$,${\boldsymbol{\mathcal{T}}}_\theta$ and ${\boldsymbol{\mathcal{R}}}_\theta\in \mathbb{R}^{3\times 3}$ by
\begin{align}
\begin{bmatrix}
\widetilde{\boldsymbol{\mathcal{Q}}}_\theta  &	\boldsymbol{\mathcal{R}}_\theta 
\vspace{2mm} \\
\boldsymbol{\mathcal{R}}_\theta  ^{\mathbf{T}}&  \boldsymbol{\mathcal{T}}_{\!\! \theta}     
\end{bmatrix}=\begin{bmatrix}
\cos\theta\, \id  &	\sin\theta\, \id\vspace{2mm}\\
-\sin\theta\, \id&  \cos\theta\, \id    
\end{bmatrix}\begin{bmatrix}
\widetilde{\boldsymbol{\mathcal{Q}} } &	\boldsymbol{\mathcal{R}}\vspace{2mm}\\
\boldsymbol{\mathcal{R}}^{\mathbf{T}}&  \boldsymbol{\mathcal{T}}    
\end{bmatrix}\begin{bmatrix}
\cos\theta\, \id  &	-\sin\theta\, \id\vspace{2mm}\\
\sin\theta\, \id&  \cos\theta\, \id    
\end{bmatrix},
\end{align}
where we write  $\widetilde{\boldsymbol{\mathcal{Q}}}=\boldsymbol{\mathcal{Q}}-k\,v^2\,{\id}$, and $\theta$ is an arbitrary angle. These matrices may be seen as the counterpart of $\boldsymbol{\mathcal{T}},\widetilde{\boldsymbol{\mathcal{Q}}}$ and $\boldsymbol{\mathcal{R}}$ and they are obtained by rotation of the old coordinate system about $e_3$ by an angle $\theta$
\begin{align}\label{21}
\boldsymbol{\mathcal{T}}_{\!\! \theta} =\cos^2\theta\,\boldsymbol{\mathcal{T}}-\sin\theta\cos\theta\,(\boldsymbol{\mathcal{R}}+\boldsymbol{\mathcal{R}}^T)+\sin^2\theta\,\widetilde{\boldsymbol{\mathcal{Q}}},\notag\vspace{2mm}\\
\boldsymbol{\mathcal{R}}_\theta  =\cos^2\theta\,\boldsymbol{\mathcal{R}}+\sin\theta\cos\theta
\,(\boldsymbol{\mathcal{T}}-\widetilde{\boldsymbol{\mathcal{Q}}})-\sin^2\theta\,\boldsymbol{\mathcal{R}}^T,\vspace{2mm}\\
\widetilde{\boldsymbol{\mathcal{Q}}}_\theta=\cos^2\theta\,\widetilde{\boldsymbol{\mathcal{Q}}}+\sin\theta\cos\theta\,(\boldsymbol{\mathcal{R}}+\boldsymbol{\mathcal{R}}^T)+\sin^2\theta\,\boldsymbol{\mathcal{T}}.\notag
\end{align}
The new matrices  $\boldsymbol{\mathcal{T}}$ and $\boldsymbol{\mathcal{Q}}$ remain symmetric and 
 $\widetilde{\boldsymbol{\mathcal{Q}}}_\theta$, $\boldsymbol{\mathcal{T}}_{\!\! \theta} $ and $\boldsymbol{\mathcal{R}}_\theta  $ are periodic in $\theta$ with periodicity $\pi$ and
\begin{align}\label{RTQ}
\widetilde{\boldsymbol{\mathcal{Q}}}_\theta(\theta+\frac{\pi}{2})=\boldsymbol{\mathcal{T}}_{\!\! \theta} ,\qquad \quad\boldsymbol{\mathcal{R}}_\theta(\theta+\frac{\pi}{2})=-\boldsymbol{\mathcal{R}}_\theta  ^T,\qquad\quad  \boldsymbol{\mathcal{T}}_{\!\! \theta}(\theta+\frac{\pi}{2})=\widetilde{\boldsymbol{\mathcal{Q}}}_\theta.
\end{align}
Moreover, according to the definition of the limiting speed, regarding Proposition \ref{lemmaGH}  and  Proposition  \ref{propQp},
the limiting velocity $\widehat{v}$ is in fact the lowest velocity for which the matrices $\widetilde{\boldsymbol{\mathcal{Q}}}_\theta$ and $T(\theta)$  become singular for some angle $\theta$ and  $\widetilde{\boldsymbol{\mathcal{Q}}}_\theta$ is positive definite for $0\leq v<\widehat{v}$. In view of \eqref{RTQ} so is $\boldsymbol{\mathcal{T}}_{\!\! \theta} $, too. 
	Thus by the definition of the limiting speed $\widehat{v}$ both $\widetilde{\boldsymbol{\mathcal{Q}}}_\theta$ and $\boldsymbol{\mathcal{T}}_{\!\! \theta} $ are positive definite or positive semi-definite depending on  $\theta$ (for $v=\widehat{v}$ there is at least one $\theta$ at which $\boldsymbol{\mathcal{T}}(\theta)$ has an eigenvalue 0, and likewise $\widetilde{\boldsymbol{\mathcal{Q}}}_\theta$).

We shall use $\boldsymbol{\mathcal{E}}$ exclusively to denote the unique solution of \eqref{09}, and likewise  we define $ \boldsymbol{\mathcal{E}}_{\!\theta}   $ to be the unique solution of the matrix problem\footnote{This solution exists since $\boldsymbol{\mathcal{T}}_{\!\! \theta} $ and $\widetilde{\boldsymbol{\mathcal{Q}}}_\theta$ are positive definite. Then,  Theorem \ref{thmielke1} will be used since it is valid for $v=0$, too.}
\begin{align}\label{17}
\boldsymbol{\mathcal{T}}_{\!\! \theta} \,  \boldsymbol{\mathcal{E}}^2_\theta   - {\rm i}\, (\boldsymbol{\mathcal{R}}_\theta  +\boldsymbol{\mathcal{R}}_\theta  ^T) \,\boldsymbol{\mathcal{E}}_{\!\theta}   -\widetilde{\boldsymbol{\mathcal{Q}}}_\theta=0\qquad \qquad\qquad\text{Re}\,\, \text{spec}\,  \boldsymbol{\mathcal{E}}_{\!\theta}   >0,
\end{align}
while the matrix $\boldsymbol{\mathcal{M}}_\theta$ has a form such that 
\begin{align}
 \boldsymbol{\mathcal{E}}_{\!\theta}   =\boldsymbol{\mathcal{T}}^{-1}_{\!\! \theta}\, (\boldsymbol{\mathcal{M}}_\theta + {\rm i}\, \,\boldsymbol{\mathcal{R}}_\theta  ^T).
\end{align}
In terms of the matrix $\boldsymbol{\mathcal{M}}_\theta$ the equation \eqref{17} reads
\begin{align}\label{eqMt}
(\boldsymbol{\mathcal{M}}_\theta - {\rm i}\, \,\boldsymbol{\mathcal{R}}_\theta )\boldsymbol{\mathcal{T}}^{-1}_{\!\! \theta}(\boldsymbol{\mathcal{M}}_\theta + {\rm i}\, \,\boldsymbol{\mathcal{R}}_\theta  ^T)-\boldsymbol{\mathcal{Q}}^T_\theta=0.
\end{align}
The following results established in \cite{fu2002new}  remain valid in our framework, too.
\begin{theorem}
	\label{Mielketh}
\begin{itemize}Assume  the constitutive coefficients satisfy the conditions 	\eqref{d12}
	and   $0\leq v<\widehat{v}$. Then,
	\item[i)] The Hermitian matrix $\boldsymbol{\mathcal{M}}_\theta $ defined above  is independent of $\theta$.
\item [ii)]  Denoting by $\boldsymbol{\mathcal{M}}$ and $\boldsymbol{\mathcal{E}}$ the corresponding values of   $\boldsymbol{\mathcal{M}}_\theta$, and $\boldsymbol{\mathcal{E}}_{\!\theta}$ for $\theta=0$, respectively,  then \linebreak
$
	 \boldsymbol{\mathcal{E}}_{\!\theta}   =(\cos\theta\, \id+ {\rm i}\, \sin\theta \boldsymbol{\mathcal{E}})^{-1}(\cos\theta \,\boldsymbol{\mathcal{E}}+ {\rm i}\, \,\sin\theta \,\id).
$
\item [iii)] 
	$
	\dd\int_{0}^{\pi} 
	 \boldsymbol{\mathcal{E}}_{\!\theta}   \, d\theta=\pi\, \id.
	$
\item[iv)] 
	The unique solution of the algebraic Riccati equation \eqref{12} that satisfies $\text{\rm Re\,[spec}\,(\boldsymbol{\mathcal{T}}^{-1}(\boldsymbol{\mathcal{M}}+{\rm i}\,R^T))]>0$ is given explicitly by
	\begin{align}\label{explMielke}
	\boldsymbol{\mathcal{M}}_v=\mathbf{H}_v^{-1}+ {\rm i}\, \, \mathbf{H}_v^{-1}\, \mathbf{S}_v, \qquad \text{with} \qquad \mathbf{H}_v=\frac{1}{\pi}\dd\int_{0}^{\pi}	\boldsymbol{\mathcal{T}}_{\!\! \theta} ^{-1}\, d\theta, \qquad  \mathbf{S}_v=-\frac{1}{\pi}\int_{0}^{\pi}\boldsymbol{\mathcal{T}}_{\!\! \theta} ^{-1}\boldsymbol{\mathcal{R}}_\theta  ^T\, d\theta.
\end{align}
	\end{itemize}
\end{theorem}
 \begin{proof}
	The reader may consult the paper by Fu and Mielke \cite{fu2002new}  or the Appendix, where we have rewritten the proof in our notations.
\end{proof}
From a computational point of view, the decisive advantage given by the above result is that, since $\boldsymbol{\mathcal{T}}_{\!\! \theta}$ and $\boldsymbol{\mathcal{R}}_\theta$ depend on the wave speed $v$, we obtain the explicit   form of the secular equation $\det \boldsymbol{\mathcal{M}}_v=0$ without a priori knowing the analytical expressions (as function of the wave speed ) of the eigenvalues  that satisfy \eqref{x9}  and the associated eigenvector  $d^{(k)}$, which is the main difficulty in almost all the generalised models, with exception of some models for which this task is straightforward, e.g., classical isotropic linear elasticity \cite{Achenbach,hayes1962note},  or the theory of materials with voids \cite{NunziatoCowin79,ChiritaGhiba2,straughan2008stability} after imposing restrictive conditions upon the constitutive coefficients.  After the secular equation is solved, the task of finding the eigenvalues  that satisfy \eqref{x9}  and the associated eigenvector  $d^{(k)}$ becomes a purely numerical task, avoiding symbolic (analytical) computations.

\subsection{The main result: Existence and uniqueness of Rayleigh waves}\label{MR}

Moreover,  Fu and Mielke's method has another   advantage since we are able to show the existence and uniqueness of a subsonic wave speed, solution of the secular equation, which ensures in the end that there exists an acceptable $\boldsymbol{\mathcal{M}}_v$  such that $\text{Re spec}\,\boldsymbol{\mathcal{E}}$ is positive, i.e. the solution satisfies both the boundary conditions \eqref{03} and the decay conditions \eqref{04}. As we will explain in the following, the matrix $\boldsymbol{\mathcal{M}}_v$ determined by \eqref{explMielke} satisfies the condition 3., i.e. $\tr(\boldsymbol{\boldsymbol{\mathcal{M}}}_v)\geq 0$, and $\langle  w,\boldsymbol{\boldsymbol{\mathcal{M}}}_v\,w\rangle \geq 0$ for all real vectors $w$, and this will imply the existence of a unique subsonic solution of the secular equation. 

\begin{theorem}{\rm {\bf [The main result of this paper]}}
	Assume  the constitutive coefficients satisfy the conditions 	\begin{align}\label{d11wa}
	2\,\mu_{\rm e} +\lambda_{\rm e} >0,\qquad \qquad \mu_{\rm e} >0,\qquad \qquad \mu_{\rm c} >0,\qquad \qquad   \alpha_1+\alpha_2>0,
	\end{align}
	then the secular equation
	\begin{align}
	\det\boldsymbol{\boldsymbol{\mathcal{M}}}_v=0,
	\end{align}
	where $\boldsymbol{\mathcal{M}}_v$ is given by  $\eqref{explMielke}$, has a unique admissible solution $0\leq v<\widehat{v}$. In other words there {\bf exists a unique Rayleigh wave propagating in the Cosserat medium}.
\end{theorem}
\begin{proof}
First, we explain why,	if $\boldsymbol{\mathcal{E}}_v$ solves \eqref{13E},  then the corresponding $\boldsymbol{\mathcal{M}}_v$ obtained from $\eqref{10}_2$ has the following properties
\begin{enumerate}
	\item $\boldsymbol{\mathcal{M}}_v$ is Hermitian,
	\item  $\frac{d\boldsymbol{\boldsymbol{\mathcal{M}}}_v}{d v}$ is negative definite,
	\item  $\tr(\boldsymbol{\boldsymbol{\mathcal{M}}}_v)\geq 0$, and $\langle  w,\boldsymbol{\boldsymbol{\mathcal{M}}}_v\,w\rangle \geq 0$ for all real vectors $w$ for  all $0\leq v\leq  \widehat{v}$,
	\item $\boldsymbol{\mathcal{M}}_v$ is and positive definite for  all $0\leq v< \widehat{v}$.
\end{enumerate}
To this aim, we can use the  arguments explained in \cite[page 13]{mielke2004uniqueness}. Since due to Theorem \ref{thmielke1} we know that  $\boldsymbol{\mathcal{M}}_v$ is Hermitian, $\mathbf{H}_v$ and $\mathbf{S}_v$ are both real matrices and $\mathbf{H}_v$ is symmetric, it follows that $\mathbf{H}_v^{-1}\mathbf{S}_v$ is skew-symmetric. Hence, ${\rm tr}(\boldsymbol{\mathcal{M}}_v)={\rm tr}(\mathbf{H}_v^{-1})$ and  $\bigl\langle\boldsymbol{\boldsymbol{\mathcal{M}}}_v\, w, w\bigr\rangle=\bigl\langle \mathbf{H}_v^{-1}\, w, w\bigr\rangle
$ $\forall\, w\in\mathbb{R}^3$, since $\mathbf{H}_v^{-1}\mathbf{S}_v$ is skew-symmetric. Since $\mathbf{H}_v^{-1}$ is positive semi-definite\footnote{Note that $\mathbf{H}_v$ is not well defined for $v=\widehat{v}$ since for this value of the wave speed there exists an angle $\theta\in[0,\pi]$ such that $\boldsymbol{\mathcal{T}}_{\!\! \theta}$ has a zero eigenvalue. However, $\mathbf{H}_v^{-1}$ is well-defined for the limit $v\to \widehat{v}$ and at this limit it admits a zero eigenvalue, since considering the contrary it follows that $\mathbf{H}_v$ is well-defined for the limit $v\to \widehat{v}$.  This will imply that $\boldsymbol{\mathcal{T}}_{\!\! \theta}^{-1}$ is defined for the limit $v\to \widehat{v}$ and that  $\boldsymbol{\mathcal{T}}_{\!\! \theta}$ and $\boldsymbol{\mathcal{Q}}_\theta$ are positive definite for the limit $v\to \widehat{v}$, too, a fact that is contrary to the property of the limiting speed.} for all $0\leq v\leq \widehat{v}$ and positive definite for  all $0\leq v< \widehat{v}$, it follows that $\boldsymbol{\mathcal{M}}_v$ determined by \eqref{explMielke} satisfies the condition 3 and moreover $\boldsymbol{\mathcal{M}}_v$ is  positive definite for  all $0\leq v< \widehat{v}$. Note that at $v=\widehat{v}$ at least one of the eigenvalues of $\mathbf{H}_v^{-1}$ must vanish.

\begin{figure}[h!]
	\centering
	\includegraphics[width=8cm]{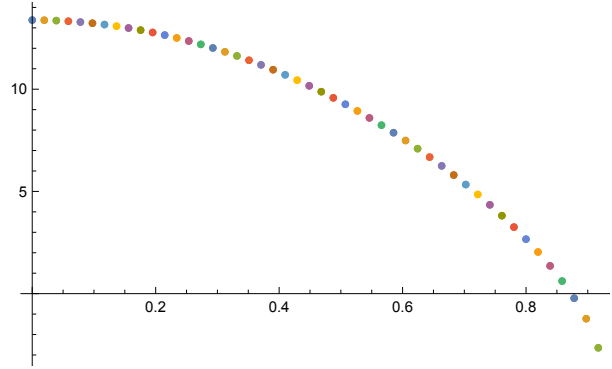}
	\caption{\footnotesize A plot of $\det\boldsymbol{\boldsymbol{\mathcal{M}}}_v$ with respect to the surface waves speed $v$ for the  aluminum-epoxy composite for a set of equidistant values in the interval $[0,\widehat{v})$. This curve illustrates that $\det\boldsymbol{\mathcal{M}}_v$ is  a decreasing function of the wave speed $v$.}
	\label{fig:method}
\end{figure} 

The rest of the proof is  clearly explained in \cite[page 2531]{fu2002new} as in the following.  Since $\frac{d\boldsymbol{\boldsymbol{\mathcal{M}}}_v}{d v}$ is negative definite, the eigenvalues of $\boldsymbol{\mathcal{M}}_v$ are monotone decreasing functions of $v$ defined on  $[0,\widehat{v})$. Let us remark that at $v=0$ the eigenvalues of $\boldsymbol{\mathcal{M}}_0$ are positive, since $\boldsymbol{\mathcal{M}}_v$ is positive definite at $v=0$, that $\det\boldsymbol{\boldsymbol{\mathcal{M}}}_v=\lambda_1\,\lambda_2\,\lambda_3$, where $\lambda_1,\,\lambda_2,\,\lambda_3$ denote the eigenvalues of $\boldsymbol{\mathcal{M}}_v$,  that $\det\boldsymbol{\boldsymbol{\mathcal{M}}}_0>0$ and that the map $v\mapsto \det\boldsymbol{\boldsymbol{\mathcal{M}}}_v$ is monotone decreasing on $[0,\widehat{v})$, too. 
 Thus, there  exists a solution of the secular equation $\det\boldsymbol{\boldsymbol{\mathcal{M}}}_v=0 $ only if an eigenvalue of $\boldsymbol{\mathcal{M}}_v$ decreases at zero at $0<v=v_R<\widehat{v}$. Moreover, if such a $v_R$ exists it is unique, in the sense that only one eigenvalues  of $\boldsymbol{\mathcal{M}}_v$ may decrease to zero for a value $0<v=v_R<\widehat{v}$ of the wave speed, since if two eigenvalues would share this property, then at $v=\widehat{v}$ the matrix $\boldsymbol{\mathcal{M}}_v$ should have  two\footnote{There cannot exist three negative eigenvalues since $\tr(\boldsymbol{\mathcal{M}}_v)\geq 0$.} eigenvalues which are negative which will violate the positive semi-definiteness\footnote{We recall that the positive semi-definiteness of $\mathbf{H}_v^{-1}$ is equivalent to the  positive semi-definiteness of $\boldsymbol{\mathcal{M}}_v$.}   of $\mathbf{H}_v^{-1}$ at $v=\widehat{v}$ since at least one eigenvalue of $\mathbf{H}_v^{-1}$ at $v=\widehat{v}$ is zero. In the same manner we argue that at $v=v_R$ zero there is not a repeated eigenvalue of $\boldsymbol{\mathcal{M}}_v$. We conclude the proof  by pointing out  that there exists a unique $v_R$ such that an eigenvalue of $\boldsymbol{\mathcal{M}}_v$ decreases to zero at $0<v=v_R<\widehat{v}$, and therefore that $\det\boldsymbol{\boldsymbol{\mathcal{M}}}_{v_R}=0 $, while $\det\boldsymbol{\boldsymbol{\mathcal{M}}}_{v}>0 $ for all $0<v<v_R$ and $\det\boldsymbol{\boldsymbol{\mathcal{M}}}_{v}<0 $ for all $v_R<v<\widehat{v}$.
 \end{proof}
 
Here, we  also illustrate the statements from the above paragraph numerically for the   aluminum-epoxy composite considered by Gauthier \cite{Gauthier82} and  Eringen \cite{Eringen99}. Since the integral representation \eqref{explMielke} is an explicit expression for the surface impedance matrix, we simply increase $v$ in small steps from $0$ to  the determined $\widehat{v}$ at every step. Plotting the secular equation $\det\boldsymbol{\mathcal{M}}_v$ with respect to the  wave speed $v$ gives the curve shown in Figure \ref{fig:method}. We observe how simple it is to find an approximation of the wave speed, once the integral representation \eqref{explMielke} is given.

Comparing with the left hand side of the secular equation presented in \cite[Eq. (4.7)]{ChiritaGhiba3} the left hand side of our new secular equation involve a strictly decreasing function, see Figure \ref{fig:method} and \cite[Fig. 1]{ChiritaGhiba3}. Moreover, we have analytically proven that the left hand side defines a decreasing function for all  materials.

{In our analysis we have considered  $L_{\rm c}>0$. Thus,  the results  are
valid for the full Cosserat medium. Part of our results and estimates may be immediately applicable to the case of the reduced Cosserat model (considered in \cite{grekova2009waves,kulesh2009problem}) by simply taking $L_{\rm c}\to 0$. However,  since we do not know explicitly how the analytical form of the solution $v$ of the secular equation depends on $L_{\rm c}$, we are not able to obtain the form of the solution $v$ for the reduced Cosserat model by simply letting $L_{\rm c}\to 0$ in the expression of the solution $v$ for the full Cosserat model. But,  all our  calculations may be adapted to the case $L_{\rm c}=0 $.}
 
\section{Numerical implementation}\label{Num1}\setcounter{equation}{0}

In this section we consider $k=1\, {\rm mm}^{-1}$ and the constitutive coefficients obtained by  Gauthier \cite{Gauthier82}, see also \cite{Eringen99}, for  aluminum-epoxy composite.  According to \cite[pages 164-165]{Eringen99}, in the Eringen notations \cite{hassanpour2017micropolar}, for such a material we have
\begin{align}
\lambda_{\rm Eringen} &= 7.59 \, \text{GPa}, 
\qquad \qquad 
\mu_{\rm Eringen} = 1.89 \, \text{GPa}, \qquad \qquad \kappa_{\rm Eringen}= 0.00788\, \mu_{\rm Eringen}, \\
\rho &= \frac{\kappa_{\rm Eringen}}{0.0067}, \qquad \qquad \ 
j_{\rm Eringen}\, = 0.0196 \, {\rm mm}^2, \qquad \quad 
\gamma_{\rm Eringen} = 7.11\, j_{\rm Eringen}\,\, \mu_{\rm Eringen},\notag
\end{align}
 while in our notation the same onstitutive coefficients are represented by
\begin{align}
&\lambda_{\rm e} =7.59\, \text{GPa},\qquad \qquad \mu_{\rm c} =\frac{\kappa}{2}=0.0074466\, \text{GPa},\qquad \qquad\ \  \mu_{\rm e} =\mu_{\rm Eringen}+\frac{\kappa}{2}=1.89745 \, \text{GPa},\\
&\rho=2.22287\, \frac{g}{{\rm mm}^3},\qquad\ \   j\,\mu_{\rm e}\,\tau_{\rm c}^2\,=0.0196 \,{\rm mm}^2,\qquad\qquad  {\mu_{\rm e}\,L_{\rm c}^2}\,\gamma=0.263383 \, \text{GPa}\times {\rm mm}.\notag
\end{align}

According to the results presented in the previous sections, for known numerical values of all constitutive parameters, we identify an algorithm to approximate numerically the problem of the propagation of seismic waves:

	\begin{enumerate}
	\item[I.] 	\texttt{A first algorithm:}\begin{enumerate}
			\item[	\texttt{Step 1:}] 	\texttt{Identify the limiting speed $\widehat{v}$.}     Since the constitutive coefficients satisfy \eqref{d11n}, using Proposition \ref{lemmaCG1} and after direct substitution in its analytical form, we find the value of the limiting speed $\boldsymbol{\widehat{v}=0.925507}$. Let us notice that it defers from that established in \cite{ChiritaGhiba3}, because after reverification we have remarked that all the numerical computations in \cite{ChiritaGhiba3} are done for other values of the constitutive parameters, which do not match those proposed   by Gauthier \cite{Gauthier82} for aluminum-epoxy composite. Due to a common misunderstanding of the notations, in \cite{ChiritaGhiba3} it is considered that ${\mu_{\rm e}\,L_{\rm c}^2}\,\gamma= 5.8546$ and $J=\rho\, j\,\mu_{\rm e}\,\tau_{\rm c}^2\,=0.4357$.  However, the numerical calculation  given in \cite{ChiritaGhiba3} are correct, but for the  before mentioned  values of ${\mu_{\rm e}\,L_{\rm c}^2}\,\gamma$ and $J$. 
			We have repeated  the calculations in the sense of the approach given in \cite{ChiritaGhiba3} and we have found  a complete agreement with the value of the limiting speed considered in the current paper, i.e., $\boldsymbol{\widehat{v}=0.925507}$.
			\item[	\texttt{Step 2:}] \texttt{Find the solution of the secular equation.}   We have applied formula \eqref{explMielke} to compute $\boldsymbol{\mathcal{M}}_v$ on a set of 50 values in the interval $[0,\widehat{v})$. We have to  compute numerically the needed integrals from \eqref{explMielke} since we did not reach the symbolic values of them. For these values we have computed $\det\boldsymbol{\boldsymbol{\mathcal{M}}}_v$, too. 
			We consider these values of $\det\boldsymbol{\boldsymbol{\mathcal{M}}}_v$ for a set of equidistant values in $[0,\widehat{v})$ and we use  interpolation to find an approximation function $v\mapsto f(v)$ of the function  $v\mapsto \det\boldsymbol{\boldsymbol{\mathcal{M}}}_v$, on $[0,\widehat{v})$. 
			
			\begin{figure}[h!]
				\centering
				\includegraphics[width=10cm]{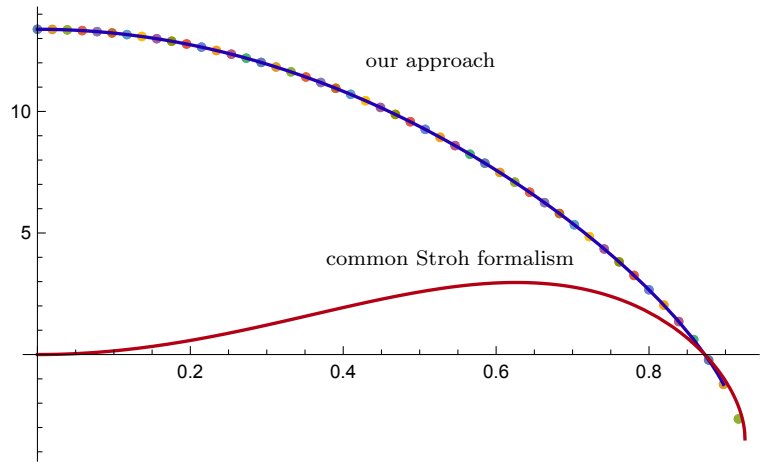}
			\put(-165,75){\footnotesize common Stroh formalism}
			\put(-150,150){\footnotesize our approach}	\caption{\footnotesize The plots of the approximation of $v\mapsto \det\boldsymbol{\boldsymbol{\mathcal{M}}}_v$ (blue curve) and of $v\mapsto s(v)$ (red curve), where $s(v)$ from \eqref{v50} defines the secular equations in the Stroh formalism approach \cite{ChiritaGhiba3}, with respect to the surface waves speed $v$ for the  aluminum-epoxy composite. These two function have the same (unique) root in the interval $[0,\widehat{v})$.}
				\label{fig:method2}
			\end{figure} 
		
			After that, we find the root  of $f$ on $[0,\widehat{v})$. Since this root given by mathematical software does not lead to a vanishing  $\det\boldsymbol{\boldsymbol{\mathcal{M}}}_v$,  we are looking in the neighbourhood of this root for a $v$ such that $\det\boldsymbol{\boldsymbol{\mathcal{M}}}_v$ is close to zero. Such a value is $\boldsymbol{v_R=0.8730352}$. This is the approximate value  we have found for the wave speed, i.e. the approximate solution of our  secular equation. We work with 7 decimals since the values obtained for the approximations of $\det\boldsymbol{\boldsymbol{\mathcal{M}}}_v$ are very sensitive to small changes of $v$. 
			
			We have considered the same coefficients in the secular equation established in \cite{ChiritaGhiba3} and we have remarked, see Figure \ref{fig:method2}, that both  functions defining the secular equation in our form and the form given in \cite{ChiritaGhiba3}, respectively, vanish in the same value in the interval $[0, \widehat{v})$.  We have approximated also the solution of the secular equation $s(v)=0$ from \cite{ChiritaGhiba3} and we have found the approximation of the corresponding admissible wave speed to be $0.87296$, which is not far from the approximate wave speed given by our secular equation $\det\boldsymbol{\boldsymbol{\mathcal{M}}}_v=0$, i.e., $\boldsymbol{v_R=0.8730352}$.  We mention we cannot obtain the precise value of $v_R$ and that a new numerical strategy or a better mathematical software  could lead to a better accuracy.

			\item[	\texttt{Step 3:}] \texttt{Construct the amplitudes and the solution}. For $v=v_R$, we find $y(0)$ as solution of the algebraic system \eqref{12}. Then we 	 construct $y$ from \eqref{08}. Finally, we 
			construct the solution $\mathcal{U}$ from \eqref{x5}. 
			We find the approximate solution of \eqref{12}$_2$ to be
			\begin{align}
			y(0)=\left(
			\begin{array}{c}
			\varsigma\\
			-1.66731\, {\rm i}\, 	\varsigma \\
			-0.0120298\, {\rm i}\, 	\varsigma \\
			\end{array}
			\right), \qquad 	\varsigma\in \mathbb{C}.
			\end{align}
		Since $\boldsymbol{\mathcal{M}}_{v_R}$  is approximated by $\boldsymbol{\mathcal{M}}_{v_R}=\left(
		\begin{array}{ccc}
		1.01413 &  -0.608012 \, {\rm i}\, & 0.00513355 \, {\rm i}\, \\
		0.608012 \, {\rm i}\, & 0.365425 & -0.0463072 \\
		 -0.00513355 \, {\rm i}\, & -0.0463072 & 6.00576 
		\end{array}
		\right)$, the matrix $\boldsymbol{\mathcal{E}}_{v_R}=\boldsymbol{\mathcal{T}}^{-1}(\boldsymbol{\mathcal{M}}_{v_R}+ {\rm i}\, \,\boldsymbol{\mathcal{R}}^T)$ is numerically approximated by
		\begin{align}
		\boldsymbol{\mathcal{E}}_{v_R}=\left(
		\begin{array}{ccc}
		1.18322 & 0.282484 \, {\rm i}\,  & 0.00598908 \, {\rm i}\,  \\
		0.785418 \, {\rm i}\,  & 0.0712845 & -0.00904174 \\
		0.00706743 \, {\rm i}\,  & -0.00766037 & 0.993447 
		\end{array}
		\right)
		\end{align}
		and, using \eqref{08}, the function $y(x_2)$ is determined. Then, using $y(x_2):=\begin{footnotesize}\begin{pmatrix} \frac{1}{\sqrt{\rho}} &0	&0
		\\
		0&\frac{1}{\sqrt{\rho}} 
		&0
		\\0
		& 0 & \frac{1}{\sqrt{\rho\,j\,\mu_{\rm e}\,\tau_{\rm c}^2\,}} \end{pmatrix}\end{footnotesize}^{-1}\,z(x_2)$ we find $z(x_2)$. 
		In the end, from \eqref{x5} we find the solution  \begin{align}
		\mathcal{U}(x_1,x_2,t)=\begin{footnotesize}\begin{pmatrix}u_1(x_1,x_2,t)
		\\
		u_2(x_1,x_2,t)
		\\\vartheta _3(x_1,x_2,t)
		\end{pmatrix}\end{footnotesize}={\rm Re}\left[\begin{footnotesize}\begin{pmatrix} z_1(x_2)
		\\
		z_2(x_2)
		\\{\rm i}\,z_3(x_2)
		\end{pmatrix}\end{footnotesize} e^{ {\rm i}\, k\, (  x_1-vt)}\right],
		\end{align}
		represented in Figure \ref{u1u2}.

		\end{enumerate}
	
	\begin{figure}[h!]
		\centering 
		\begin{subfigure}{.31\textwidth}
			\includegraphics[width=\linewidth]{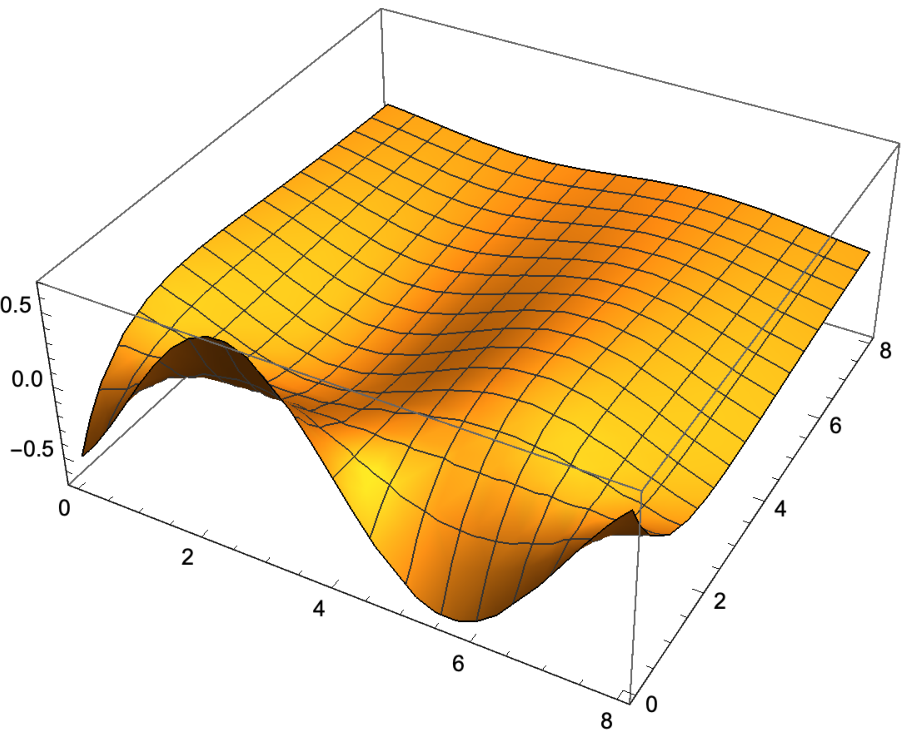}
			\caption{Plot of the $u_1$-component of the displacement.}
		\end{subfigure}\quad\  
		\begin{subfigure}{.31\textwidth}
			\includegraphics[width=\linewidth]{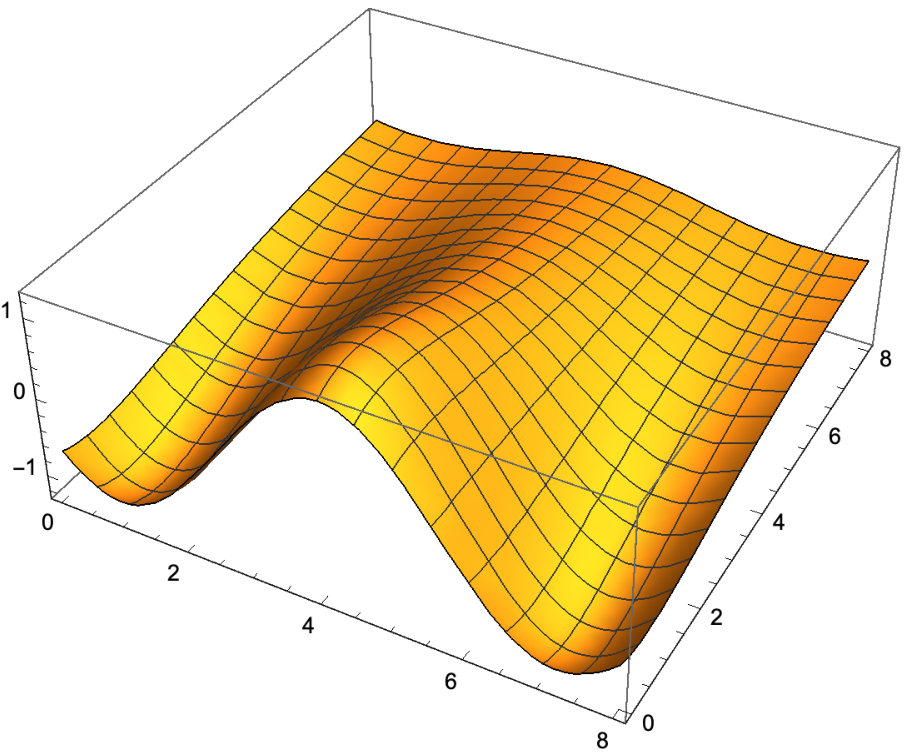}
			\caption{Plot of the $u_2$-component of the displacement.}
		\end{subfigure}\quad \ 
		\begin{subfigure}{.31\textwidth}
			\includegraphics[width=\linewidth]{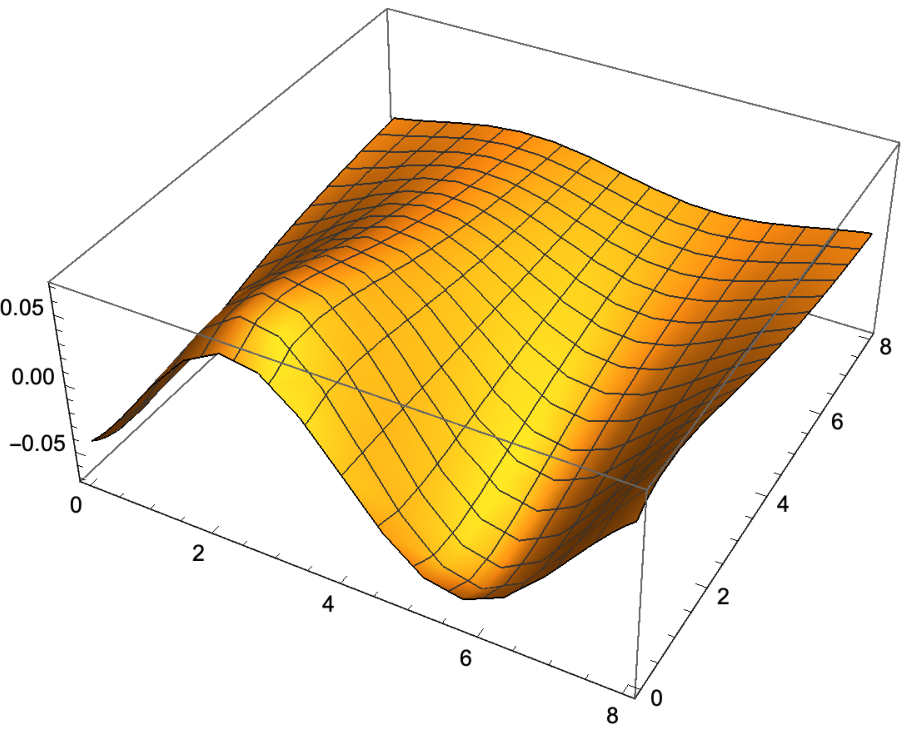}
			\caption{Plot of $\vartheta_3$-component of the micro-rotation vector.}
		\end{subfigure}
		\caption{The plot of the solution at time $t=1$ and for the choice $	\varsigma={\rm i}\,$.}\label{u1u2}
	\end{figure}
\end{enumerate}

In Figure \ref{kCe} we present the dependence of the wave speed on the wave number in the framework of the linear Cosserat theory and classical linear elasticity. In contrast to  classical elasticity where the wave speed does not depend on the wave number, in the linear Cosserat model there is a dispersion curve describing such a dependency. Moreover, it seems that the speed goes asymptotically to a finite  value of the wave speed for large values of the wave number. The computations for linear classical elasticity as limit case of the approach which is done in this paper for linear Cosserat elasticity are presented in more details in  Section \ref{Classic}.

\begin{figure}[h!]
	\centering 
		\includegraphics[width=0.5\linewidth]{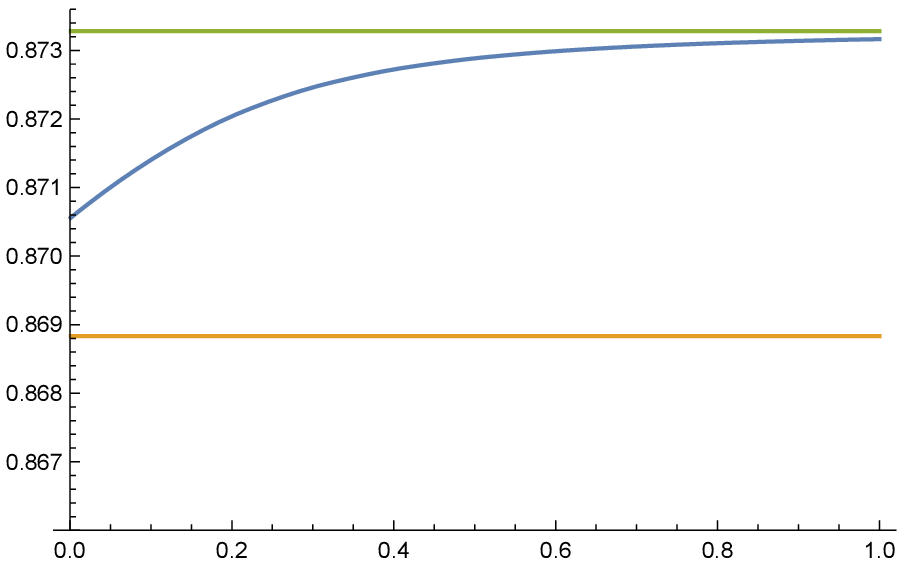}
			\put(-220,150){\footnotesize the upper bound value of the wave speed in linear Cosserat elasticity}
		\put(-150,65){\footnotesize linear classical elasticity (no dispersion)}
			\put(-150,120){\footnotesize linear Cosserat elasticity (dispersion curve)}
				\put(-220,150){\footnotesize the upper bound value of the wave speed in linear Cosserat elasticity}
				\put(-120,-14){\footnotesize {\bf wave number}}
				\put(-260,50){\begin{turn}{90}\footnotesize {\bf wave speed}\end{turn}}
	\caption{\footnotesize The dependence of the wave speed on the wave number for aluminum-epoxy composite ($\lambda_{\rm e} =7.59\, \text{GPa}, \ \mu_{\rm c} =\frac{\kappa}{2}=0.0074466\, \text{GPa},\  \mu_{\rm e} =\mu_{\rm Eringen}+\frac{\kappa}{2}=1.89745 \, \text{GPa},\ 
		\rho=2.22287\, \frac{g}{{\rm mm}^3}, j\,\mu_{\rm e}\,\tau_{\rm c}^2\,=0.0196 \,{\rm mm}^2, \ {\mu_{\rm e}\,L_{\rm c}^2}\,\gamma=0.263383 \, \text{GPa}\times {\rm mm} $) vs. for a linear elastic material having the parameters ($\lambda =7.59\, \text{GPa},\   \mu =1.89745 \, \text{GPa},
			\rho=2.22287\, \frac{g}{{\rm mm}^3}.$) For the considered material parameters, the speed goes to $0.87327989$ for large values of the wave speed.}\label{kCe} 
\end{figure}

\begin{figure}[h!]
	\centering 
	\includegraphics[width=0.55\linewidth]{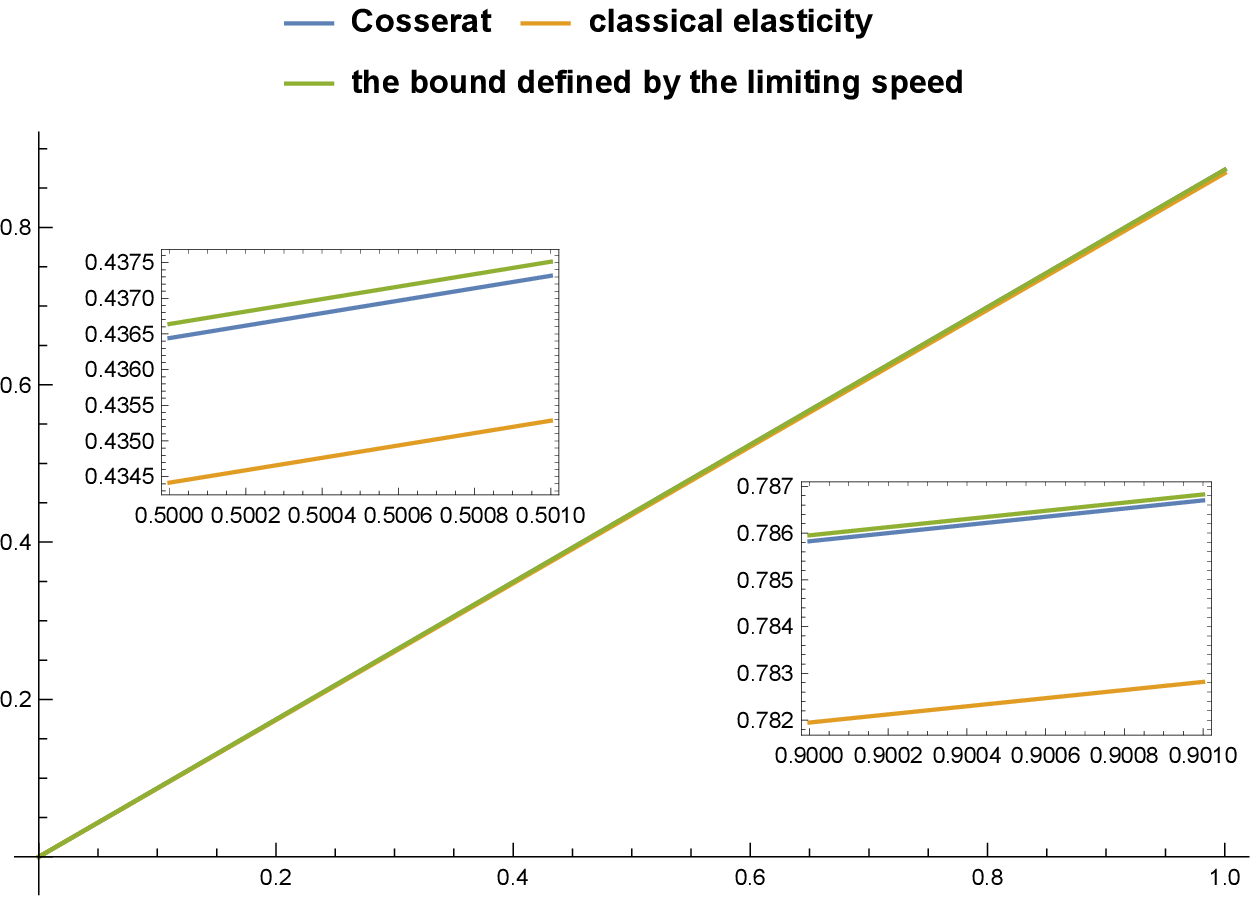}
	\put(-120,-14){\footnotesize {\bf wave number}}
	\put(-280,50){\begin{turn}{90}\footnotesize {\bf wave frequency}\end{turn}}
	\caption{\footnotesize The dependence of the wave frequency on the wave number for aluminum-epoxy composite ($\lambda_{\rm e} =7.59\, \text{GPa}, \ \mu_{\rm c} =\frac{\kappa}{2}=0.0074466\, \text{GPa},\  \mu_{\rm e} =\mu_{\rm Eringen}+\frac{\kappa}{2}=1.89745 \, \text{GPa},\ 
		\rho=2.22287\, \frac{g}{{\rm mm}^3}, j\,\mu_{\rm e}\,\tau_{\rm c}^2\,=0.0196 \,{\rm mm}^2, \ {\mu_{\rm e}\,L_{\rm c}^2}\,\gamma=0.263383 \, \text{GPa}\times {\rm mm} $) vs. for a linear elastic material having the parameters ($\lambda =7.59\, \text{GPa},\   \mu =1.89745 \, \text{GPa},
		\rho=2.22287\, \frac{g}{{\rm mm}^3}.$) From the magnified windows and since the orange curve is linear (expressing the dependency on the wave frequency as function of the wave number in classical elasticity) we observe that the wave frequency as function of the wave number is not linear and we have dispersion. However, there are no band-gap. It is clear that dispersion curve lie under (on the right side of) all the dispersion curves of the real plane waves, since the $\omega(k)=k \, v(k)$ with $0\leq v(k)<\widehat{v}(k)$ and the limiting speed $\widehat{v}(k)$ represent the minimum slope of all the possible dispersion curve at $k$, for the real plane waves.} \label{wk} 
\end{figure}

{From Figure \ref{wk} we observe that the frequency increases as function of the wave number since the  wave speed increases as function of the wave number (see Figure \ref{kCe}) and $\omega(k)=k\, v(k)$. Moreover, since $0\leq v(k)<\widehat{v}$, the mapping $k\to v(k)$ is bounded. In addition, the mapping $k\to v(k)$ is smooth. Therefore, the range of the mapping $k\to \omega(k)$ is $[0,\infty)$ and we conclude that for this material there is no band-gap. Even if, at a first look from Figure \ref{wk}, it seems  that the wave frequency depends linearly as function of the wave number (as in classical elasticity), from the magnified picture we see that this is not true. This is also clear since the wave speed is not constant as function of the wave number and therefore the group velocity $\frac{d\omega(k)}{dk}$ is not constant. Even if we have seen this aspect only for this material (we do not have an analytical proof yet), we extrapolate this remark and we conjecture that, in Cosserat elasticity  the wave frequency is not linear as function of the wave number  and we have dispersion. This can also be observed from  Figure \ref{dwk} in which the group velocity is given.}

\begin{figure}[h!]
	\centering 
	\includegraphics[width=0.55\linewidth]{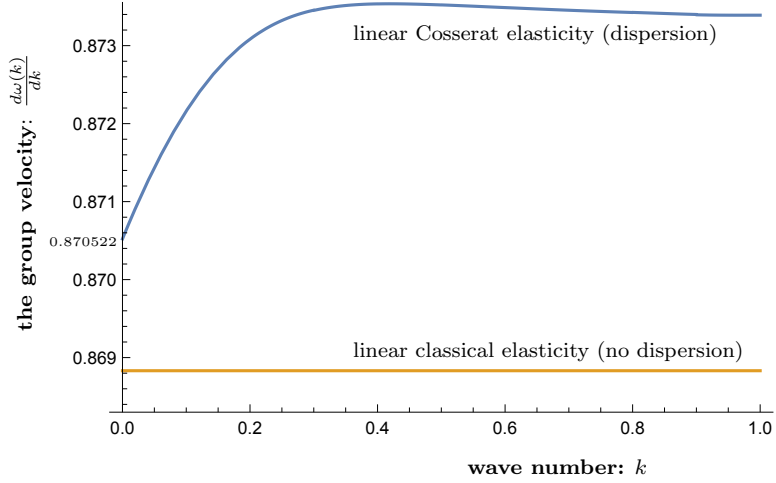}
\put(-290,40){\begin{turn}{90}	\footnotesize {\bf the group velocity}: $\frac{d\omega(k)}{dk}$\end{turn}}
\put(-120,-14){	\footnotesize {\bf wave number:} $k$}
\put(-160,30){\footnotesize linear classical elasticity (no dispersion)}
\put(-160,150){\footnotesize linear Cosserat elasticity (dispersion)}
\put(-275,72){\tiny 0.870522}
	\caption{\footnotesize The dependence of the wave frequency on the wave number for aluminum-epoxy composite ($\lambda_{\rm e} =7.59\, \text{GPa}, \ \mu_{\rm c} =\frac{\kappa}{2}=0.0074466\, \text{GPa},\  \mu_{\rm e} =\mu_{\rm Eringen}+\frac{\kappa}{2}=1.89745 \, \text{GPa},\ 
		\rho=2.22287\, \frac{g}{{\rm mm}^3}, j\,\mu_{\rm e}\,\tau_{\rm c}^2\,=0.0196 \,{\rm mm}^2, \ {\mu_{\rm e}\,L_{\rm c}^2}\,\gamma=0.263383 \, \text{GPa}\times {\rm mm} $) vs. for a linear elastic material having the parameters ($\lambda =7.59\, \text{GPa},\   \mu =1.89745 \, \text{GPa},
		\rho=2.22287\, \frac{g}{{\rm mm}^3}.$) From the magnified windows and since the orange curve is linear (expressing the dependency on the wave frequency as function of the wave number in classical elasticity) we observe that the wave frequency as function of the wave number is not linear and we have dispersion. However, there are no band-gap.} \label{dwk} 
\end{figure}

We will not explicit it here, but another numerical approach is also possible. We indicate it in the following:
\texttt{
\begin{enumerate}
	\item[II.] The second algorithm:\begin{enumerate}
	\item[Step 1:] Identify the limiting speed $\widehat{v}$: from \eqref{limCh} in our case, since this limiting speed was previously found by Chiri\c t\u a and Ghiba in \cite{ChiritaGhiba3}.
	\item[Step 2:] Consider  the hermitian matrix  $\boldsymbol{\mathcal{M}}_v$ in the form\begin{align}
	\boldsymbol{\mathcal{M}}_v=	\left({\begin{array}{ccc}\boldsymbol{\mathcal{M}}_1&
		\boldsymbol{\mathcal{M}}_3+ {\rm i}\,\boldsymbol{\boldsymbol{\mathcal{M}}}_4	&\boldsymbol{\mathcal{M}}_5+ {\rm i}\,\boldsymbol{\boldsymbol{\mathcal{M}}}_6\\\boldsymbol{\mathcal{M}}_3- {\rm i}\,\boldsymbol{\boldsymbol{\mathcal{M}}}_4&\boldsymbol{\boldsymbol{\mathcal{M}}}_2&\boldsymbol{\mathcal{M}}_7+ {\rm i}\,\boldsymbol{\boldsymbol{\mathcal{M}}}_8\\\boldsymbol{\mathcal{M}}_5- {\rm i}\, \,\boldsymbol{\mathcal{M}}_6&\boldsymbol{\mathcal{M}}_7- {\rm i}\,\boldsymbol{\boldsymbol{\mathcal{M}}}_8&\boldsymbol{\mathcal{M}}_9
		\end{array}}\right) 
	\end{align} 
	 and solve both the  Riccati equation and the secular equation 
	\begin{align}\label{76}
	(\boldsymbol{\mathcal{M}}- {\rm i}\,\boldsymbol{\mathcal{R}})\boldsymbol{\mathcal{T}}^{-1}(\boldsymbol{\mathcal{M}}+ {\rm i}\,\boldsymbol{\mathcal{R}}^T)-\boldsymbol{\mathcal{Q}}+k^2\, v^2\,{\id}=0,\qquad \det\boldsymbol{\mathcal{M}}_v=0.
	\end{align}
	With the condition $\eqref{76}_2$ this matrix has two positive eigenvalues and  one is a zero \linebreak eigenvalue. Writing separately the real and imaginary part of the equation yields \linebreak nine quadratic equation for the $\boldsymbol{\mathcal{M}}_i\,\, (i=1,2...9)$. Together with the secular equation we have an algebraic system of 10 equations for the 10 unknowns $\boldsymbol{\mathcal{M}}_1,\boldsymbol{\mathcal{M}}_2,...,\boldsymbol{\mathcal{M}}_9, v$. 
	\begin{enumerate}\item Solve this nonlinear system. \item Choose only those solutions $(v,\boldsymbol{\boldsymbol{\mathcal{M}}}_{v})$ for which $0<v<\widehat{v}$, $\tr\boldsymbol{\boldsymbol{\mathcal{M}}}_{v}>0$ and  $\boldsymbol{\mathcal{M}}_{v}$ is positive definite. 
	\item Take $(v_R,\boldsymbol{\boldsymbol{\mathcal{M}}}_{v_R})$ as solution of the nonlinear system of equations. It has to be  \linebreak only one solution $(v_R,\boldsymbol{\boldsymbol{\mathcal{M}}}_{v_R})$ with these properties.
	\item To avoid possible numerical errors, check again  that the corresponding matrix $\boldsymbol{\mathcal{E}}$ \linebreak computed with \eqref{10}  satisfies $\text{\rm Re\,spec\,}\boldsymbol{\mathcal{E}}>0$. 
\end{enumerate}
	\item[Step 3:] For $v=v_R$ and $\boldsymbol{\mathcal{M}}=\boldsymbol{\mathcal{M}}_{v_R}$, find $y(0)$ as solution of the algebraic system \eqref{12} and  \linebreak
construct $y(x_2)$ from \eqref{08}.
	\item[Step 4:]	 Construct the solution $\mathcal{U}$ from \eqref{x5}. 
	\end{enumerate}
\end{enumerate}}

\section{From linear Cosserat theory to classical linear elasticity: A consistency check and comparison of the results}\label{Classic}\setcounter{equation}{0}

In this section we rediscover the results from classical linear elasticity as a limit case of the results obtained in linear Cosserat theory. First, we remark that  
\begin{align}\label{vancond}  \qquad \|\mathrm{D}\, \axl  \,\mathbf{A}\|^2=0\qquad 
\text{
implies that}
\qquad 
\axl\, \,\mathbf{A}=\textrm{constant} \quad \textrm{in} \quad \Omega, 
\end{align}
which corresponds to a time dependent rigid (macroscopic) movement of the entire body. In addition, under Dirichlet homogeneous boundary conditions on $ \axl  \,\mathbf{A}$ we find that $ \axl  \,\mathbf{A}$ vanishes in the entire body at any time.

In fact, by looking at the expression of the total energy, we observe that the energy due to the microrotation is
\begin{align}&\,\rho\,j\,\mu_{\rm e}\,\tau_{\rm c}^2\,\|(\axl\, \,\mathbf{A})_{,t}\|^2+\mu_{\rm c} \|\skw(\mathrm{D}u -\mathbf{A})\|^2
\\&+\frac{\mu_{\rm e}L_{\rm c}^2}{2}\left[{a_1 } \|\dev\,\sym (\mathrm{D}\, \axl  \,\mathbf{A})\|^2+{a_2}\|\skw (\mathrm{D}\, \axl  \,\mathbf{A})\|^2+ \frac{4\,a_3}{3}\,[\tr(\mathrm{D}\, \axl  \,\mathbf{A})]^2\right] 
\end{align}
and that  the situation described by \eqref{vancond}  is formally equivalent to the case 
\begin{align}
L_{\rm c}\to \infty,
\end{align}
 since the total energy  has to remain finite.
  Moreover, assuming also that 
 \begin{align}
\mu_{\rm c}\to 0 \qquad\text{and}\qquad  \tau_{\rm c}\to 0, \ \ \text{too},
 \end{align}
 the entire energy due to the microrotation vanishes and we are  back in the framework of classical linear elasticity, i.e., the  elastic energy density is
 \begin{align}
 W(\mathrm{D}u )=&\,  \mu_{\rm e} \,\lVert \dev_3\, \sym\,\mathrm{D}u \rVert ^{2}+\frac{2\,\mu_{\rm e} +3\,\lambda_{\rm e} }{6}\left[\mathrm{tr} \left(\mathrm{D}u \right)\right]^{2}.\qquad \qquad \qquad \qquad \qquad 
 \end{align}
 
In the following we explain that all the results concerning the propagation of the seismic wave in the framework of Cosserat theory are still valid for Cosserat couple modulus $\mu_{\rm c}\to 0$,
and we rediscover the results concerning the propagation of the seismic waves from classical linear elasticity without assuming further conditions on the internal length scale $L_{\rm c}$ or on  the internal time scale $\tau_{\rm c}$ (as was the case for real plane waves), since the domain of the admissible speeds is already bounded and cannot reach large values defined by large values of the internal  length scale $L_{\rm c}$ or small values of the internal time scale $\tau_{\rm c}$.

Letting $\mu_{\rm c}\to 0$, the matrices involved in the Riccati equation become
\begin{align}
\boldsymbol{\mathcal{T}}:=&\,k^2\,\begin{footnotesize}\begin{pmatrix} \frac{\mu_{\rm e} }{\rho} &0	&0
\\
0& \frac{2\,\mu_{\rm e}  +\lambda_{\rm e} }{\rho}
&0
\\0
& 0 & \frac{{\mu_{\rm e}\,L_{\rm c}^2}\,\gamma}{\rho\,j\,\mu_{\rm e}\,\tau_{\rm c}^2\,} \end{pmatrix}\end{footnotesize}, \qquad \qquad 
\boldsymbol{\mathcal{R}}:= k\,\begin{footnotesize}\begin{pmatrix} 0& k\,\frac{\lambda_{\rm e} }{\rho}& 0
\\
k\,\frac{\mu_{\rm e}}{\rho}&0 
&0
\\0  
& 0 & 0 \end{pmatrix}\end{footnotesize},\notag\\
\boldsymbol{\mathcal{Q}}:=&\begin{footnotesize}\begin{pmatrix} k^2\frac{2\,\mu_{\rm e}  +\lambda_{\rm e}}{\rho}-k^2\,v^2&0	&0
\\
0&k^2\,\frac{\mu_{\rm e} }{\rho}- k^2\,v^2
&0
\\0
&0 &k^2\,\frac{{\mu_{\rm e}\,L_{\rm c}^2}\,\gamma}{\rho\,j\,\mu_{\rm e}\,\tau_{\rm c}^2\,}- k^2\,v^2 \end{pmatrix}\end{footnotesize}.
\end{align}
Thus, we have
\begin{align}
\boldsymbol{\mathcal{T}}_{\!\! \theta}&=
\begin{footnotesize}\begin{pmatrix}
\frac{k^2 \,\mu _{\rm e} \cos ^2(\theta )}{\rho }+\frac{k^2 \sin ^2(\theta ) \left(\lambda _{\rm e}+2 \,\mu _{\rm e}-\rho \, v^2\right)}{\rho } & -\frac{1}{2}\sin (2\,\theta )  \left(\frac{k^2 \lambda _{\rm e}}{\rho }+\frac{k^2 \,\mu _{\rm e}}{\rho }\right) & 0 \\
-\frac{1}{2}\sin (2\,\theta )  \left(\frac{k^2 \lambda _{\rm e}}{\rho }+\frac{k^2 \,\mu _{\rm e}}{\rho }\right) & \frac{k^2 \cos ^2(\theta ) \left(\lambda _{\rm e}+2 \,\mu _{\rm e}\right)}{\rho }+\frac{k^2 \sin ^2(\theta ) \left(\mu _{\rm e}-\rho \, v^2\right)}{\rho } & 0 \\
0 & 0 & \frac{ L_{\rm c}^2 \, \gamma   k^2  }{ \tau_{\rm c}^2 \, j  \rho }-\sin^2 (\theta ){k^2   v^2} \\
 \end{pmatrix}\end{footnotesize},\\
\boldsymbol{\mathcal{R}}_\theta&=
\begin{footnotesize}\begin{pmatrix}
\frac{1}{2}\sin (2\,\theta ) \left(\frac{k^2 \,\mu _{\rm e}}{\rho }-\frac{k^2 \left(\lambda _{\rm e}+2 \,\mu _{\rm e}-\rho \, v^2\right)}{\rho }\right) & \frac{k^2 \lambda _{\rm e} \cos ^2(\theta )}{\rho }-\frac{k^2 \,\mu _{\rm e} \sin ^2(\theta )}{\rho } & 0 \\
\frac{k^2 \,\mu _{\rm e} \cos ^2(\theta )}{\rho }-\frac{k^2 \lambda _{\rm e} \sin ^2(\theta )}{\rho } & \frac{1}{2}\sin (2\,\theta )  \left(\frac{k^2 \left(\lambda _{\rm e}+2 \,\mu _{\rm e}\right)}{\rho }-\frac{k^2 \left(\mu _{\rm e}-\rho \, v^2\right)}{\rho }\right) & 0 \\
0 & 0 & \frac{1}{2}\sin (2\,\theta ) k^2   v^2 \\
 \end{pmatrix}\end{footnotesize},\notag\\
\boldsymbol{\mathcal{Q}}_\theta&=
\begin{footnotesize}\begin{pmatrix}
\frac{k^2 \,\mu _{\rm e} \sin ^2(\theta )}{\rho }+\frac{k^2 \cos ^2(\theta ) \left(\lambda _{\rm e}+2 \,\mu _{\rm e}-\rho \, v^2\right)}{\rho } & \frac{1}{2}\sin (2\,\theta )  \left(\frac{k^2 \lambda _{\rm e}}{\rho }+\frac{k^2 \,\mu _{\rm e}}{\rho }\right) & 0 \\
\frac{1}{2}\sin (2\,\theta )  \left(\frac{k^2 \lambda _{\rm e}}{\rho }+\frac{k^2 \,\mu _{\rm e}}{\rho }\right) & \frac{k^2 \sin ^2(\theta ) \left(\lambda _{\rm e}+2 \,\mu _{\rm e}\right)}{\rho }+\frac{k^2 \cos ^2(\theta ) \left(\mu _{\rm e}-\rho \, v^2\right)}{\rho } & 0 \\
0 & 0 & \frac{ L_{\rm c}^2 \, \gamma   k^2 }{ \tau_{\rm c}^2 \, j  \rho }-\cos^2 (\theta )\,k^2  v^2 
 \end{pmatrix}\end{footnotesize}.\notag
\end{align}
Consequently, we deduce 
\begin{align}
\boldsymbol{\mathcal{T}}_{\!\! \theta} ^{-1}&=\begin{footnotesize}\begin{pmatrix}
\frac{\rho  \left(\lambda _{\rm e} \cos ^2(\theta )+\frac{1}{2} \mu _{\rm e} (\cos (2 \theta )+3)-\rho \, v^2 \sin ^2(\theta )\right)}{k^2 \left(\rho \, v^2 \sin ^2(\theta )-\mu _{\rm e}\right) \left(-\lambda _{\rm e}-2 \,\mu _{\rm e}+\rho \, v^2 \sin ^2(\theta )\right)} & \frac{\rho  \sin (\theta ) \cos (\theta ) \left(\lambda _{\rm e}+\mu _{\rm e}\right)}{k^2 \left(\rho \, v^2 \sin ^2(\theta )-\mu _{\rm e}\right) \left(-\lambda _{\rm e}-2 \,\mu _{\rm e}+\rho \, v^2 \sin ^2(\theta )\right)} & 0 \\
\frac{\rho  \sin (\theta ) \cos (\theta ) \left(\lambda _{\rm e}+\mu _{\rm e}\right)}{k^2 \left(\rho \, v^2 \sin ^2(\theta )-\mu _{\rm e}\right) \left(-\lambda _{\rm e}-2 \,\mu _{\rm e}+\rho \, v^2 \sin ^2(\theta )\right)} & \frac{\rho  \left(\lambda _{\rm e} \sin ^2(\theta )-\frac{1}{2} \mu _{\rm e} (\cos (2 \theta )-3)+\rho  \left(-v^2\right) \sin ^2(\theta )\right)}{k^2 \left(\rho \, v^2 \sin ^2(\theta )-\mu _{\rm e}\right) \left(-\lambda _{\rm e}-2 \,\mu _{\rm e}+\rho \, v^2 \sin ^2(\theta )\right)} & 0 \\
0 & 0 & \frac{ \tau_{\rm c}^2 \, j  \rho }{k^2 \left( L_{\rm c}^2 \, \gamma   - \tau_{\rm c}^2 \, j  \rho \, v^2 \sin ^2(\theta )\right)} \\
 \end{pmatrix}\end{footnotesize},\\\notag
\boldsymbol{\mathcal{T}}_{\!\! \theta} ^{-1}\boldsymbol{\mathcal{R}}_\theta  ^T&=
\begin{footnotesize}\begin{pmatrix}
\sin (2 \theta ) \left(\frac{\mu _{\rm e}}{\lambda _{\rm e}+2 \,\mu _{\rm e}-\rho \, v^2 \sin ^2(\theta )}+\frac{\rho \, v^2-2 \,\mu _{\rm e}}{2 \,\mu _{\rm e}-2 \rho \, v^2 \sin ^2(\theta )}\right) & \frac{2 \,\mu _{\rm e}^2 \cos ^2(\theta )+\lambda _{\rm e} \left(\mu _{\rm e} \cos (2 \theta )+\rho \, v^2 \sin ^2(\theta )\right)}{\left(\rho \, v^2 \sin ^2(\theta )-\mu _{\rm e}\right) \left(-\lambda _{\rm e}-2 \,\mu _{\rm e}+\rho \, v^2 \sin ^2(\theta )\right)} & 0 \\
\frac{\mu _{\rm e} \left(\lambda _{\rm e} \cos (2 \theta )+\sin ^2(\theta ) \left(\rho \, v^2-2 \,\mu _{\rm e}\right)\right)}{\left(\rho \, v^2 \sin ^2(\theta )-\mu _{\rm e}\right) \left(-\lambda _{\rm e}-2 \,\mu _{\rm e}+\rho \, v^2 \sin ^2(\theta )\right)} & \frac{\sin (2 \theta ) \left(\mu _{\rm e} \left(2 \lambda _{\rm e}+2 \,\mu _{\rm e}+\rho \, v^2\right)-\rho ^2 v^4 \sin ^2(\theta )\right)}{2 \left(\rho \, v^2 \sin ^2(\theta )-\mu _{\rm e}\right) \left(-\lambda _{\rm e}-2 \,\mu _{\rm e}+\rho \, v^2 \sin ^2(\theta )\right)} & 0 \\
0 & 0 & \frac{1}{\frac{ L_{\rm c}^2 \, \gamma   \csc (\theta ) \sec (\theta )}{ \tau_{\rm c}^2 \, j  \rho \, v^2}-\tan (\theta )} \\
 \end{pmatrix}\end{footnotesize}.
\end{align}
Hence, the matrices $\mathbf{H}_v$ and $\mathbf{S}_v$ from \eqref{explMielke} are
\begin{align}
\mathbf{H}_v&=\begin{footnotesize}\begin{pmatrix}
\frac{\frac{1}{\sqrt{\frac{\rho  \left(\frac{\lambda _{\rm e}+2 \,\mu _{\rm e}}{\rho }-v^2\right)}{\lambda _{\rm e}+2 \,\mu _{\rm e}}}}-\sqrt{\frac{\rho  \left(\frac{\mu _{\rm e}}{\rho }-v^2\right)}{\mu _{\rm e}}}}{k^2 v^2} & 0 & 0 \\
0 & \frac{\frac{1}{\sqrt{\frac{\rho  \left(\frac{\mu _{\rm e}}{\rho }-v^2\right)}{\mu _{\rm e}}}}-\sqrt{\frac{\rho  \left(\frac{\lambda _{\rm e}+2 \,\mu _{\rm e}}{\rho }-v^2\right)}{\lambda _{\rm e}+2 \,\mu _{\rm e}}}}{k^2 v^2} & 0 \\
0 & 0 & \frac{ \tau_{\rm c}^2 \, j  \rho }{2 k^2 \sqrt{ L_{\rm c}^2 \, \gamma    \tau_{\rm c}^2 \, \rho\,  j \left(\frac{ L_{\rm c}^2 \, \gamma   }{ \tau_{\rm c}^2 \, j  \rho }-v^2\right)}} \\
 \end{pmatrix}\end{footnotesize},\notag\\
&=\begin{footnotesize}\begin{pmatrix}
\frac{\frac{{c}_{l}}{\sqrt{{c}_{l}^2-v^2}}-\frac{\sqrt{{c}_{t}^2-v^2}}{{c}_{t}}}{k^2 v^2} & 0 & 0 \\
0 & \frac{\frac{{c}_{t}}{\sqrt{{c}_{t}^2-v^2}}-\frac{\sqrt{{c}_{l}^2-v^2}}{{c}_{l}}}{k^2 v^2} & 0 \\
0 & 0 & \frac{{c}_{m}}{2 k^2 \sqrt{{c}_{m}^2-v^2}} \\
 \end{pmatrix}\end{footnotesize},\\
\mathbf{S}_v&=\begin{footnotesize}\begin{pmatrix}
0 & \frac{2 \sqrt{\mu _{\rm e} \left(\frac{\mu _{\rm e}}{\rho }-v^2\right) \left(\frac{\lambda _{\rm e}+2 \,\mu _{\rm e}}{\rho }-v^2\right)}-\sqrt{\lambda _{\rm e}+2 \,\mu _{\rm e}} \left(\frac{2 \,\mu _{\rm e}}{\rho }-v^2\right)}{\sqrt{\rho } v^2 \sqrt{\frac{\lambda _{\rm e}+2 \,\mu _{\rm e}}{\rho }-v^2}} & 0 \\
-\frac{2 \sqrt{\mu _{\rm e} \left(\frac{\mu _{\rm e}}{\rho }-v^2\right) \left(\frac{\lambda _{\rm e}+2 \,\mu _{\rm e}}{\rho }-v^2\right)}-\sqrt{\lambda _{\rm e}+2 \,\mu _{\rm e}} \left(\frac{2 \,\mu _{\rm e}}{\rho }-v^2\right)}{v^2 \sqrt{\lambda _{\rm e}+2 \,\mu _{\rm e}} \sqrt{\frac{\rho }{\mu _{\rm e}}} \sqrt{\frac{\mu _{\rm e}}{\rho }-v^2}} & 0 & 0 \\
0 & 0 & 0 \\
 \end{pmatrix}\end{footnotesize}\notag\\
&=\begin{footnotesize}\begin{pmatrix}
0 & \frac{2\, {c}_{t} \sqrt{\left({c}_{t}^2-v^2\right) \left({c}_{l}^2-v^2\right)}-{c}_{l} \left(2\, {c}_{t} ^2-v^2\right)}{v^2 \sqrt{{c}_{l}^2-v^2}} & 0 \\
-\frac{2\, {c}_{t} \sqrt{\left({c}_{t}^2-v^2\right) \left({c}_{l}^2-v^2\right)}-{c}_{l} \left(2\, {c}_{t} ^2-v^2\right)}{\frac{{c}_{l} v^2 \sqrt{{c}_{t}^2-v^2}}{{c}_{t}}} & 0 & 0 \\
0 & 0 & 0 \\
 \end{pmatrix}\end{footnotesize},\notag
\end{align}
where
\begin{align} {c}_{l}&=\sqrt{\displaystyle\frac{2\,\mu_{\rm e}  +\lambda_{\rm e}}{\rho}}, \qquad \qquad
{c}_{t}=\sqrt{\displaystyle\frac{\mu_{\rm e}}{\rho}}, \qquad\qquad\ \,  {c}_{m}=\sqrt{\displaystyle\frac{{L_{\rm c}^2}\,\gamma}{\rho\, j\,\tau_{\rm c}^2\,}}
\notag 
\end{align}
Note that for $\mu_{\rm c}\to 0$ we have
\begin{align}
\mathfrak{c}_{p}={c}_{l}, \qquad\qquad \mathfrak{c}_{s}={c}_{t}, \qquad\qquad \mathfrak{c}_{m_1}={c}_{m}, \qquad\qquad \mathfrak{c}_{m_2}=\min\{\mathfrak{c}_{s}, \mathfrak{c}_{m_1}\}=\min\{{c}_{t}, {c}_{m}\},
\end{align}
and that a similar proof with  that of Proposition \ref{ChGhi} holds true in the case $\mu_{\rm c}\to 0$, too.
\begin{proposition}For $
	2\,\mu_{\rm e} +\lambda_{\rm e} >0,  \mu_{\rm e} >0,  \gamma>0,$ the limiting speed in Cosserat elastic materials is given by
	\begin{align}\label{limChel}
	\widehat{v}:=\inf_{\theta\in(-\frac{\pi}{2},\frac{\pi}{2})} v_\theta\equiv \min\left\{ 	 {c}_{l}, {c}_{t},{c}_{m}\right\}.
	\end{align}
\end{proposition}
By considering $\mu_{\rm c} \to 0$, Theorem \ref{Mielketh} allows us to affirm that  if  the constitutive coefficients satisfy the conditions $
2\,\mu_{\rm e} +\lambda_{\rm e} >0, \ \mu_{\rm e} >0,  \gamma>0,$
	and   $0\leq v<\widehat{v}=\inf_{\theta\in(-\frac{\pi}{2},\frac{\pi}{2})} v_\theta\equiv \min\left\{ 	 {c}_{l}, {c}_{t},{c}_4\right\}$, then
the unique solution of the algebraic Riccati equation \eqref{12} that satisfies $\text{\rm Re\,[spec}\,(\boldsymbol{\mathcal{T}}^{-1}(\boldsymbol{\mathcal{M}}+{\rm i}\,R^T))]>0$  given  by \eqref{explMielke}
takes  the form
\begin{align}\boldsymbol{\mathcal{M}}_v=\left(
\begin{array}{ccc}
\frac{{c}_{t}\,k^2 \, v^2 \sqrt{{c}_{l}^2-v^2}}{{c}_{l} {c}_{t}-\sqrt{{c}_{l}^2-v^2} \sqrt{{c}_{t}^2-v^2}} & \frac{{\rm i}\, {c}_{t} k^2 \left(2\, {c}_{t} \sqrt{\left(v^2-{c}_{l}^2\right) \left(v^2-{c}_{t}^2\right)}+{c}_{l} \left(v^2-2\, {c}_{t} ^2\right)\right)}{{c}_{l} {c}_{t}-\sqrt{{c}_{l}^2-v^2} \sqrt{{c}_{t}^2-v^2}} & 0 \\
\frac{{\rm i}\, {c}_{t} k^2 \left(2\, {c}_{t} \sqrt{\left(v^2-{c}_{l}^2\right) \left(v^2-{c}_{t}^2\right)}+{c}_{l} \left(v^2-2\, {c}_{t} ^2\right)\right)}{\sqrt{{c}_{l}^2-v^2} \sqrt{{c}_{t}^2-v^2}-{c}_{l} {c}_{t}} & \frac{{c}_{l} k^2 v^2 \sqrt{{c}_{t}^2-v^2}}{{c}_{l} {c}_{t}-\sqrt{{c}_{l}^2-v^2} \sqrt{{c}_{t}^2-v^2}} & 0 \\
0 & 0 & \frac{2\, k^2 \sqrt{{c}_{m}^2-v^2}}{{c}_{m}} \\
\end{array}
\right).
\end{align}
The secular equation in the Mielke-Fu's form is
\begin{align}\label{secM}
0=\det\boldsymbol{\boldsymbol{\mathcal{M}}}_v=&
\frac{2 \,k^2 \sqrt{{c}_{m}^2-v^2}}{{c}_{m}}\frac {{c}_{l} {c}_{t} v^4 \sqrt{{c}_{l}^2-v^2} \sqrt{{c}_{t}^2-v^2}-{c}_{t}^2 \left(2\, {c}_{t} \sqrt{\left({c}_{l}^2-v^2\right) \left({c}_{t}^2-v^2\right)}+{c}_{l} \left(v^2-2\, {c}_{t} ^2\right)\right)^2}{\left({c}_{l} {c}_{t}-\sqrt{{c}_{l}^2-v^2} \sqrt{{c}_{t}^2-v^2}\right)^2}.
\end{align}
 An admissible wave speed has to satisfy $0\leq v<\widehat{v}=\inf_{\theta\in(-\frac{\pi}{2},\frac{\pi}{2})} v_\theta\equiv \min\left\{ 	 {c}_{l}, {c}_{t},{c}_{m}\right\}$. Therefore, the secular equation in the Mielke-Fu's form is equivalent to 
 \begin{align}\label{Mielkes2}
 0=s_{\rm Mielke-Fu}(v)\equiv\frac {{c}_{l} {c}_{t} v^4 \sqrt{{c}_{l}^2-v^2} \sqrt{{c}_{t}^2-v^2}-{c}_{t}^2 \left(2\, {c}_{t} \sqrt{\left({c}_{l}^2-v^2\right) \left({c}_{t}^2-v^2\right)}+{c}_{l} \left(v^2-2\, {c}_{t} ^2\right)\right)^2}{\left({c}_{l} {c}_{t}-\sqrt{{c}_{l}^2-v^2} \sqrt{{c}_{t}^2-v^2}\right)^2}.
 \end{align}
 We mention that Mielke and Fu did not write explicitly  this equation for elastic isotropic material. Since 
 \begin{align}
 \lim_{v\to 0}\frac {{c}_{l} {c}_{t} v^4 \sqrt{{c}_{l}^2-v^2} \sqrt{{c}_{t}^2-v^2}-{c}_{t}^2 \left(2\, {c}_{t} \sqrt{\left({c}_{l}^2-v^2\right) \left({c}_{t}^2-v^2\right)}+{c}_{l} \left(v^2-2\, {c}_{t} ^2\right)\right)^2}{\left({c}_{l} {c}_{t}-\sqrt{{c}_{l}^2-v^2} \sqrt{{c}_{t}^2-v^2}\right)^2}=-4\, {c}_{t}^4\, k^4
 \end{align}
 the value $v=0$ is  not a singular point of $\det\boldsymbol{\boldsymbol{\mathcal{M}}}_v$
and the secular equation reduces to
\begin{align}\label{secM2}
{c}_{l} {c}_{t} v^4 \sqrt{{c}_{l}^2-v^2} \sqrt{{c}_{t}^2-v^2}-{c}_{t}^2 \left(2\, {c}_{t} \sqrt{\left({c}_{l}^2-v^2\right) \left({c}_{t}^2-v^2\right)}+{c}_{l} \left(v^2-2\, {c}_{t} ^2\right)\right)^2=0.
\end{align}

The secular equation \eqref{Mielkes2} and its equivalent form \eqref{secM2} do not coincide with  the classical well known form of the secular equation, i.e.
\begin{align}
0=s_{\rm classic}\equiv\, \sqrt{\left(1-\frac{v^2}{{c}_{l}^2}\right)\,\left(1-\frac{v^2}{{c}_{t}^2}\right)}-\left(2-\frac{v^2}{{c}_{t}^2}\right)^2.
\end{align}
 Indeed, even when the Stroh formalism is used \cite{Destrade,destrade2007seismic}, the obtained  secular equation does not coincide with the classical form, and it is given by
 \begin{align}
s_{\rm Stroh}\equiv\frac{\rho \, v^2 \sqrt{c_{11}-\rho \, v^2} \sqrt{c_{66}-\rho \, v^2}}{\sqrt{c_{11} c_{66}}}-\frac{\left(\rho \, v^2-c_0\right) \left(\rho v^2-c_{66}\right)}{c_{66}}
 \end{align}
 where
 \begin{align}
 c_{12}=\lambda_{\rm e}, \qquad c_{66}=\mu_{\rm e},  \qquad  c_{11}=\lambda_{\rm e} +2 \mu_{\rm e}, \qquad  c_0=c_{11}-\frac{c_{12}^2}{c_{11}} .
 \end{align}
 However, all these three forms of the secular equations are equivalent in the sense that \textit{they predict the same wave speed in the same domain of admissible wave speeds}. We illustrate this numerically, see Figure \ref{Figsec123}, for a specific material which is related to aluminium-epoxy.
\begin{figure}[h!]
	\centering 
			\includegraphics[width=0.6\linewidth]{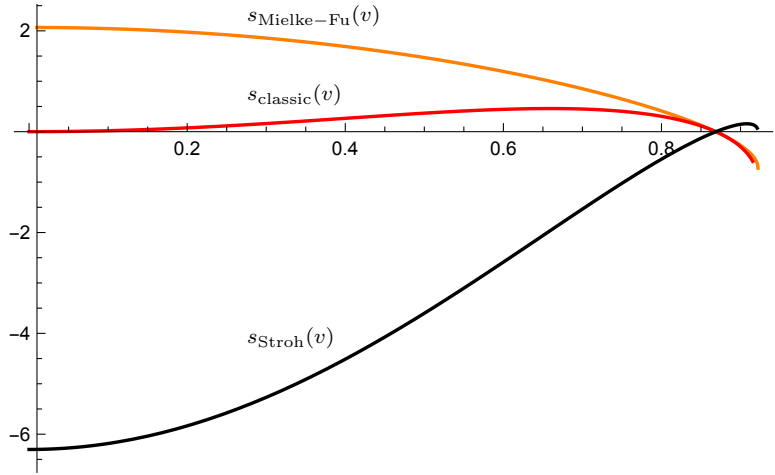}
			\put(-200,50){\footnotesize $s_{\rm Stroh}(v)$}
			\put(-200,142){\footnotesize $s_{\rm classic}(v)$}
			\put(-200,172){\footnotesize $s_{\rm Mielke-Fu}(v)$}
		\caption{\footnotesize All the alternative forms of the secular equations are equivalent. For a linear elastic material having the parameters $\lambda_{\rm e} =7.59\, \text{GPa}, \, \mu_{\rm e} =1.89745 \, \text{GPa}, \,
			\rho=2.22287\, \frac{g}{{\rm mm}^3}$ they give us the same value of the wave speed, i.e., $v_{\rm R}=0.868832$.}
	\label{Figsec123}
\end{figure}
 We point out again that there is no dependence of the wave speed on the wave number $k$, contrary with what happens when the Cosserat theory is considered, see Figure \ref{kCe}.

Moreover, after solving the system $\boldsymbol{\boldsymbol{\mathcal{M}}}_{v_{\rm R}}\, y(0)=0$ we obtain that the last component of $y(0)$ vanishes, since $\frac{2 k^2 \sqrt{{c}_{m}^2-v_{\rm R}^2}}{{c}_{m}}\neq 0$ due to the fact the solution is in the admissible set, which implies that $v_{\rm R}<{c}_{m}$.  This means that the micro-rotation vanishes and only macroscopic displacements govern the movement of the half-space. Therefore, from \eqref{08}, the wave propagation solution will be
\begin{align}\label{08c}
y(x_2)=e^{-k\,x_2 \,\boldsymbol{\mathcal{E}}_{v_{\rm R}}}\begin{pmatrix}y_1(0)\\
y_2(0)\end{pmatrix}
\end{align}
where $\boldsymbol{\mathcal{E}}_{v_{\rm R}}\in\mathbb{C}^{3 \times 3}$  is given by  
\begin{align}\label{10c} \boldsymbol{\mathcal{E}}_{v_{\rm R}}=\boldsymbol{\mathcal{T}}^{-1}(\boldsymbol{\mathcal{M}}_{v_{\rm R}}+ {\rm i}\, \,\boldsymbol{\mathcal{R}}^T),
\end{align}
and
$y_1(0), 
	y_2(0)$ are solutions of the algebraic system
	\begin{align}
	\begin{footnotesize}\begin{pmatrix}
		{{c}_{t}\,k^2 \, v^2 \sqrt{{c}_{l}^2-v^2}} & {{\rm i}\, {c}_{t} k^2 \left(2\, {c}_{t} \sqrt{\left(v^2-{c}_{l}^2\right) \left(v^2-{c}_{t}^2\right)}+{c}_{l} \left(v^2-2\, {c}_{t} ^2\right)\right)} \\
	-	{{\rm i}\, {c}_{t} k^2 \left(2\, {c}_{t} \sqrt{\left(v^2-{c}_{l}^2\right) \left(v^2-{c}_{t}^2\right)}+{c}_{l} \left(v^2-2\, {c}_{t} ^2\right)\right)} & {{c}_{l} k^2 v^2 \sqrt{{c}_{t}^2-v^2}}
	\end{pmatrix}\,\begin{pmatrix}y_1(0)\\
	y_2(0)\end{pmatrix}\end{footnotesize}=0.\notag
\end{align}

\section{Final remarks}
In this paper we have shown that the approach proposed by Mielke and Fu \cite{fu2002new,mielke2004uniqueness} for the study of the propagation of seismic waves in anisotropic linear elastic materials can be used in the framework of the isotropic Cosserat elastic materials, too. One of the big advantages of this new approach, compared with the classical methods like Stroh formalism, is that it leads to the proof of the existence and uniqueness of the solution of the obtained secular equation. While the existence of solutions was proven before for several problems concerning the propagation of seismic waves in materials with microstructure \cite{ChiritaGhiba2,ChiritaGhiba3}, but using restrictive conditions upon the constitutive coefficients, the uniqueness of the solution was, in the best situation, only conjectured in such generalized theories. Indeed, the problem of existence and uniqueness of the solution of the secular equation remains unsolved in almost all generalized linear theories from elasticity. This is the case since   the explicit forms of the corresponding secular equations are not completely analytically written. It is not the case when    Mielke and Fu's approach is used \cite{fu2002new,mielke2004uniqueness}, the secular equation being written as $\det\boldsymbol{\boldsymbol{\mathcal{M}}}_v=0$, where the hermitian matrix $\boldsymbol{\mathcal{M}}_v$ has a known integral form. Moreover, due to the form of the hermitian matrix, it is possible to prove that the secular equation  $\det\boldsymbol{\boldsymbol{\mathcal{M}}}_v=0$ has a unique admissible solution. Therefore, with this paper we close the problem of the propagation of seismic waves in isotropic linear Cosserat elastic materials and we propose two numerically viable strategies to solve this problem for specific materials. 

\begin{footnotesize}
	
	\bigskip
	
	\noindent\textbf{Acknowledgement.}

	The  work of Ionel-Dumitrel Ghiba  was supported by a grant of the Romanian Ministry of Research
	and Innovation, CNCS--UEFISCDI, project number
	PN-III-P1-1.1-TE- 2021-0783, within PNCDI III.  
	
		Hassam Khan thanks HEC/DAAD overseas scholarship scheme for~MSc leading to PhD, 2016 (57343333).  
	
	Angela Madeo acknowledges support from the European Commission through the funding of the ERC Consolidator
	Grant META-LEGO, N$^\circ$ 101001759. 
	
	Patrizio Neff acknowledges support  in the framework of the Priority Programme SPP 2256 "Variational Methods for Predicting Complex Phenomena in Engineering Structures and Materials" funded by the Deutsche Forschungsgemeinschaft (DFG, German research foundation), (Project-ID 440935806) "A variational scale-dependent transition scheme - from Cauchy elasticity to the relaxed micromorphic continuum".
	\bigskip
	
	\bibliographystyle{plain} %plain

\appendix
\section{An overview of the proofs from Fu and Mielke in our notation}\setcounter{equation}{0}

In this section, we will carry out  a matrix algebraic analysis  of the Riccati equation \eqref{12} of the linear Cosserat model to show that the properties listed in the previous section hold true for Cosserat isotropic materials, too. The content of this Appendix is almost entirely based on the results presented by Fu and Mielke in \cite{fu2001nonlinear,fu2002new} and we do not claim any merit in finding (and proving) them. The only aim of this appendix is to follow step by step the results from \cite{fu2001nonlinear,fu2002new} in order see that they can be applied in the framework of the linear isotropic Cosserat model, too.

\begin{proposition}{\rm (Exactly as from Fu and Mielke's paper)}
	The matrix problem 
	\begin{align}\label{13}
	 \boldsymbol{\mathcal{T}}{\mathcal{E}}^2- {\rm i}\, (\boldsymbol{\mathcal{R}}+ \boldsymbol{\mathcal{R}}^T)E- \boldsymbol{\mathcal{Q}}+ k^2 v^2{\id}=0,\qquad \textrm{\rm Re\,spec\,}\boldsymbol{\mathcal{E}}>0,
	\end{align}
	where `$\textrm{\rm Re\,spec\,}\boldsymbol{\mathcal{E}} $  means the real part of spectra of $\boldsymbol{\mathcal{E}}$  and $\boldsymbol{\mathcal{T}}$,$\boldsymbol{\mathcal{T}}$, $\boldsymbol{\mathcal{R}}$ are defined in \eqref{nTQ}, has a unique solution for $\boldsymbol{\mathcal{E}}$. 
\end{proposition}
\begin{proof}
	Let $\boldsymbol{\mathcal{E}}$ be the solution of \eqref{13}. Let $r$ be the eigenvalue of $\boldsymbol{\mathcal{E}}$ and $a$ be the associated eigenvector (so that $\boldsymbol{\mathcal{E}}a=ra$). It follows from \eqref{13} that $r$ and $a$  must satisfy the eigenvalue problem
	\begin{align}\label{15}
	\{r^2	 \boldsymbol{\mathcal{T}}-{\rm i}\,r( \boldsymbol{ \boldsymbol{\mathcal{R}}}+ \boldsymbol{ \boldsymbol{\mathcal{R}}}^T)- \boldsymbol{ \boldsymbol{\mathcal{Q}}}+k^2 v^2{\id}\}a=0, \qquad 
	\det\{r^2	 \boldsymbol{\mathcal{T}}-{\rm i}\,r( \boldsymbol{ \boldsymbol{\mathcal{R}}}+ \boldsymbol{ \boldsymbol{\mathcal{R}}}^T)- \boldsymbol{ \boldsymbol{\mathcal{Q}}}+k^2 v^2{\id}\}=0.
	\end{align}
	In the explicit form, the characteristic equation $\eqref{15}_2$ 
	takes the form
	\begin{align} 
	r^6+P_1 r^4+P_2 r^2+P_3 =0,
	\end{align}
	where $P_1, P_2, P_3, P_4$ are given by \eqref{defpC}. 
	We observe that characteristic equation is a six degree polynomial with real coefficients. If $r$ is the root of the characteristic equation then so is $-\overline{r}$.  Thus, we conclude that the characteristic equation has three roots with positive real parts. Collecting the corresponding eigenspaces defines $\boldsymbol{\mathcal{E}}$ uniquely.
\end{proof}
\begin{theorem}{\rm (Exactly as from Fu and Mielke's paper)}\label{mT1}
	If $\boldsymbol{\mathcal{E}}$ solves \eqref{13}, then $\boldsymbol{\mathcal{M}}$ obtained from $\eqref{10}_2$ is Hermitian.
\end{theorem}
\begin{proof}
	By  taking   transpose  and complex conjugate of $\eqref{12}_1$ we obtain
	\begin{align}\label{16}
	(\boldsymbol{\mathcal{M}}^T+ {\rm i}\,\boldsymbol{ \boldsymbol{\mathcal{R}}}) \boldsymbol{\mathcal{T}}^{-1}(\boldsymbol{\mathcal{M}}- {\rm i}\,\boldsymbol{ \boldsymbol{\mathcal{R}}}^T)- \boldsymbol{ \boldsymbol{\mathcal{Q}}}^T+k^2 v^2 {\id}=0\quad  \Leftrightarrow\ \quad 
	(\overline{	\boldsymbol{\mathcal{M}}}^T- {\rm i}\,\boldsymbol{ \boldsymbol{\mathcal{R}}})\mathcal{ \boldsymbol{\mathcal{T}}}^{-1}(\overline{	\boldsymbol{\mathcal{M}}}+ {\rm i}\,\boldsymbol{ \boldsymbol{\mathcal{R}}}^T)- \boldsymbol{ \boldsymbol{\mathcal{Q}}}^T+k^2 v^2  {\id}=0.
	\end{align}
	On subtracting  \eqref{16}$_1$ from \eqref{16}$_2$, we get
	\begin{align}
	(\boldsymbol{\mathcal{M}}-\overline{	\boldsymbol{\mathcal{M}}}^T)\boldsymbol{\mathcal{E}}+\overline{\boldsymbol{\mathcal{E}}}^T(\boldsymbol{\mathcal{M}}-\overline{	\boldsymbol{\mathcal{M}}}^T)=0,
	\end{align}
	which has the form of a Liaponov matrix equation
	\begin{align}\label{18}
	\mathbf{X}\boldsymbol{\mathcal{E}}+\overline{\boldsymbol{\mathcal{E}}}^T\mathbf{X}=\mathbf{B}.
	\end{align}
	The equation \eqref{18} has a unique solution $\mathbf{X}$.
	For $\text{Re}\,\text{spec}(\boldsymbol{\mathcal{E}})>0$ it is given \cite{barnet1990matrix} by 
	\begin{align}
	\mathbf{X}=\int_{0}^{\infty} e^{-t\,\overline{\boldsymbol{\mathcal{E}}}^T}\,\mathbf{B}\, e^{-t\,\boldsymbol{\mathcal{E}}} dt.
	\end{align}
	Since we have $\mathbf{B}=0$ the unique solution of \eqref{18} is $\mathbf{X}=0$ and hence $\boldsymbol{\mathcal{M}}=\overline{\boldsymbol{\mathcal{M}}}^T$.
\end{proof}
\begin{remark}
	{\rm 	Every Hermitian matrix has real eigenvalues.	The determinant of Hermitian matrix is a real number. \textsl{}A Hermitian matrix is said to be postive positive if it has positive eigenvalues.}
\end{remark}

In the proof of the following theorem, there is only a small difference compared to the proof given by Fu and Mielke \cite{fu2002new}.
\begin{theorem}{\rm (Exactly as from Fu and Mielke's paper)} 
	Let $\boldsymbol{\mathcal{M}}$ and $\boldsymbol{\mathcal{E}}$ be the same as in Theorem \ref{mT1}. Then the matrix $\frac{d \boldsymbol{\mathcal{M}}}{dv}$ is negative definite.
\end{theorem}
\begin{proof}
	Differentiation $\eqref{12}_1$ w.r.t $v$ gives
	\begin{align}
	\frac{d}{dv}(\boldsymbol{\mathcal{M}}- {\rm i}\,\boldsymbol{ \boldsymbol{\mathcal{R}}})\boldsymbol{\mathcal{E}}+(\boldsymbol{\mathcal{M}}- {\rm i}\,\boldsymbol{ \boldsymbol{\mathcal{R}}})\frac{d}{dv}\boldsymbol{\mathcal{E}}-\frac{d}{dv} \boldsymbol{ \boldsymbol{\mathcal{Q}}}+\frac{d}{dv}( k^2 v^2 {\id})&=0,\notag\\
	\frac{d}{dv}\boldsymbol{\mathcal{M}}\,\boldsymbol{\mathcal{E}}+(\boldsymbol{\mathcal{M}}- {\rm i}\,\boldsymbol{ \boldsymbol{\mathcal{R}}})\frac{d}{dv}\boldsymbol{\mathcal{E}}&=-2\, v \,k^2 {\id},\notag\\
	\frac{d}{dv}\boldsymbol{\mathcal{M}}\,\boldsymbol{\mathcal{E}}+(\boldsymbol{\mathcal{M}}- {\rm i}\,\boldsymbol{ \boldsymbol{\mathcal{R}}})\frac{d}{dv}\{ \boldsymbol{\mathcal{T}}^{-1}(\boldsymbol{\mathcal{M}}+ {\rm i}\,\boldsymbol{ \boldsymbol{\mathcal{R}}}^T)\}&=-2\, v \,k^2 {\id},\\
	\frac{d}{dv}\boldsymbol{\mathcal{M}}\,\boldsymbol{\mathcal{E}}+(\boldsymbol{\mathcal{M}}- {\rm i}\,\boldsymbol{ \boldsymbol{\mathcal{R}}}) \boldsymbol{\mathcal{T}}^{-1}\frac{d}{dv}\boldsymbol{\mathcal{M}}&=-2\, v \,k^2 {\id},\notag\\
	\frac{d}{dv}\boldsymbol{\mathcal{M}}\,\boldsymbol{\mathcal{E}}+(\overline{\boldsymbol{\mathcal{M}}}^T- {\rm i}\,\boldsymbol{ \boldsymbol{\mathcal{R}}}) \boldsymbol{\mathcal{T}}^{-1}\frac{d}{dv}\boldsymbol{\mathcal{M}}&=-2\, v \,k^2\, {\id},\qquad\qquad \boldsymbol{\mathcal{M}}=\overline{\boldsymbol{\mathcal{M}}}^T.\notag
	\end{align}
	Finally, we obtain
	\begin{align}\label{23}
	\frac{d}{dv}\boldsymbol{\mathcal{M}}\,\boldsymbol{\mathcal{E}}+\overline{\boldsymbol{\mathcal{E}}}^T\frac{d}{dv}\boldsymbol{\mathcal{M}}=-2\, k^2\, v\, {\id}.
	\end{align}   Equation \eqref{23}  has a unique solution for $\boldsymbol{\mathcal{M}}$ given by
	\begin{align}
	\frac{d}{dv}\boldsymbol{\mathcal{M}}=\int_{0}^{\infty} e^{-t\overline{\boldsymbol{\mathcal{E}}}^T}(-2\,\rho\, k^2 \,v\, {\id} )\, e^{-t\boldsymbol{\mathcal{E}}} dt=-2\,k^2 \, v\,\int_{0}^{\infty} e^{-t\overline{\boldsymbol{\mathcal{E}}}^T}  {\id} e^{-t\boldsymbol{\mathcal{E}}} dt.
	\end{align} 
	Thus, for arbitrary non-zero vector $\eta \in \mathbb{C}^n$ we have
	\begin{align}
	\langle	\overline{\eta }, \frac{d}{dv}\boldsymbol{\mathcal{M}}\,\eta\rangle=&-2\,k^2 \, v\,\langle \overline{\eta } ,\int_{0}^{\infty}e^{-t\overline{\boldsymbol{\mathcal{E}}}^T}\,  {\id}\,e^{-t\boldsymbol{\mathcal{E}}} dt \,\eta\rangle_{\mathbb{C}^{n}}\notag=-2\,k^2\,  v\int_{0}^{\infty}\langle   \,\overline{e^{-t{\boldsymbol{\mathcal{E}}}}\eta } , \,e^{-t\boldsymbol{\mathcal{E}}} dt \,\eta\rangle_{\mathbb{C}^{n}},\\
	=&-2\,k^2 \, v\int_{0}^{\infty}\langle\overline{\zeta}(t) , \,\zeta(t)\rangle_{\mathbb{C}^{n}}\,dt =-2\,k^2 \, v\int_{0}^{\infty}\underbrace{\norm{\zeta(t)}^2}_{\geq a^+>0}\,dt\, ,\quad \ \ \ \ \ \ \ \text{where} \,\,\ \ \zeta(t)= e^{-t\boldsymbol{\mathcal{E}}}\eta.
	\end{align}
	Since $\zeta(0)=\eta \neq 0$ and $\zeta(t)$ is continuous at $t=0$ (so that $\zeta(t)$ is non-zero at least in a small but finite interval), we have $\langle\overline{\eta }, \frac{d}{d\,v}\boldsymbol{\mathcal{M}}\,\eta\rangle<0$ and hence $\frac{d}{d\,v} \boldsymbol{\mathcal{M}}$ is negative definite.
\end{proof}

\begin{theorem}{\rm (Exactly as from Fu and Mielke's paper)}
	The Hermitian matrix $\boldsymbol{\mathcal{M}}_\theta $ satisfying  \eqref{17}   is independent of $\theta$.
\end{theorem}
\begin{proof} Only in this proof, we denote by $f'(\theta)=\frac{d\,f}{d\, \theta}(\theta)$. On	differentiating $\eqref{21}_1$  w.r.t $\theta$ we get
	\begin{align}
	 \boldsymbol{\mathcal{T}}_{\!\! \theta} '=2\, \sin\theta \cos\theta(\widetilde{ \boldsymbol{ \boldsymbol{\mathcal{Q}}}}- \boldsymbol{\mathcal{T}})-\cos^2\theta( \boldsymbol{ \boldsymbol{\mathcal{R}}}+ \boldsymbol{ \boldsymbol{\mathcal{R}}}^T)+\sin^2\theta( \boldsymbol{ \boldsymbol{\mathcal{R}}}+ \boldsymbol{ \boldsymbol{\mathcal{R}}}^T).
	\end{align}
	It is obvious to see
	\begin{align}
	 \boldsymbol{ \boldsymbol{\mathcal{R}}}_\theta  ^T=\cos^2\theta   \boldsymbol{ \boldsymbol{\mathcal{R}}}^T-\sin^2\theta   \boldsymbol{ \boldsymbol{\mathcal{R}}}+\sin\theta\cos\theta (  \boldsymbol{\mathcal{T}}-\widetilde{ \boldsymbol{ \boldsymbol{\mathcal{Q}}}}),\notag
	\end{align}
	and
	\begin{align}
	- \boldsymbol{ \boldsymbol{\mathcal{R}}}_\theta  - \boldsymbol{ \boldsymbol{\mathcal{R}}}_\theta  ^T=2\, \sin\theta \cos\theta(\widetilde{ \boldsymbol{ \boldsymbol{\mathcal{Q}}}}- \boldsymbol{\mathcal{T}})-\cos^2\theta( \boldsymbol{ \boldsymbol{\mathcal{R}}}+ \boldsymbol{ \boldsymbol{\mathcal{R}}}^T)+\sin^2\theta( \boldsymbol{ \boldsymbol{\mathcal{R}}}+ \boldsymbol{ \boldsymbol{\mathcal{R}}}^T)\notag,
	\end{align}
	\begin{align}\label{22}
	 \boldsymbol{\mathcal{T}}_{\!\! \theta} '=	- \boldsymbol{ \boldsymbol{\mathcal{R}}}_\theta  - \boldsymbol{ \boldsymbol{\mathcal{R}}}_\theta  ^T.
	\end{align}
	Similarly,
	\begin{align}
	 \boldsymbol{ \boldsymbol{\mathcal{R}}}_\theta  '=-2\, \cos\theta\sin\theta( \boldsymbol{ \boldsymbol{\mathcal{R}}}+ \boldsymbol{ \boldsymbol{\mathcal{R}}}^T)+(\cos^2\theta-\sin^2\theta)( \boldsymbol{\mathcal{T}}-\widetilde{ \boldsymbol{ \boldsymbol{\mathcal{Q}}}}),\notag
	\end{align}
and	we see that
	\begin{align}
	 \boldsymbol{\mathcal{T}}_{\!\! \theta} -{ \boldsymbol{ \boldsymbol{\mathcal{Q}}}}_\theta=(\cos^2\theta-\sin^2\theta)( \boldsymbol{\mathcal{T}}-\widetilde{ \boldsymbol{ \boldsymbol{\mathcal{Q}}}})-2\, \sin\theta\cos\theta( \boldsymbol{ \boldsymbol{\mathcal{R}}}+ \boldsymbol{ \boldsymbol{\mathcal{R}}}^T)\notag.
	\end{align}
Therefore
	\begin{align}\label{24}
	 \boldsymbol{ \boldsymbol{\mathcal{R}}}_\theta  '=	 \boldsymbol{\mathcal{T}}_{\!\! \theta} -	{ \boldsymbol{ \boldsymbol{\mathcal{Q}}}}_\theta,
	\end{align}
	and
	\begin{align}
	{ \boldsymbol{ \boldsymbol{\mathcal{Q}}}}_\theta'=2\, \sin\theta \cos\theta( \boldsymbol{\mathcal{T}}- \boldsymbol{ \boldsymbol{\mathcal{Q}}})+\cos^2\theta( \boldsymbol{ \boldsymbol{\mathcal{R}}}+ \boldsymbol{ \boldsymbol{\mathcal{R}}}^T)-\sin^2\theta( \boldsymbol{ \boldsymbol{\mathcal{R}}}+ \boldsymbol{ \boldsymbol{\mathcal{R}}}^T).\notag
	\end{align}	
	It is also easy to see
	\begin{align}\label{25}
	{ \boldsymbol{ \boldsymbol{\mathcal{Q}}}}_\theta'=	 \boldsymbol{ \boldsymbol{\mathcal{R}}}_\theta  +	 \boldsymbol{ \boldsymbol{\mathcal{R}}}_\theta  ^{T},
	\end{align}
	and after	differentiating $ \boldsymbol{\mathcal{T}}_{\!\! \theta}  \boldsymbol{\mathcal{T}}_{\!\! \theta} ^{-1}=\id$, we also get
	\begin{align}\label{z6}
	( \boldsymbol{\mathcal{T}}_{\!\! \theta} ^{-1})'=- \boldsymbol{\mathcal{T}}_{\!\! \theta} ^{-1} \boldsymbol{\mathcal{T}}_{\!\! \theta} ' \boldsymbol{\mathcal{T}}_{\!\! \theta} ^{-1}.
	\end{align}
	On differentiating $\eqref{12}_1$ w.r.t $\theta$ we get
{	\begin{align}\label{26}
	(\boldsymbol{\mathcal{M}}_\theta - {\rm i}\,\boldsymbol{ \boldsymbol{\mathcal{R}}}_\theta  )'  \boldsymbol{\mathcal{E}}_{\!\theta}   +	(\boldsymbol{\mathcal{M}}_\theta - {\rm i}\,\boldsymbol{ \boldsymbol{\mathcal{R}}}_\theta  )  \boldsymbol{\mathcal{E}}_{\!\theta}   '-\boldsymbol{\mathcal{M}}_\theta '=0,\notag\\
	%(\boldsymbol{\mathcal{M}}_\theta '- {\rm i}\,\boldsymbol{ \boldsymbol{\mathcal{R}}}_\theta  ')  \boldsymbol{\mathcal{E}}_{\!\theta}   +	(\boldsymbol{\mathcal{M}}_\theta - {\rm i}\,\boldsymbol{ \boldsymbol{\mathcal{R}}}_\theta  )\{- \boldsymbol{\mathcal{T}}_{\!\! \theta} ^{-1} \boldsymbol{\mathcal{T}}_{\!\! \theta} ' \boldsymbol{\mathcal{T}}_{\!\! \theta} ^{-1}(\boldsymbol{\mathcal{M}}_\theta + {\rm i}\,\boldsymbol{ \boldsymbol{\mathcal{R}}}_\theta  ^T)+ \boldsymbol{\mathcal{T}}_{\!\! \theta} ^{-1}(\boldsymbol{\mathcal{M}}_\theta '+ {\rm i}\, ( \boldsymbol{ \boldsymbol{\mathcal{R}}}_\theta  ^T )')\}-{ \boldsymbol{ \boldsymbol{\mathcal{Q}}}}_\theta'=0,\notag\\
%	(\boldsymbol{\mathcal{M}}_\theta '- {\rm i}\,\boldsymbol{ \boldsymbol{\mathcal{R}}}_\theta  ')  \boldsymbol{\mathcal{E}}_{\!\theta}   +	(\boldsymbol{\mathcal{M}}_\theta - {\rm i}\,\boldsymbol{ \boldsymbol{\mathcal{R}}}_\theta  ) \boldsymbol{\mathcal{T}}_{\!\! \theta} ^{-1}\{(\boldsymbol{\mathcal{M}}_\theta + {\rm i}\, ( \boldsymbol{ \boldsymbol{\mathcal{R}}}_\theta  ^T)')- \boldsymbol{\mathcal{T}}_{\!\! \theta} ' \boldsymbol{\mathcal{T}}_{\!\! \theta} ^{-1}(\boldsymbol{\mathcal{M}}_\theta + {\rm i}\,\boldsymbol{ \boldsymbol{\mathcal{R}}}_\theta  ^T )\}-{ \boldsymbol{ \boldsymbol{\mathcal{Q}}}}_\theta'=0,\\
	(\boldsymbol{\mathcal{M}}_\theta '- {\rm i}\,\boldsymbol{ \boldsymbol{\mathcal{R}}}_\theta  ')  \boldsymbol{\mathcal{E}}_{\!\theta}   +	(\boldsymbol{\mathcal{M}}_\theta - {\rm i}\,\boldsymbol{ \boldsymbol{\mathcal{R}}}_\theta  ) \boldsymbol{\mathcal{T}}_{\!\! \theta} ^{-1}\{(\boldsymbol{\mathcal{M}}_\theta '+ {\rm i}\, ( \boldsymbol{ \boldsymbol{\mathcal{R}}}_\theta  ^T )')- \boldsymbol{\mathcal{T}}_{\!\! \theta} ' \boldsymbol{\mathcal{T}}_{\!\! \theta} ^{-1}(\boldsymbol{\mathcal{M}}_\theta + {\rm i}\,\boldsymbol{ \boldsymbol{\mathcal{R}}}_\theta  ^T)\}-{ \boldsymbol{ \boldsymbol{\mathcal{Q}}}}_\theta'=0,\notag\\
	(\boldsymbol{\mathcal{M}}_\theta '- {\rm i}\,\boldsymbol{ \boldsymbol{\mathcal{R}}}_\theta  ')  \boldsymbol{\mathcal{E}}_{\!\theta}   +\overline{\boldsymbol{\mathcal{E}}}_\theta ^T(\boldsymbol{\mathcal{M}}_\theta '+ {\rm i}\, ( \boldsymbol{ \boldsymbol{\mathcal{R}}}_\theta  ^T )')-\overline{\boldsymbol{\mathcal{E}}}_\theta ^T \boldsymbol{\mathcal{T}}_{\!\! \theta} '{\boldsymbol{\mathcal{E}}}_\theta-{ \boldsymbol{ \boldsymbol{\mathcal{Q}}}}_\theta'=0,\\
	\boldsymbol{\mathcal{M}}_\theta ' \boldsymbol{\mathcal{E}}_{\!\theta}   +\overline{\boldsymbol{\mathcal{E}}}_\theta ^T	\boldsymbol{\mathcal{M}}_\theta '- {\rm i}\,\boldsymbol{ \boldsymbol{\mathcal{R}}}_\theta  ' \boldsymbol{\mathcal{E}}_{\!\theta}   + {\rm i}\, \overline{\boldsymbol{\mathcal{E}}}_\theta ^T( \boldsymbol{ \boldsymbol{\mathcal{R}}}_\theta  ^T )'-\overline{\boldsymbol{\mathcal{E}}}_\theta ^T \boldsymbol{\mathcal{T}}_{\!\! \theta} '{\boldsymbol{\mathcal{E}}}_\theta-{ \boldsymbol{ \boldsymbol{\mathcal{Q}}}}_\theta'=0,\notag
	\end{align}}
	and 	using \eqref{22},\eqref{24} and \eqref{25} in $\eqref{26}_6$, yields
	\begin{align}\label{27}
	\boldsymbol{\mathcal{M}}_\theta ' \boldsymbol{\mathcal{E}}_{\!\theta}   +\overline{\boldsymbol{\mathcal{E}}}_\theta ^T\boldsymbol{\mathcal{M}}_\theta '- {\rm i}\, ( \boldsymbol{ \boldsymbol{\mathcal{R}}}_\theta  )' \boldsymbol{\mathcal{E}}_{\!\theta}   +{\rm i}\,\overline{\boldsymbol{\mathcal{E}}}_\theta ^T( \boldsymbol{ \boldsymbol{\mathcal{R}}}_\theta  )'-\overline{\boldsymbol{\mathcal{E}}}_\theta ^T \boldsymbol{\mathcal{T}}_{\!\! \theta} ' \boldsymbol{\mathcal{E}}_{\!\theta}   - \boldsymbol{ \boldsymbol{\mathcal{Q}}}'_\theta =0,\notag\\
	\boldsymbol{\mathcal{M}}_\theta ' \boldsymbol{\mathcal{E}}_{\!\theta}   +\overline{\boldsymbol{\mathcal{E}}}_\theta ^T\boldsymbol{\mathcal{M}}_\theta '- {\rm i}\, ( \boldsymbol{\mathcal{T}}_{\!\! \theta} - { \boldsymbol{ \boldsymbol{\mathcal{Q}}}}_\theta) \boldsymbol{\mathcal{E}}_{\!\theta}  + {\rm i}\,\overline{\boldsymbol{\mathcal{E}}}_\theta ^T( \boldsymbol{\mathcal{T}}_{\!\! \theta} - { \boldsymbol{ \boldsymbol{\mathcal{Q}}}}_\theta)+ {\boldsymbol{\mathcal{E}}}_\theta ^T( \boldsymbol{ \boldsymbol{\mathcal{R}}}_\theta  + \boldsymbol{ \boldsymbol{\mathcal{R}}}_\theta  ^T) \boldsymbol{\mathcal{E}}_{\!\theta}   -( \boldsymbol{ \boldsymbol{\mathcal{R}}}_\theta  + \boldsymbol{ \boldsymbol{\mathcal{R}}}_\theta  ^T)=0.
	\end{align} 
	On replacing  first ${ \boldsymbol{ \boldsymbol{\mathcal{Q}}}}_\theta$ in \eqref{27}$_2$ by $ {\boldsymbol{\mathcal{E}}}_\theta ^T(\boldsymbol{\mathcal{M}}_\theta + {\rm i}\,\boldsymbol{ \boldsymbol{\mathcal{R}}}_\theta  ^T)$ and the second  ${ \boldsymbol{ \boldsymbol{\mathcal{Q}}}}_\theta$ by $(\boldsymbol{\mathcal{M}}_\theta - {\rm i}\,\boldsymbol{ \boldsymbol{\mathcal{R}}}_\theta  ) \boldsymbol{\mathcal{E}}_{\!\theta}    $, see \eqref{eqMt}, we deduce
	\begin{align}
	\boldsymbol{\mathcal{M}}_\theta ' \boldsymbol{\mathcal{E}}_{\!\theta}   + {\boldsymbol{\mathcal{E}}}_\theta ^T\boldsymbol{\mathcal{M}}_\theta '- {\rm i}\, [\{ \boldsymbol{\mathcal{T}}_{\!\! \theta} - {\boldsymbol{\mathcal{E}}}_\theta ^T(\boldsymbol{\mathcal{M}}_\theta + {\rm i}\,\boldsymbol{ \boldsymbol{\mathcal{R}}}_\theta  ^T)\} \boldsymbol{\mathcal{E}}_{\!\theta}   + {\rm i}\,{\boldsymbol{\mathcal{E}}}_\theta ^T\{ \boldsymbol{\mathcal{T}}_{\!\! \theta} - (\boldsymbol{\mathcal{M}}_\theta - {\rm i}\,\boldsymbol{ \boldsymbol{\mathcal{R}}}_\theta){\boldsymbol{\mathcal{E}}}_\theta \}] +{\boldsymbol{\mathcal{E}}}_\theta ^T( \boldsymbol{ \boldsymbol{\mathcal{R}}}_\theta  + \boldsymbol{ \boldsymbol{\mathcal{R}}}_\theta  ^T) \boldsymbol{\mathcal{E}}_{\!\theta}    -( \boldsymbol{ \boldsymbol{\mathcal{R}}}_\theta  + \boldsymbol{ \boldsymbol{\mathcal{R}}}_\theta  ^T)=0,\notag\\
	\boldsymbol{\mathcal{M}}_\theta ' \boldsymbol{\mathcal{E}}_{\!\theta}   + {\boldsymbol{\mathcal{E}}}_\theta ^T\boldsymbol{\mathcal{M}}_\theta '- {\rm i}\,\boldsymbol{\mathcal{T}}_{\!\! \theta}  \boldsymbol{\mathcal{E}}_{\!\theta}   + {\rm i}\,{\boldsymbol{\mathcal{E}}}_\theta^T  \boldsymbol{\mathcal{T}}_{\!\! \theta} - {\boldsymbol{\mathcal{E}}}_\theta^T  \boldsymbol{ \boldsymbol{\mathcal{R}}}_\theta  ^T \boldsymbol{\mathcal{E}}_{\!\theta}   - {\boldsymbol{\mathcal{E}}}_\theta^T  \boldsymbol{ \boldsymbol{\mathcal{R}}}_\theta    \boldsymbol{\mathcal{E}}_{\!\theta}   + {\boldsymbol{\mathcal{E}}}_\theta ^T( \boldsymbol{ \boldsymbol{\mathcal{R}}}_\theta  + \boldsymbol{ \boldsymbol{\mathcal{R}}}_\theta  ^T) \boldsymbol{\mathcal{E}}_{\!\theta}   -( \boldsymbol{ \boldsymbol{\mathcal{R}}}_\theta  + \boldsymbol{ \boldsymbol{\mathcal{R}}}_\theta  ^T)=0,\\
	\boldsymbol{\mathcal{M}}_\theta ' \boldsymbol{\mathcal{E}}_{\!\theta}   + {\boldsymbol{\mathcal{E}}}_\theta ^T\boldsymbol{\mathcal{M}}_\theta '	- {\rm i}\,\boldsymbol{\mathcal{T}}_{\!\! \theta}   \boldsymbol{\mathcal{E}}_{\!\theta}   + {\rm i}\,{\boldsymbol{\mathcal{E}}}_\theta^T  \boldsymbol{\mathcal{T}}_{\!\! \theta} -( \boldsymbol{ \boldsymbol{\mathcal{R}}}_\theta  + \boldsymbol{ \boldsymbol{\mathcal{R}}}_\theta  ^T)=0,\notag\\
	\boldsymbol{\mathcal{M}}_\theta ' \boldsymbol{\mathcal{E}}_{\!\theta}   + {\boldsymbol{\mathcal{E}}}_\theta ^T\boldsymbol{\mathcal{M}}_\theta '	- {\rm i}\,\boldsymbol{\mathcal{T}}_{\!\! \theta} \{ \boldsymbol{\mathcal{T}}_{\!\! \theta} ^{-1}(\boldsymbol{\mathcal{M}}_\theta + {\rm i}\,\boldsymbol{ \boldsymbol{\mathcal{R}}}_\theta  ^T)\}-( \boldsymbol{ \boldsymbol{\mathcal{R}}}_\theta  + \boldsymbol{ \boldsymbol{\mathcal{R}}}_\theta  ^T)+ {\rm i}\, \{(\boldsymbol{\mathcal{M}}_\theta - {\rm i}\,\boldsymbol{ \boldsymbol{\mathcal{R}}}_\theta  ) \boldsymbol{\mathcal{T}}_{\!\! \theta} ^{-1}\} \boldsymbol{\mathcal{T}}_{\!\! \theta} =0,\notag\\	
	\boldsymbol{\mathcal{M}}_\theta ' \boldsymbol{\mathcal{E}}_{\!\theta}   + {\boldsymbol{\mathcal{E}}}_\theta ^T\boldsymbol{\mathcal{M}}_\theta '	- {\rm i}\, \cancel{\boldsymbol{\mathcal{M}}_\theta }+ {\rm i}\, \cancel{\boldsymbol{\mathcal{M}}_\theta }+	\cancel{( \boldsymbol{ \boldsymbol{\mathcal{R}}}_\theta  + \boldsymbol{ \boldsymbol{\mathcal{R}}}_\theta  ^T)}-\cancel{	( \boldsymbol{ \boldsymbol{\mathcal{R}}}_\theta  + \boldsymbol{ \boldsymbol{\mathcal{R}}}_\theta  ^T)}=0,\notag	
	\end{align}
	finaly we obtain,
	\begin{align}
	\boldsymbol{\mathcal{M}}_\theta ' \boldsymbol{\mathcal{E}}_{\!\theta}   + {\boldsymbol{\mathcal{E}}}_\theta ^T\boldsymbol{\mathcal{M}}_\theta '=0
	\end{align}
	which is another homogeneous Liaponov matrix equation. It then follows that $\boldsymbol{\mathcal{M}}_\theta '=0$, (see  \eqref{18}) so $\boldsymbol{\mathcal{M}}_\theta $ is independent of $\theta$.
\end{proof}
We have observed that $ \boldsymbol{\mathcal{E}}_{\!\theta}   $ reduces to $\boldsymbol{\mathcal{E}}$ defined in Theorem $2.1$ 
when $\theta=0$,  we have $\boldsymbol{\mathcal{M}}_\theta \equiv \boldsymbol{\mathcal{M}}$, where $\boldsymbol{\mathcal{M}}$ is the corresponding $\boldsymbol{\mathcal{M}}$ given in Theorem $(2.1)$ . Thus
\begin{align}\label{28}
\boldsymbol{\mathcal{E}}_{\!\theta}   = \boldsymbol{\mathcal{T}}_{\!\! \theta} ^{-1}(\boldsymbol{\mathcal{M}}+ {\rm i}\,\boldsymbol{ \boldsymbol{\mathcal{R}}}_\theta  ^T).
\end{align}
On	differentiating \eqref{28} with respect to $\theta$, we get
\begin{align}\label{29}
\boldsymbol{\mathcal{E}}_{\!\theta}   '&=( \boldsymbol{\mathcal{T}}_{\!\! \theta} ^{-1})'(\boldsymbol{\mathcal{M}}+ {\rm i}\,\boldsymbol{ \boldsymbol{\mathcal{R}}}_\theta  ^T)+ \boldsymbol{\mathcal{T}}_{\!\! \theta} ^{-1}(\boldsymbol{\mathcal{M}}+ {\rm i}\,\boldsymbol{ \boldsymbol{\mathcal{R}}}_\theta  ^T)'=( \boldsymbol{\mathcal{T}}_{\!\! \theta} ^{-1})'(\boldsymbol{\mathcal{M}}+ {\rm i}\,\boldsymbol{ \boldsymbol{\mathcal{R}}}_\theta  ^T)+ \boldsymbol{\mathcal{T}}_{\!\! \theta} ^{-1}( {\rm i}\,\boldsymbol{ \boldsymbol{\mathcal{R}}}_\theta  ^T)',
\end{align}
making use of \eqref{z6}, we obtain
\begin{align}\label{30}
\boldsymbol{\mathcal{E}}_{\!\theta}   '&=- \boldsymbol{\mathcal{T}}_{\!\! \theta} ^{-1} \boldsymbol{\mathcal{T}}_{\!\! \theta} ' \boldsymbol{\mathcal{T}}_{\!\! \theta} ^{-1}(\boldsymbol{\mathcal{M}}+ {\rm i}\,\boldsymbol{ \boldsymbol{\mathcal{R}}}_\theta  ^T)+ {\rm i}\,\boldsymbol{\mathcal{T}}_{\!\! \theta} ^{-1}( \boldsymbol{\mathcal{T}}_{\!\! \theta} - { \boldsymbol{ \boldsymbol{\mathcal{Q}}}}_\theta)= \boldsymbol{\mathcal{T}}_{\!\! \theta} ^{-1}( \boldsymbol{ \boldsymbol{\mathcal{R}}}_\theta+ \boldsymbol{ \boldsymbol{\mathcal{R}}}_\theta  ^T)+ {\rm i}\, (\id- \boldsymbol{\mathcal{T}}_{\!\! \theta} ^{-1} { \boldsymbol{ \boldsymbol{\mathcal{Q}}}}_\theta).
\end{align}
From \eqref{28} it is easy to see
\begin{align}\label{31}
 \boldsymbol{ \boldsymbol{\mathcal{R}}}_\theta  ^T= {\rm i}\, \boldsymbol{\mathcal{M}}- {\rm i}\,\boldsymbol{\mathcal{T}}_{\!\! \theta}  \boldsymbol{\mathcal{E}}_{\!\theta}   
\end{align}
which after substituting \eqref{31} in $\eqref{30}_2$, leads us to
\begin{align}\label{z10}
\boldsymbol{\mathcal{E}}_{\!\theta}   '&= \boldsymbol{\mathcal{T}}_{\!\! \theta} ^{-1}\{ \boldsymbol{ \boldsymbol{\mathcal{R}}}_\theta  + {\rm i}\, (\boldsymbol{\mathcal{M}}- \boldsymbol{\mathcal{T}}_{\!\! \theta} ) \boldsymbol{\mathcal{E}}_{\!\theta}   \} \boldsymbol{\mathcal{E}}_{\!\theta}   + {\rm i}\, (\id- \boldsymbol{\mathcal{T}}_{\!\! \theta} ^{-1} { \boldsymbol{ \boldsymbol{\mathcal{Q}}}}_\theta)= \boldsymbol{\mathcal{T}}_{\!\! \theta} ^{-1}\{ {\rm i}\, (\boldsymbol{\mathcal{M}}- {\rm i}\,\boldsymbol{ \boldsymbol{\mathcal{R}}}_\theta  )- {\rm i}\,\boldsymbol{\mathcal{T}}_{\!\! \theta} ) \boldsymbol{\mathcal{E}}_{\!\theta}   \} \boldsymbol{\mathcal{E}}_{\!\theta}   + {\rm i}\, (\id- \boldsymbol{\mathcal{T}}_{\!\! \theta} ^{-1} { \boldsymbol{ \boldsymbol{\mathcal{Q}}}}_\theta)\\&= \boldsymbol{\mathcal{T}}_{\!\! \theta} ^{-1}\{ {\rm i}\,\boldsymbol{ \boldsymbol{\mathcal{Q}}}- {\rm i}\,\boldsymbol{\mathcal{T}}_{\!\! \theta}  \boldsymbol{\mathcal{E}}_{\!\theta}   \} \boldsymbol{\mathcal{E}}_{\!\theta}   + {\rm i}\, (\id- \boldsymbol{\mathcal{T}}_{\!\! \theta} ^{-1} { \boldsymbol{ \boldsymbol{\mathcal{Q}}}}_\theta)=\cancel{ {\rm i}\,\boldsymbol{\mathcal{T}}_{\!\! \theta} ^{-1} \boldsymbol{ \boldsymbol{\mathcal{Q}}}}- {\rm i}\,\boldsymbol{\mathcal{E}}_{\!\theta}   ^2+\id  {\rm i}\,-\cancel{ {\rm i}\,\boldsymbol{\mathcal{T}}_{\!\! \theta} ^{-1} { \boldsymbol{ \boldsymbol{\mathcal{Q}}}}_\theta}=\id  {\rm i}\, - {\rm i}\,\boldsymbol{\mathcal{E}}_{\!\theta}   ^2.\notag
\end{align} 
After integrating this matrix differential equation subject to the condition $\boldsymbol{\mathcal{E}}_{0}=\boldsymbol{\mathcal{E}}$, we obtain
\begin{proposition} {\rm (Exactly as from Fu and Mielke's paper)} We have
	\begin{align}\label{32}
	\boldsymbol{\mathcal{E}}_{\!\theta}   =(\cos\theta \id+ {\rm i}\, \sin\theta \boldsymbol{\mathcal{E}})^{-1}(\cos\theta \boldsymbol{\mathcal{E}}+ {\rm i}\, \sin\theta \id).
	\end{align}
\end{proposition}
\begin{proof}
Since, $\mathbf{J}(\theta)\,\mathbf{J}^{-1}(\theta)=\id$	after differentiating, it can be observed that for any matrix $\mathbf{J}(\theta)$, we have
	\begin{align}
	(\mathbf{J}^{-1} (\theta))'=-\mathbf{J}^{-1} (\theta)\mathbf{J}'(\theta)\mathbf{J}^{-1} (\theta),
	\end{align}	
	it then follows that 
	\begin{align}
	(\mathbf{J}^{-1} (\theta) \mathbf{J}'(\theta))'=-(\mathbf{J}^{-1} (\theta) \mathbf{J}'(\theta))^2+\mathbf{J}^{-1} (\theta) \mathbf{J}''\theta),
	\end{align}
	or equivalently
	\begin{align}
	(-{\rm i}\,\mathbf{J}^{-1} (\theta) \mathbf{J}'(\theta))'=- {\rm i}\, (-{\rm i}\,\mathbf{J}^{-1} (\theta) \mathbf{J}'(\theta))^2-{\rm i}\,\mathbf{J}^{-1} (\theta) \mathbf{J}''\theta).
	\end{align}
	This suggests the transformation
	\begin{align}\label{z11}
	\boldsymbol{\mathcal{E}}_{\theta}=-{\rm i}\,\mathbf{J}^{-1} (\theta) \mathbf{J}'(\theta)
	\end{align}
	since on substituting \eqref{z11} in \eqref{z10}, we arrive at
	\begin{align}\label{z12}
	\mathbf{J}''\theta)+\mathbf{J}(\theta)=0,
	\end{align}
	the general solution of \eqref{z12} is
	\begin{align}
	\mathbf{J}(\theta)=\cos\theta \,\mathbf{c}_{l}+ {\rm i}\, \sin\theta\, \mathbf{c}_{t},\notag
	\end{align}
	where, $\mathbf{c}_{l}$ and $\mathbf{c}_{t}$ are constant matrices.
	From the condition $\boldsymbol{\mathcal{E}}_0=\boldsymbol{\mathcal{E}}$ we obtain $\boldsymbol{\mathcal{E}}=- {\rm i}\,\mathbf{c}_{l} \mathbf{c}_{t}$. Thus, $\mathbf{J}(\theta)=\mathbf{c}_{l}(\cos\theta \id+ {\rm i}\,\sin\theta\, \boldsymbol{\mathcal{E}})$ and
	\begin{align}\label{z13}
	\boldsymbol{\mathcal{E}}_{\!\theta}   =-{\rm i}\,\mathbf{J}^{-1} (\theta) \mathbf{J}(\theta)=(\cos\theta \,\id+ {\rm i}\,\sin\theta\, \boldsymbol{\mathcal{E}})^{-1}(\cos\theta \boldsymbol{\mathcal{E}}+ {\rm i}\,\sin\theta \,\id)
	\end{align}
	and the proof is complete.
\end{proof}
If $\lambda  $ is an eigenvalue of $\boldsymbol{\mathcal{E}}$. then by \eqref{13} the corresponding eigenvalues of $ \boldsymbol{\mathcal{E}}_{\!\theta}   $ is
\begin{align}
\lambda \, ( \theta)=\Phi(\lambda_{\rm e} ,\theta),\quad\text{where} \quad\Phi(\lambda  ,\theta)=\frac{\lambda  \,\cos\theta+ {\rm i}\, \sin\theta}{\cos\theta+ {\rm i}\, \lambda  \, \sin\theta}= {\rm i}\, \frac{d}{d\theta}\text{ln}(\cos\theta+ {\rm i}\, \lambda  \,\sin\theta).
\end{align}
The real part of $\lambda  \, ( \theta)$ is indeed positive. It then follows
\begin{align}
\int_{0}^{\pi} \lambda  \, ( \theta)\, d\theta&=- {\rm i}\, \int_{0}^{\pi} \frac{d}{d\theta}\text{ln}(\cos\theta+ {\rm i}\, \lambda  \, \sin\theta)\,d\theta
=- {\rm i}\, \,\text{ln}(\cos\theta+ {\rm i}\, \lambda  \, \sin\theta)\Bigr|_{0}^{\pi}=\pi\,.
\end{align}
So this leads to 
\begin{proposition}{\rm (Exactly as from Fu and Mielke's paper)}
	We	have
	\begin{align}
	\int_{0}^{\pi} 
	\boldsymbol{\mathcal{E}}_{\!\theta}   \, d\theta=\pi \id
	\end{align}
\end{proposition}
\begin{proof} We use the spectral calculus of matrices. By choosing any closed curve $\Gamma$ in the complex plane surrounding all the  eigenvectors of $\boldsymbol{\mathcal{E}}$ with positive real parts and lying in the open half plane $\text{Re}\, \lambda  > 0$. Then by Proposition $2.1$ we have
	\begin{align}
	\int_{0}^{\pi} 	 \boldsymbol{\mathcal{E}}_{\!\theta}   \, d\theta&=\int_{0}^{\pi} \frac{1}{2\pi  {\rm i}\, }\oint_\Gamma \Phi(\lambda ,\theta) (\lambda \id -\boldsymbol{\mathcal{E}})^{-1}\,d \lambda \,d\theta=\frac{1}{2\pi  {\rm i}\, }\oint_\Gamma\int_{0}^{\pi}  \Phi(\lambda_{\rm e} \, ,\theta)\,d\theta\, (\lambda \, \id -\boldsymbol{\mathcal{E}})^{-1}\,d \lambda_{\rm e}  ,\notag\\
	&=\frac{1}{2\pi  {\rm i}\, }\oint_\Gamma \, (\lambda \, \id -\boldsymbol{\mathcal{E}})^{-1}\pi\,d \lambda_{\rm e}  .\notag
	\end{align}
	The conclusion of the proposition follows then since $\int_{0}^{\pi} \lambda_{\rm e} \, ( \theta)\, d\theta=\pi$.
\end{proof}
On integrating \eqref{28}, we get
\begin{align}\label{final}
\int_{0}^{\pi}	 \boldsymbol{\mathcal{E}}_{\!\theta}   \, d\theta=\int_{0}^{\pi}\,  \boldsymbol{\mathcal{T}}_{\!\! \theta} ^{-1}(\boldsymbol{\mathcal{M}}+ {\rm i}\,\boldsymbol{ \boldsymbol{\mathcal{R}}}_\theta  ^T)\, d\theta\quad &\Leftrightarrow\quad 
\int_{0}^{\pi}	 \boldsymbol{\mathcal{T}}_{\!\! \theta} ^{-1}\boldsymbol{\mathcal{M}}\, d\theta=\int_{0}^{\pi}\, \boldsymbol{\mathcal{E}}_{\!\theta}   \, d\theta-  {\rm i}\, \int_{0}^{\pi} \boldsymbol{\mathcal{T}}_{\!\! \theta} ^{-1} \boldsymbol{ \boldsymbol{\mathcal{R}}}_\theta  ^T\, d\theta,\notag\\ \quad &\Leftrightarrow\quad 
\Big(\int_{0}^{\pi}	 \boldsymbol{\mathcal{T}}_{\!\! \theta} ^{-1}\, d\theta\Big) \boldsymbol{\mathcal{M}}=\pi\id-  {\rm i}\, \int_{0}^{\pi} \boldsymbol{\mathcal{T}}_{\!\! \theta} ^{-1} \boldsymbol{ \boldsymbol{\mathcal{R}}}_\theta  ^T\, d\theta.
\end{align} 
Finally we obtain
\begin{theorem}{\rm (Exactly as from Fu and Mielke's paper)}
	The unique solution of the \textit{algebraic Riccati equation} \eqref{12} that satisfies \\ $\text{\rm Re\,spec}\,( \boldsymbol{\mathcal{T}}^{-1}(\boldsymbol{\mathcal{M}}+{\rm i}\,\boldsymbol{\mathcal{R}}^T))>0$ is given explicitly by
	\begin{align}
	\boldsymbol{\mathcal{M}}=\Big(\int_{0}^{\pi}	 \boldsymbol{\mathcal{T}}_{\!\! \theta} ^{-1}\, d\theta\Big)^{-1}\Big(\pi\id-  {\rm i}\, \int_{0}^{\pi} \boldsymbol{\mathcal{T}}_{\!\! \theta} ^{-1} \boldsymbol{ \boldsymbol{\mathcal{R}}}_\theta  ^T\, d\theta\Big).
	\end{align}
\end{theorem}
\begin{proof}
	We multiply the relation \eqref{final} by $\Big(\int_{0}^{\pi}	 \boldsymbol{\mathcal{T}}_{\!\! \theta} ^{-1}\, d\theta\Big)^{-1}$.
\end{proof}

\end{footnotesize}
\end{document}